\documentclass{amsart}

\usepackage{amsmath,amsthm,amssymb}
\usepackage{mathtools}
\usepackage{tikz,tikz-3dplot,tikz-cd,graphicx}
\usepackage{standalone}
\usepackage{youngtab}
\usepackage{adjustbox}
\usepackage{tabularx}
\usetikzlibrary{shapes.geometric}
\graphicspath{
  {./}
  {figures/}
  {figures/rectangles/}
  {figures/deg5example/}
  {figures/annuli/}
  {figures/polynomials/}
  {figures/brrect/}
}
\usepackage[hidelinks]{hyperref}

\theoremstyle{plain}
\newtheorem{thm}{Theorem}[section]
\newtheorem{lem}[thm]{Lemma}
\newtheorem{cor}[thm]{Corollary}
\newtheorem{prop}[thm]{Proposition}
\newtheorem{conj}{Conjecture}

\newtheorem{mainthm}{Theorem}

\theoremstyle{definition}
\newtheorem{defn}[thm]{Definition}
\newtheorem{rem}[thm]{Remark}

\newtheorem{exmp}[thm]{Example}
\newtheorem{example}[thm]{Example}

\newcommand{\mathsc}[1]{{\normalfont\textsc{#1}}}

\newcommand{\N}{\mathbb{N}}
\newcommand{\Z}{\mathbb{Z}}

\newcommand{\R}{\mathbb{R}}

\newcommand{\Bb}{\mathbf{B}}
\newcommand{\Ib}{\mathbf{I}}\newcommand{\Jb}{\mathbf{J}}

\newcommand{\T}{\mathbb{T}}\newcommand{\Tb}{\mathbf{T}}

\newcommand{\pb}{\mathbf{p}}
\newcommand{\qb}{\mathbf{q}}
\newcommand{\bb}{\mathbf{b}}
\newcommand{\db}{\mathbf{d}}

\newcommand{\Hb}{\mathbf{H}}
\newcommand{\Rb}{\mathbf{R}}

\newcommand{\Lb}{\mathbf{L}}
\newcommand{\Xb}{\mathbf{X}}
\newcommand{\Yb}{\mathbf{Y}}
\newcommand{\Zb}{\mathbf{Z}}
\newcommand{\Wb}{\mathbf{W}}
\newcommand{\Pb}{\mathbf{P}}
\newcommand{\Qb}{\mathbf{Q}}
\newcommand{\Ub}{\mathbf{U}}
\newcommand{\Vb}{\mathbf{V}}

\newcommand{\rev}{\textrm{rev}}

\newcommand{\aff}{\mathsc{Aff}}

\newcommand{\F}{\mathbb{F}}
\newcommand{\C}{\mathbb{C}}
\newcommand{\Cp}{\C P}
\newcommand{\Ch}{\widehat \C}

\newcommand{\wh}{\widehat}
\newcommand{\D}{\mathbb{D}}
\newcommand{\inter}{\mathrm{int}}

\newcommand{\into}{\hookrightarrow}
\newcommand{\onto}{\twoheadrightarrow}

\newcommand{\cat}{\text{\mathsc{CAT}}}

\newcommand{\bary}{\mathsc{Bary}}
\newcommand{\rts}{\mathbf{rts}}
\newcommand{\cpt}{\mathbf{cpt}}
\newcommand{\cvl}{\mathbf{cvl}}

\newcommand{\width}{\mathbf{wt}}

\newcommand{\zer}{\mathbf{0}}
\newcommand{\cb}{\mathbf{c}}
\newcommand{\one}{\mathbf{1}}

\newcommand{\euler}{\strut \chi}

\newcommand{\linear}{\textrm{lin}}
\newcommand{\topolo}{\textrm{top}}

\newcommand{\set}{\mathsc{Set}}

\newcommand{\mult}{\mathsc{Mult}}
\newcommand{\spart}{\mathsc{SetPart}}
\newcommand{\ipart}{\mathsc{IntPart}}
\newcommand{\shape}{\mathsc{Shape}}
\newcommand{\sym}{\mathsc{Sym}}

\newcommand{\len}{\mathsc{len}}

\newcommand{\simp}{\Delta}

\newcommand{\geocomb}{\mathsc{GeoCom}}
\newcommand{\braid}{\mathsc{Braid}}

\newcommand{\krew}{\mathsc{Krew}}
\newcommand{\perm}{\mathsc{Perm}}

\newcommand{\conf}{\mathsc{Conf}}
\newcommand{\uconf}{\mathsc{UConf}}

\newcommand{\poly}{\ensuremath{\mathsc{Poly}}}
\newcommand{\LL}{\mathsc{LL}}

\newcommand{\ncpart}{\mathsc{NCPart}}
\newcommand{\ncmatch}{\mathsc{NCMat}}

\newcommand{\ncperm}{\mathsc{NCPerm}}

\newcommand{\br}{\mathsc{Br}}

\newcommand{\cmpt}{\mathsc{cmpt}}
\newcommand{\bt}{\begin{tabular}}
\newcommand{\et}{\end{tabular}}
\newcommand{\size}[1]{|#1|}
\newcommand{\comp}{\mathsc{Comp}}

\newcommand{\ba}{\mathbf{a}}
\newcommand{\bm}{\mathbf{m}}
\newcommand{\bx}{\mathbf{x}}
\newcommand{\bz}{\mathbf{z}}
\newcommand{\bw}{\mathbf{w}}

\newcommand{\bj}{\mathbf{J}}

\newcommand{\openint}{\protect
  \begin{tikzpicture}[baseline]
    \draw[thick] (0,.1)--(.25,.1);
    \filldraw[color=black,fill=white] (0,.1) circle (.3mm) (.25,.1) circle (.3mm);
  \end{tikzpicture}
}

\newcommand{\closedint}{\protect
  \begin{tikzpicture}[baseline]
    \draw[thick] (0,.1)--(.25,.1);
    \filldraw[color=black,fill=black!70!white] (0,.1) circle (.3mm) (.25,.1) circle (.3mm);
  \end{tikzpicture}
}

\newcommand{\circleint}{\protect
  \begin{tikzpicture}[baseline]
    \filldraw[thick,color=black,fill=white] (0,.1) circle (1.4mm);
  \end{tikzpicture}
}

\newcommand{\opencircleint}{\protect
  \begin{tikzpicture}[baseline]
    \coordinate (0) at (0,.1);
    \filldraw[thick,color=black,fill=white] (0) circle (1.4mm);
    \filldraw[color=black,fill=white] (-.14,.1) circle (.3mm);
  \end{tikzpicture}
}

\newcommand{\opensquare}{\protect
  \begin{tikzpicture}[baseline]
    \filldraw[thick,color=white!70!black,fill=white!80!black] (0,0) rectangle (.2,.2);
  \end{tikzpicture}
}

\newcommand{\closedsquare}{\protect
  \begin{tikzpicture}[baseline]
    \filldraw[thick,fill=white!80!black] (0,0) rectangle (.2,.2);
  \end{tikzpicture}
}

\newcommand{\opendisk}{\protect
  \begin{tikzpicture}[baseline]
    \filldraw[thick,color=white!70!black,fill=white!80!black] (0,.1) circle (.15);
  \end{tikzpicture}
}

\newcommand{\openannulus}{\protect
  \begin{tikzpicture}[baseline]
    \filldraw[thick,color=white!70!black,fill=white!80!black] (0,.1) circle (.15);
    \filldraw[thick,color=white!70!black,fill=white] (0,.1) circle (.05);
  \end{tikzpicture}
}

\newcommand{\closeddisk}{\protect
  \begin{tikzpicture}[baseline]
    \filldraw[thick,color=black,fill=white!80!black] (0,.1) circle (.15);
  \end{tikzpicture}
}

\newcommand{\closedannulus}{\protect
  \begin{tikzpicture}[baseline]
    \filldraw[thick,color=black,fill=white!80!black] (0,.1) circle (.15);
    \filldraw[thick,color=black,fill=white] (0,.1) circle (.05);
  \end{tikzpicture}
}

\newcounter{joncomments}

\newcounter{michaelcomments}

\begin{document}

\title[Geometric Combinatorics of Polynomials II]{Geometric Combinatorics of Polynomials II: \\
  Polynomials and Cell Structures}
\author{Michael Dougherty} 
\email{doughemj@lafayette.edu}
\address{Department of Mathematics, Lafayette College,
  Easton, PA 18042}
\author{Jon McCammond}
\email{jon.mccammond@math.ucsb.edu}
\address{Department of Mathematics, UC Santa Barbara, 
  Santa Barbara, CA 93106} 
\date{\today}

\begin{abstract}
  This article introduces a finite piecewise Euclidean cell complex
  homeomorphic to the space of monic centered complex polynomials of
  degree $d$ whose critical values lie in a fixed closed rectangular
  region.  We call this the \emph{branched rectangle complex} since
  its points are indexed by marked $d$-sheeted planar branched covers
  of the fixed rectangle.  The vertices of the cell structure are
  indexed by the combinatorial ``basketballs'' studied by Martin,
  Savitt and Singer. 
  Structurally, the branched rectangle complex is a full
  subcomplex of a direct product of two copies of the order complex of
  the noncrossing partition lattice.  Topologically, it is
  homeomorphic to the closed $2n$-dimensional ball where $n=d-1$.
  Metrically, the simplices in each factor are orthoschemes.  It can
  also be viewed as a compactification of the space of all monic
  centered complex polynomials of degree $d$.

  We also introduce a finite piecewise Euclidean cell complex
  homeomorphic to the space of monic centered complex polynomials of
  degree $d$ whose critical values lie in a fixed closed annular
  region.  We call this the \emph{branched annulus complex} since its
  points are indexed by marked $d$-sheeted planar branched covers of
  the fixed annulus.  It can be constructed from the branched
  rectangle complex as a cellular quotient by isometric face
  identifications.  And it can be viewed as a compactification of the
  space of all monic centered complex polynomials of degree $d$ with
  distinct roots.
  
  Finally, the branched annulus complex deformation retracts to the
  \emph{branched circle complex}, which we identify with the dual
  braid complex. The space of polynomials with distinct roots is one
  of the earliest classifying spaces for the $d$-strand braid group,
  and the dual braid complex is a more recent classifying space
  derived from the braid group's dual Garside structure.  Our explicit
  embedding of one classifying space as a spine of the other provides 
  a direct proof that the two classifying spaces are homotopy
  equivalent.
\end{abstract}

\maketitle

\section*{Introduction}

Let $\poly_d^{mc}$ be the space of monic centered complex polynomials
of degree $d$ and let $\poly_d^{mc}(\Ub)$ be the subspace of polynomials
whose critical values lie in $\Ub \subset \C$.  In this article we
describe a finite piecewise Euclidean cell structure on
$\poly_d^{mc}(\Ub)$ in four cases: when $\Ub$ is a closed interval
$\closedint$, a circle $\circleint$, a closed rectangle
$\closedsquare$, or a closed annulus $\closedannulus$.  Spaces
homeomorphic to the first two have already appeared in the literature
under different names.  The last two are being introduced here.

The overall flavor of our results is best illustrated by a concrete
example.  Consider the case where $d=3$ and $\Ub = \closedint$ is the
real interval $[-2,2] \subset \C$.  The space
$\poly_3^{mc}(\closedint)$, of monic centered cubic polynomials with
critical values in the interval $[-2,2]$, includes the polynomial
$p(z) = z^3 + b$ for any $b\in [-2,2]$, since it has a double critical
point at $0$ and a double critical value at $b$, as well as the
(rescaled) Chebyshev polynomial $p(z) = z^3-3z$ with critical points
$\pm 1$ and critical values $\pm 2$.  More generally, it is
straightforward to compute that the exact set of polynomials
satisfying these conditions are those of the form $p(z) = z^3 -3a^2
\omega^k z +b$ with $a \in [0,1]$, $b \in [2a^3-2,2-2a^3]$, $k \in
\{0,1,2\}$, and where $\omega = \frac{1 + i \sqrt{3}}{2}$ is a cube
root of unity.  Note that when $a=0$, the value of $k$ is
irrelevant. Topologically, the space is three triangles with a common
hypotenuse, depicted in the left hand side of Figure~\ref{fig:cpoly}.
The metric is provided by the space of possible critical values.  Two
distinguishable critical values in $[-2,2]$ are represented by a point
in the square $[-2,2]^2$, and removing our ability to distinguish them
corresponds to folding the square along the diagonal line where they
are equal to produce the space $\mult_2(\closedint)$ of $2$-element
multisets in $[-2,2]$.  The right-angled triangle
$\mult_2(\closedint)$, shown on the right hand side of
Figure~\ref{fig:cpoly}, is also known as a $2$-dimensional
orthoscheme.  The generically $3$-to-$1$ map from
$\poly_3^{mc}(\closedint)$ on the left to $\mult_2(\closedint)$ on the
right is a restriction of the Lyashko--Looijenga map that sends a
monic centered complex polynomial to (the monic polynomial determined
by) its multiset of critical values.

\begin{figure}
  \centering    
  \begin{tabular}{cc}
    \includegraphics{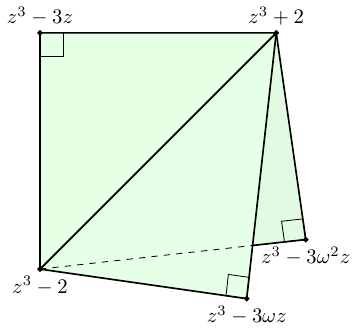} & 
    \includegraphics{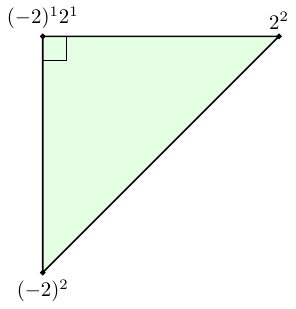}\\
    $\poly_3^{mc}(\closedint)$ & $\mult_2(\closedint)$\\
  \end{tabular}
  \caption{The right-angled triangle shown on the right is the space
    $\mult_2(\protect\closedint)$ of $2$ unlabeled points in the
    closed interval $[-2,2]$. The three right-angled triangles with a
    common hypotenuse shown on the left form the space
    $\poly_3^{mc}(\protect\closedint)$ of monic centered cubic
    polynomials with critical values in $[-2,2]$.  The extremal
    situations have been labeled in both cases.\label{fig:cpoly}}
\end{figure}

More generally, $\mult_d(\closedint)$ can be viewed
as a metric simplex known as a standard $d$-dimensional orthoscheme,
and we can pull back through the Lyashko--Looijenga map to obtain a
simplicial structure for $\poly_d^{mt}(\closedint)$, the space
of monic degree-$d$ polynomials up to precomposition with a
translation\footnote{Considering polynomials up to precomposition with
a translation is equivalent to centering the roots at the origin. The latter
is preferable for concrete examples, but we will use the former for
the statements of theorems. See Remark~\ref{rem:center-trans}.}  with
critical values in $\closedint$. When the metric from the orthoscheme is
pulled back through the Lyashko--Looijenga map, $\poly_d^{mt}(\closedint)$ is assigned 
a \emph{stratified Euclidean metric} (Definition~\ref{def:strata-euclid-metric}) 
distinct from the usual Euclidean metric inherited from $\C^{d-1}$ when the coefficients 
of the monic centered polynomial are used as coordinates.
For higher values of $d$, the combinatorial structure can be 
described using the lattice of noncrossing partitions $\ncpart_d$
and its order complex $|\ncpart_d|_\Delta$.

\begin{mainthm}[Intervals: Theorem~\ref{thm:main-intervals}]
    \label{mainthm:intervals}
    The space $\poly_d^{mt}(\closedint)$ of polynomials with critical values in a 
    closed interval (with the stratified Euclidean metric) is isometric to the 
    order complex $|\ncpart_d|_\Delta$ (with the orthoscheme metric). 
\end{mainthm}

The noncrossing partition lattice $\ncpart_d$, defined by Kreweras in
1972 \cite{kreweras72}, has  $C_d = \frac{1}{d+1} \binom{2d}{d}$ elements
and $d^{d-2}$ maximal chains, where the numbers $C_d$ are the ubiquitous 
Catalan numbers \cite{St15-catalan}.   So by Theorem~\ref{mainthm:intervals}, 
the simplicial complex $\poly_d^{mt}(\closedint) \cong \size{\ncpart_d}_\simp$ 
has $C_d$ vertices and $d^{d-2}$ top-dimensional simplices. 
Alternatively, every polynomial in $\poly_d^{mt}(\closedint)$ can be uniquely 
labeled by a marked planar metric graph called a \emph{banyan} or \emph{branched line} 
(Example~\ref{ex:banyans}).  The complex $\br_d^m(\closedint)$ of all marked 
planar $d$-branched lines (Definition~\ref{def:br-line-cplx}) is the same as the order complex
$\size{\ncpart_d}_\simp$, so one way to restate Theorem~\ref{mainthm:circles} 
is that $\poly_d^{mt}(\closedint) \cong \br_d^m(\closedint)$.
See \cite{baumeister19} and \cite{mccammond06} for surveys of the various ways in which 
noncrossing partitions arise.  

Noncrossing partitions started appearing in geometric
group theory around the year 2000, starting with the work of Birman--Ko--Lee
\cite{birman-ko-lee98} on a new ``dual presentation'' for the
symmetric group $\sym_d$ and the braid group $\braid_d$. Tom Brady \cite{brady01} and David Bessis \cite{bessis-03} soon turned this
presentation into a new classifying space for the braid group.  More specifically, 
the dual presentation utilizes Biane's correspondence between $\ncpart_d$ and $[1,\delta]$, the set of
permutations which appear in minimal length factorizations of the $d$-cycle $\delta =
(1\ 2\ \cdots\ d)$ into transpositions \cite{biane97}. By an
appropriate identification of cells in the order complex of
$\ncpart_d$, Brady defined a simplicial complex which is a classifying
space for $\braid_d$ \cite{brady01}, and this space was later endowed
with the ``orthoscheme metric'' in work of Brady and the second author
\cite{BrMc-10}.  With the orthoscheme metric, this classifying space
is called the \emph{dual braid complex} $K_d$, and it is an example of an
\emph{interval complex} when the same construction is applied in more general settings.  
It was conjectured in \cite{BrMc-10} that the dual braid complex $K_d$ is locally $\cat(0)$, which would
imply that $\braid_d$ is a $\cat(0)$ group. This has been proven for
$d \leq 7$ \cite{BrMc-10,haettel-kielak-schwer-16,jeong23} but remains
open in general. 
Our second main theorem provides a polynomial version of the dual braid
complex $K_d$.

\begin{mainthm}[Circles: Theorem~\ref{thm:main-circles}]
    \label{mainthm:circles}
    The space $\poly_d^{mt}(\circleint)$ of polynomials with critical
    values in a circle is homeomorphic to a quotient of the complex $\poly_d^{mt}(\closedint)$ 
    by face identifications. As a metric $\Delta$-complex, $\poly_d^{mt}(\circleint)$ 
    is the dual braid complex $K_d$ with the orthoscheme metric.
\end{mainthm}

The quotient map transforming $\poly_d^{mt}(\closedint)$ into
$\poly_d^{mt}(\circleint)$ identifies all of the vertices and many of
the lower-dimensional faces. The cell structure on $\poly_d^{mt}(\circleint)$ 
remains $(d-1)$-dimensional with $d^{d-2}$ top-dimensional simplices, but now 
with only one vertex.  Although it is no longer a simplicial complex, 
it is a $\Delta$-complex in the sense of Hatcher \cite{hatcher02-algtop}.
As in the interval case, there is an alternative way to label the points in 
$\poly_d^{mt}(\circleint)$.  Every polynomial in $\poly_d^{mt}(\circleint)$ can 
be uniquely labeled by a marked planar metric graph called a \emph{cactus} or 
\emph{branched circle} (Example~\ref{ex:cacti}).  The complex $\br_d^m(\circleint)$ 
of all marked planar $d$-branched circles (Definition~\ref{def:br-circle-cplx}) 
 is the same as the dual braid complex $K_d$, so one way to restate 
 Theorem~\ref{mainthm:circles} is $\poly_d^{mt}(\circleint) \cong \br_d^m(\circleint)$.

Banyans and cacti were initially described in the first article of this series 
\cite{DoMc22}, although the approach we take here is both more general and more 
direct. Algebraically, banyans and cacti can be viewed as encoding \emph{linear} 
and \emph{circular} factorizations of $\delta$, respectively.  This perspective is 
discussed in \cite{dm-cnc}, where we 
introduce a continuous version of noncrossing partitions that is closely connected to the work of 
W.~Thurston and his collaborators on the space of complex polynomials \cite{thurston20}.

Also, it is worth noting that the universal cover of the dual braid complex $K_d$ can 
be expressed as the direct product of $\mathbb{R}$ with a $(d-2)$-dimensional piecewise Euclidean 
simplicial complex in which each top-dimensional cell is a Coxeter simplex of type
$\widetilde{A}_{d-2}$ \cite[Section~8]{BrMc-10}. In view of Theorem~\ref{mainthm:circles},
the circle action of $\mathbb{S}^1$ on $\poly_d^{mt}(\circleint)$ obtained by rotation of 
$\C$ about the origin corresponds to moving in the $\mathbb{R}$ factor in the universal cover.

The cell structure we put on the polynomial space $\poly_d^{mt}(\closedsquare)$ is more complicated.
A closed rectangle $\closedsquare \subset \C$ is a direct product $\closedint \times
\closedint$ of two intervals and, somewhat surprisingly, we show that $\poly_d^{mt}(\closedsquare)$ 
embeds as a subspace of the product $\poly_d^{mt}(\closedint) \times \poly_d^{mt}(\closedint)$. 
This is our third main theorem.

\begin{mainthm}[Rectangles: Theorem~\ref{thm:main-rectangles}]
    \label{mainthm:rectangles}
    The space $\poly_d^{mt}(\closedsquare)$ of polynomials with critical
    values in a closed rectangle (with the stratified Euclidean metric) 
    is isometric to a subcomplex of 
    $|\ncperm_d|_\Delta \times |\ncperm_d|_\Delta$
    (with the orthoscheme metric). 
\end{mainthm}

The cell structure on $\poly_d^{mt}(\closedsquare)$ is \emph{bisimplicial} 
in the sense that every cell is a
product of two simplices. The vertices are enumerated by the
Fuss--Catalan numbers $C^{(4)}_d = \frac{1}{3d+1} \binom{4d}{d}$ and
labeled by the combinatorial ``basketballs'' (Definition~\ref{def:basketballs})
initially defined by Martin, Savitt and Singer in \cite{martin-savitt-singer-07}. 
There are $(d-1)! d^{d-2}$ top-dimensional cells.  The compatibility condition between 
the noncrossing partitions in each factor arise from the fact that 
they need to be combinatorial aspects of a common marked planar $d$-branched cover of the 
rectangle $\closedsquare$.
The complex $\br_d^m(\closedsquare)$ is the same as $\poly_d^{mt}(\closedsquare)$, but with points
labeled by geometric and combinatorial information associated to marked planar $d$-branched rectangles 
(Definition~\ref{def:br-rect-cplx})
so one way to restate Theorem~\ref{mainthm:rectangles} is that $\br_d^{m}(\closedsquare)$ 
embeds in $\br_d^m(\closedint) \times \br_d^m(\closedint)$.
The compatibility condition on top-dimensional cells can be restated algebraically.
The linear factorization of $\delta$ into transpositions coming from the first factor and the 
linear factorization of $\delta$ into transpositions coming from the second factor are related 
by the \emph{Hurwitz action} of a simple braid. The $d^{d-2}$ maximal factorizations and 
$(d-1)!$ simple braids parameterize the top-dimensional cells.  This fact is related to the 
noncrossing hypertrees results in \cite{Mc-ncht}.

Finally, the transition from $\closedsquare$ to $\closedannulus$ and its impact on 
$\poly_d^{mt}(\closedsquare)$ is very similar to that from $\closedint$ to $\circleint$
and its impact on $\poly_d^{mt}(\closedint)$.  In particular, the side identification in 
$\C$ that turns the rectangle $\closedsquare$ into the annulus $\closedannulus$ induces a face 
identification on the corresponding polynomial cell complex.

\begin{mainthm}[Annuli: Theorem~\ref{thm:main-annuli}]
    \label{mainthm:annuli}
    The space $\poly_d^{mt}(\closedannulus)$ of polynomials with
    critical values in a closed annulus is homeomorphic to 
    a quotient of the metric cell complex $\poly_d^{mt}(\closedsquare)$ by face identifications.
\end{mainthm}

The vertices of the bisimplicial cell structure on
$\poly_d^{mt}(\closedannulus)$ are counted by the Catalan numbers
$C_d$, and there are $(d-1)! d^{d-2}$ top-dimensional cells. The complex we call 
$\br_d^m(\closedannulus)$ is the same as a the space $\poly_d^{mt}(\closedannulus)$ but with 
points labeled by geometric and combinatorial information associated to marked planar 
$d$-branched annuli (Definition~\ref{def:std-annulus-reps}).  A detailed exploration 
of the geometric combinatorics of the compact piecewise Euclidean cell complexes 
described by these first four theorems will appear later in this series.

The metric cell structures described in the preceding theorems are 
also connected to classical polynomial spaces via homeomorphisms, 
compactifications and deformation retractions.  Our three final main theorems 
describe these connections.

\begin{mainthm}[Homeomorphisms: Theorem~\ref{thm:main-homeomorphisms}]
    \label{mainthm:homeomorphisms}
    The complex plane $\C$, the punctured plane $\C_\zer$ and the real
    line $\R$ are homeomorphic to the open rectangle $\opensquare$, the
    open annulus $\openannulus$ and the open interval $\openint$
    respectively, and these induce homeomorphisms of polynomial spaces
    $\poly_d^{mt}(\C) \cong \poly_d^{mt}(\opensquare)$,
    $\poly_d^{mt}(\C_\zer) \cong \poly_d^{mt}(\openannulus)$ and
    $\poly_d^{mt}(\R) \cong \poly_d^{mt}(\openint)$.
\end{mainthm}

\begin{mainthm}[Compactifications: Theorem~\ref{thm:main-compactifications}]
  \label{mainthm:compactifications}
  The three spaces $\poly_d^{mt}(\closedsquare)$, $\poly_d^{mt}(\closedannulus)$ and $\poly_d^{mt}(\closedint)$ are 
  compactifications of $\poly_d^{mt}(\opensquare)$,  $\poly_d^{mt}(\openannulus)$ and $\poly_d^{mt}(\openint)$.
\end{mainthm}

Combining Theorem~\ref{mainthm:homeomorphisms} with Theorem~\ref{mainthm:compactifications} shows
that $\poly_d^{mt}(\closedsquare)$, $\poly_d^{mt}(\closedannulus)$ and $\poly_d^{mt}(\closedint)$ 
can be viewed as compactifications of $\poly_d^{mt}(\C)$, $\poly_d^{mt}(\C_\zer)$ and $\poly_d^{mt}(\R)$, 
respectively.  In particular, the complex $\poly_d^{mt}(\closedsquare)$ is a manifold, or more specifically a closed 
ball with corners.  For comparison, the space $\poly_d^{mt}(\closedint)$ is not a manifold, and neither 
is the product of two copies of $\poly_d^{mt}(\closedint)$ which contains $\poly_d^{mt}(\closedsquare)$ as a 
subcomplex (Theorem~\ref{mainthm:rectangles}).

Finally, some of these polynomial spaces are deformation retracts of others.

\begin{mainthm}[Deformation retractions: Theorem~\ref{thm:main-deformations}]
    \label{mainthm:deformations}
    A deformation retraction of $\closedsquare$ onto any embedded arc
    $\closedint \subset \closedsquare$ induces a deformation retraction from 
    $\poly_d^{mt}(\closedsquare)$ to $\poly_d^{mt}(\closedint)$. Similarly, 
    a deformation retraction of $\closedannulus$ onto any core curve $\circleint 
    \subset \closedannulus$ induces a deformation retraction from 
    $\poly_d^{mt}(\closedannulus)$ to $\poly_d^{mt}(\circleint)$. 
\end{mainthm}

This shows that the space of monic degree-$d$ polynomials
with distinct roots up to translation admits a compactification with 
a metric bisimplicial cell structure such that (1) there is a deformation
retraction of this space to the subspace of polynomials with critical
values on the unit circle, and (2) the induced cell structure on the
resulting space is isometric to the dual braid complex $K_d$.

\subsection*{Connections to the literature}
Many of the objects and ideas in our main theorems have appeared
before in various forms, but we believe this is the first article to
combine all of them.  Any list of connections to the rich literature on these
topics will be incomplete, but here are some of the previous appearances 
of the main characters: combinatorial data associated to polynomials,
compactifications, cell structures, and the dual Garside structure for
the braid group.

The preimages of lines and circles that arise in the combinatorial structures 
we derive from polynomials have appeared in many places over the years---we 
list a small fraction of
them here.  In this article, we label the points in $\poly_d^{mt}(\closedint)$ by 
preimages of line segments (\emph{banyans} or \emph{branched lines}).  
Similarly, Gauss's first proof of the Fundamental Theorem of Algebra \cite{gauss1799} 
utilizes the preimages of the real and imaginary axes under a complex polynomial.  
This idea was expanded upon by Martin, Savitt and Singer in \cite{martin-savitt-singer-07}
and by Savitt in \cite{savitt-09} to introduce the notion of
\emph{basketballs}, which can be used to label the vertices of our
cell structure on $\poly_d^{mt}(\closedsquare)$.  More recently, generalizations of these 
line preimages called \emph{signatures} of polynomials have appeared in work of A'Campo
\cite{acampo20} and Combe~\cite{combe19}.

We use preimages of circles (\emph{cacti} or \emph{branched circles}) to label
the points of $\poly_d^{mt}(\circleint)$. These graphs have appeared
in a variety of fields, including complex geometry 
\cite{catanese-paluszny,catanese-wajnryb,cacti-braids-polynomials},
complex dynamics \cite{nekrashevych14,calegari22}, and algebraic
topology \cite{salvatore22}.  Special care must be taken when
comparing results from different articles, as there are many
variations (including markings and metrics) and definitions for cacti
across subdisciplines.

Of particular relevance to our article is the appearance of cacti
(as quotients of an equivalence relation on a circle) in an
unpublished manuscript of W. Thurston that was posthumously completed
by his collaborators \cite{thurston20}. In this article, Thurston
identifies the space $\poly_d^{mc}(\circleint)$ as a 
spine for the space of polynomials with $d$ distinct roots,
albeit with a different metric and without the cell structure.
See \cite{dm-cnc} for a more detailed account of the connections between Thurston's work and ours.

There are also several interesting connections between this article
and the recent work of D. Calegari \cite{calegari22}, which provides
a description for the space of monic complex polynomials with 
$d$ distinct centered roots such that every critical point is in
the attracting basin of infinity, i.e. tends toward $\infty$
under iteration. This space is known to dynamicists as the 
\emph{shift locus of degree $d$}, and Calegari shows that it
has the structure of a complex of spaces in which each piece can be
viewed either in a combinatorial manner via extended laminations 
of a circle (generalizing the work of Thurston in \cite{thurston20})
or in an algebraic manner via spaces of the form $\poly_d^{mc}(\C-\{z_1,\ldots,z_k\})$, i.e. polynomials whose critical values avoid a fixed finite
set of points. In particular, Calegari conjectures that these pieces
are $K(\pi,1)$'s, with a proof in degree $3$ \cite[Theorem~9.16]{calegari22}
which makes essential use of the dual braid complex $K_d$. 
We are interested to see if there are further connections between 
our results and Calegari's work on the shift locus.

Recent work by Wegert \cite{wegert20}, A'Campo \cite{acampo22} and
A'Campo--Papadopoulos \cite{acampo-papadopoulos24} uses the pullbacks
of both lines and circles to study polynomials in a manner which is
similar to ours. In the latter two articles, the authors use this to
introduce open cell structures for both the space $\poly_d^{m}(\C)$ of
monic degree-$d$ polynomials with arbitrary roots and
$\poly_d^{m}(\C_\zer)$, the subspace of polynomials with distinct
roots. The numbers of top-dimensional cells in these complexes match
the numbers of vertices in our cell structures on
$\poly_d^{mt}(\closedsquare)$ and $\poly_d^{mt}(\closedannulus)$
respectively, so it seems likely that the two structures are dual to
one another.

Several other recent articles in algebraic topology have featured ideas similar 
to ours. Following a presentation by the second author on the contents of
this article at an Oberwolfach mini-workshop \cite{oberwolfach24}, Bianchi pointed out 
a connection with his recent work \cite{bianchi1,bianchi2}.
It appears that Bianchi defines an abstract cell structure for $\poly_d^{mt}(\closedsquare)$
which matches ours, including an equivalent set of algebraic labels.  Bianchi's
complex, however, does not include a metric. Other relevant
work includes that of Salter \cite{salter-monodromy1,salter-monodromy2} and Salvatore
\cite{salvatore22}. Salter's articles provide an ``equicritical''
stratification on $\poly_d^{m}(\C_\zer)$ which is related to (but
coarser than) the double stratification described in
Section~\ref{sec:poly-ll-maps}. The combinatorial tools used in
Salter's work are similar to ours, but lead to distinct
structures. Salvatore introduces a cell structure for
$\poly_d^{mt}(\C_\zer)$ which uses similar objects (nested trees of
cacti), but the resulting complex does not appear to be directly
related to ours. We should also note that none of these authors connect their complexes to the 
dual structure on the braid group.

In this article, we compactify $\poly_d^{mt}(\C)$ and
$\poly_d^{mt}(\C_\zer)$ to obtain $\poly_d^{mt}(\closedsquare)$ and
$\poly_d^{mt}(\closedannulus)$, respectively.  The space
$\poly_d^{mt}(\C_\zer)$ has several compactifications from algebraic
geometry, and it would be interesting to know more about how our
compactification relates to them. For example, our work is distinct
from the well-known Fulton--MacPherson compactification of
$\poly_d^{mt}(\C_\zer)$, in which collisions of distinct roots
are recorded using their relative positions and velocities 
\cite{fulton-macpherson94}. In contrast, each collision
in our compactification is described using metric cacti which record
the relative cyclic ordering of the colliding roots.

Another distinguishing feature of our work in comparison to the references
mentioned above is the concrete connection with the dual presentation
for the braid group. In particular, the identification of
$\poly_d^{mt}(\closedint)$ with the order complex 
$\size{\ncpart_d}_\simp = \Delta([1,\delta])$
and of $\poly_d^{mt}(\circleint)$ with the dual braid complex
$K_d = K([1,\delta])$ is new---although it was known to Bessis that
$\poly_d^{m}(\mathbb{C}_\zer)$ deformation retracts onto a copy of
$K_d = K([1,\delta])$ \cite[Proposition~10.6]{bessis-03}. 
Similarly, W. Thurston et al. showed that
$\poly_d^{m}(\mathbb{C}_\zer)$ deformation retracts onto
$\poly_d^{mt}(\circleint)$ \cite{thurston20}, but without the cell
structure, metric, or the connection with the dual presentation. Our explicit
description of the cell structure for $\poly_d^{mt}(\closedannulus)$
and the induced structure on $\poly_d^{mt}(\circleint)$ will hopefully
provide a useful bridge between the geometric combinatorics of
polynomials and the geometry of the braid group.

\subsection*{Generalizations and conjectures}
The metric cell structures for the polynomial spaces
$\poly_d^{mt}(\closedsquare)$ and $\poly_d^{mt}(\closedannulus)$
presented in this article prompt several immediate
generalizations and conjectures.  First, we believe that the
stratified Euclidean metric is non-positively curved in the 
following sense.

\begin{conj}\label{conj:cat0}
  The stratified Euclidean metric on the space $\poly_d^{mt}$ of all monic degree-$d$ 
  complex polynomials up to translation is $\cat(0)$.
\end{conj}

This conjecture about a natural piecewise Euclidean metric on the space of all polynomials
has consequences for the study of braid groups.  It is not too difficult to show that for any 
fixed $d$, the space $\poly_d^{mt} = \poly_d^{mt}(\C)$ is $\cat(0)$ if and only if $\poly_d^{mt}(\closedsquare)$
is $\cat(0)$, which is true if and only if $\poly_d^{mt}(\closedannulus)$ is locally $\cat(0)$.  In other words, 
these three spaces have closely related curvature properties.  Similarly, $\poly_d^{mt}(\closedint)$ is $\cat(0)$
if and only if $\poly_d^{mt}(\circleint)$ is locally $\cat(0)$.  Finally, when the first three are (locally) $\cat(0)$
the last two are (locally) $\cat(0)$.  This means, in particular, that proving Conjecture~\ref{conj:cat0} for 
a particular $d$ would establish the main conjecture in \cite{BrMc-10} for this $d$, and show that the $d$-strand braid 
group $\braid_d$ is a $\cat(0)$ group.  We conjecture that this implication is reversible.

\begin{conj}\label{conj:cat0-equivalence}
  The stratified Euclidean metric on the space $\poly_d^{mt}$ is $\cat(0)$ if and 
  only if the orthoscheme metric on the dual braid complex $K_d$ is locally $\cat(0)$.
\end{conj}

One natural extension of our results comes from generalizing the algebraic
labels on our cell structures. For each Coxeter group $W$, there
is a corresponding Artin group $A_W$ which is the fundamental group
of the \emph{orbit configuration space} $Y_W$---the long-standing
\emph{$K(\pi,1)$ conjecture} for Artin groups claims that $Y_W$ is always a
classifying space for $A_W$.  See Paolini's survey \cite{paolini-survey} 
for background and \cite{charney-davis95,delucchi22,huang24,paolini-salvetti21} 
for progress on this problem.

In this article, we are concerned with the case where $W$ is the
symmetric group $\sym_d$, $A_W$ is the braid group $\braid_d$, and
$Y_W$ is the complement of the complex braid arrangement in $\C_d$.  Moreover,
$Y_W$ is homeomorphic to $\poly_d^{mc}(\C_\zer)$. The
linear and rectangular factorizations of $\delta$ defined in this series of articles 
generalizes to an arbitrary Coxeter group $W$, with $\delta$
replaced by a \emph{Coxeter element} of $W$. By mimicking our
construction of the branched annulus complex with algebraic cell labels, we are able
to define a metric bisimplicial structure on a manifold with corners
$\br_{W,\delta}^{m}(\closedannulus)$. By Theorems~\ref{mainthm:annuli}, \ref{mainthm:homeomorphisms} 
and \ref{mainthm:compactifications}, the space $\br_{W,\delta}^{m}(\closedannulus)$ is a 
compactification of $\poly_d^{mt}(\C_\zer)$ when $W = \sym_d$ and $\delta = (1\ \cdots\ d)$. 
We conjecture that a similar claim holds for general Coxeter groups (noting that 
$\br_{W,\delta}^{m}$ will not be compact when $W$ is infinite).

\begin{conj}\label{conj:dual-artin-groups}
  Let $W$ be a Coxeter group and let $\delta \in W$ be a Coxeter
  element. Then the bisimplicial complex
  $\br_{W,\delta}^{m}(\closedannulus)$ is homotopy equivalent to its
  interior, which is homeomorphic to the orbit configuration space $Y_W$.
\end{conj}

The fundamental group of $\br_{W,\delta}^{m}(\closedannulus)$ is the
``dual Artin group'' $A_{W,\delta}^*$, which has played an important
role in several recent advances in the study of Artin groups
\cite{delucchi22,mccammond-sulway17,paolini-salvetti21} and is
conjectured to be isomorphic to the standard Artin group $A_W$ in all
cases.  If Conjecture~\ref{conj:dual-artin-groups} can be proved with
a similar compactification to that of
Theorem~\ref{mainthm:compactifications}, this would imply the
isomorphism $A_{W,\delta}^* \cong A_W$.

\subsection*{Structure of the article}
The article is divided into three parts.  
Part~\ref{part:single-poly} describes the geometric and combinatorial data 
that can be extracted from a single monic complex polynomial.
Part~\ref{part:poly-spaces} shifts attention to spaces of polynomials and 
proves Theorems~\ref{mainthm:homeomorphisms}, \ref{mainthm:compactifications} 
and~\ref{mainthm:deformations}. And in Part~\ref{part:geo-comb} we combine 
the results from Parts~\ref{part:single-poly} and~\ref{part:poly-spaces} to 
establish Theorems~\ref{mainthm:intervals},  \ref{mainthm:circles}, \ref{mainthm:rectangles} 
and~\ref{mainthm:annuli} about specific cell complexes.
Although some of the content included here already exists in the literature, 
it is almost always written to handle much more general situations.  These discussions are 
intended to make our results accessible to a broader audience. 

Finally, it is important to note that the geometric combinatorics described in Part~\ref{part:single-poly} 
is similar to but distinct from those described in the first article in the series.
In \cite{DoMc22} the polynomial $p$ had to have \emph{distinct roots} and the 
constructions used the \emph{polar} coordinates of its (necessarily nonzero) 
critical values.  Here the polynomial $p$ is \emph{arbitrary} and the constructions use
the \emph{rectangular} coordinates of its (possibly zero) critical
values.  A connection between the two versions is discussed in
Section~\ref{sec:annulus-thmD}.  

\subsection*{Acknowledgments}
We are indebted first and foremost to Daan Krammer,
whose 2017 remarks are the origin of this entire project. 
We also thank Andrea Bianchi, Danny Calegari, Theo Douvropoulos 
and Nick Salter for their
bibliographic assistance and helpful conversations. We would also like 
to thank the organizers of the conferences and seminars over the last several years for their 
invitations as we have discussed the results of the first article in the series \cite{DoMc22} 
and previewed the results contained here, albeit in shifting language and on a shifting foundation.  
The transition from polar coordinates to the conceptually simpler and more symmetric rectangular coordinates that we alluded to above meant rewriting the 
foundational aspects of our constructions and significantly delayed the 
completion of the full article.  We are grateful for the comments we received.

Preliminary versions of these results were presented in seminars and colloquia: 
Otto-von-Guericke-Universit\"at Magdeburg (2020 MD), Temple University (2022 MD),
Haverford College (2022 MD), Northeastern University (2022 JM),
University at Albany (2023 MD), Iowa State University (2023 JM), and 
Isfahan University (2024 JM), and at conferences:
AMS Special Session on The Geometry of Complex Polynomials and Rational Functions (2020 MD), 
AMS Special Session on Groups, Geometry and Topology (2021 MD),
Braids and Beyond: the Dehornoy memorial conference at the University of Caen (2021 JM), 
Perspectives on Artin Groups, ICMS Edinburgh, Scotland (2021 JM),
Artin Groups, $\cat(0)$ geometry and related topics at the Ohio State University (2021 JM),
Garside theory and applications, Berlin, Germany (2021 JM), 
Braids in representation theory and algebraic combinatorics at ICERM (2022 JM), 
Arrangements in Ticino, Switzerland (2022 JM), 
Artin groups at the American Institute of Mathematics (2023 MD), 
a mini-workshop on Artin groups and triangulated categories at Oberwohlfach (2024 JM), 
Hot Topics: Artin groups and arrangements at SLMath (2024 JM).

\newpage
\part{A Single Polynomial}\label{part:single-poly}

The goal of Part~\ref{part:single-poly} is to introduce $4$ cell
complexes constructed from a monic complex polynomial $p \colon \C \to
\C$ with critical values in a fixed closed rectangle $\closedsquare$,
and then to define a map based on the information they contain.  The
four complexes $\Pb_p$, $\Pb'_p$, $\Qb_p$, and $\Qb'_p$ focus on the
regular points, critical points, regular values and critical values of
$p$, respectively. The letter $\Pb$ stands for ``polygon'' and
indicates a subdivided polygon in the domain. The letter $\Qb$ stands
for ``quadrilateral'' and indicates a subdivided rectangle in the
range.  The polynomial $p$ sends the ``point'' complexes with a $\Pb$
to the ``value'' complexes with a $\Qb$, and the ``critical''
complexes with primes are cellular duals of the ``regular'' complexes
without primes.  Since a cell complex and its dual contain essentially
the same information, we are really just introducing two complexes in
two different forms.  In Part~\ref{part:geo-comb} we use the
combinatorial structure of the regular point complex $\Pb_p$ and the
metric structure of critical value complex $\Qb'_p$ to construct
compact piecewise Euclidean cell complexes.

Part~\ref{part:single-poly} is structured as follows.
Section~\ref{sec:part-mult} establishes conventions for combinatorial
objects and maps. Section~\ref{sec:poly-brcover} does the same for
polynomials and branched covers. Section~\ref{sec:complex-map}
constructs the four complexes listed above.  Section~\ref{sec:nccomb}
reviews planar noncrossing combinatorics. And
Section~\ref{sec:geo-comb} defines a single map that collates the all of the geometric
and combinatorial information extracted from monic polynomials.

\section{Partitions and Multisets}\label{sec:part-mult}

This section records facts about basic combinatorial objects and the
maps between them defined for any space $\Xb$.  We begin with an example
illustrating set partitions, multisets, and integer partitions,
followed by the precise definitions.

\begin{exmp}[Partitions and multisets]\label{ex:part-mult}
  Let $\Xb$ be a space with five distinct points $a,b,c,d,e \in \Xb$ and
  consider the $7$-tuple $\bx=(a,b,a,c,d,c,e) \in \Xb^7$.  The set
  partition recording which coordinates are equal is $[\lambda] =
  \{\{1,3\},\{2\},\{4,6\},\{5\},\{7\}\}$, or $13|2|46|5|7$.  If we
  forget which entry came from which coordinate, the result is the
  $7$-element multiset $M = \{a,a,b,c,c,d,e\} =
  \{a^2,b^1,c^2,d^1,e^1\}$, or $a^2b^1c^2d^1e^1$. The multiplicities
  of $M$, $\{2,1,2,1,1\}$, form a partition $\lambda = 2^21^3$ of the
  number $7$, which can also be viewed as the sizes of the blocks of
  the set partition $[\lambda]$.  See Figure~\ref{fig:part-mult}.  In
  symbols, $\size{M} = 7$ and $\lambda \vdash 7$. The Young diagram of
    $\lambda$ is shown in Figure~\ref{fig:young}.
\end{exmp}

\begin{figure}
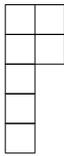

  \yng(2,2,1,1,1)
  \caption{The Young diagram of the integer partition $\lambda =
    2^21^3$.\label{fig:young}}
\end{figure}

In this article, the structure of the space $\Xb$ is always extremely simple, and we
indicate its shape using a visual shorthand.

\begin{defn}[Visual shorthand]\label{def:visual}
  The symbol $\closeddisk$ is a visual shorthand for the closed unit
  disk $\D$ or, more generally, for any closed topological disk
  embedded as a subspace of $\C$, including $\closedsquare$, which is
  shorthand for a closed rectangle.  Although the intended meaning
  should be clear from context, it is also rarely necessary to sharply
  distinguish between the unit disk and a topological disk.  See
  Corollary~\ref{cor:homeo}. Similarly, $\circleint$ is the unit
  circle $\T$ (or any embedded circle), $\closedint$ is the closed
  unit interval $\Ib=[-1,1]$ (or any embedded arc), and $\closedannulus$
  is the closed annulus formed by removing the open disk of radius
  $1/2$ from the closed unit disk $\D$ (or any embedded topological
  closed annulus).  We write $\opensquare$, $\opendisk$, and
  $\openannulus$ for the points in the interior of $\closedsquare$,
  $\closeddisk$ and $\closedannulus$ and $\openint$ for the points in
  $\closedint$ other than its two endpoints.  Finally, let
  $\opencircleint$ denote the unit circle with the point $-1$ removed.  
\end{defn}

\begin{rem}[Spaces]\label{rem:spaces}
  The reader may assume that $\Xb$ is an elementary subset of $\C$,
  such as an interval $\closedint$, a rectangle $\closedsquare$, a
  circle $\circleint$, an annulus $\closedannulus$, or $\C$ itself.
  In particular, $\Xb$ is a path-connected locally Euclidean
  Riemannian manifold of dimension $m >0$, possibly with boundary and
  possibly with corners.
\end{rem}
\begin{defn}[Set partitions]\label{def:set-part}
  Let $[n] = \{1,2,\ldots,n\}$.  A \emph{set partition of $[n]$} is a
  partition of $[n]$ into pairwise disjoint nonempty subsets whose
  union is $[n]$.  There is a standard shorthand notation for a set
  partition when $n$ is small.  Commas and brackets are removed, the
  elements in each block are listed in increasing order, and the
  blocks, separated by vertical bars, are listed in increasing order
  of their minimal elements.  See Example~\ref{ex:part-mult}.  We
  write $\spart_n$ to denote the collection of all set partitions
  of $[n]$ partially ordered by refinement.  The poset $\spart_n$
  is also known as the (set) \emph{partition lattice}.  Its unique
  minimum element is the \emph{discrete set partition} $1|2|\cdots|n$
  with $n$ blocks of size $1$, and its unique maximum element is the
  \emph{indiscrete set partition} $12\cdots n$ with $1$ block of size
  $n$.
\end{defn}

We often write $[\lambda] \vdash [n]$ as notation for a set partition
to highlight the parallel with integer partitions, but it needs to be
used with caution.  See Remark~\ref{rem:int-set}.

\begin{defn}[Multisets]\label{def:multisets}
  A (finite) \emph{multiset} $M = (S,m)$ is a finite set $S$ together
  with a \emph{multiplicity function} $m\colon S \to \N =
  \{1,2,\ldots\}$.  The number $m(x)$ is the \emph{multiplicity of
  $x\in S$}.  A (finite) multiset can be concisely described using the
  notation $M = \{x_1^{m_1},x_2^{m_2},\ldots,x_k^{m_k}\}$ where the
  underlying set is $S = \{x_1,\ldots,x_k\}$, with $x_i = x_j$ if and
  only if $i=j$, and where the exponent $m_i = m(x_i)$ denotes the
  multiplicity of the element $x_i$. In the shorthand notation, the
  commas and braces are removed.  All multisets in this article are
  finite and we drop the adjective.  The \emph{size} of a multiset is
  the sum of its multiplicities: $\size{M} = \size{M}_\mult = n =
  \sum_i m_i$.  This is the unmarked notion of size for multisets. The
  number of distinct elements in $M$ is the \emph{size of $M$ as a
  set}: $\size{M}_\set = \size{S}=k$ where $S = \set(M)$. When
  $\size{M}=n$, $M$ is an \emph{$n$-element multiset}, and when its
  elements are in $\Xb$, it is a \emph{multiset in $\Xb$}.  The
  collection $\mult_n(\Xb)$ of all $n$-element multisets in $\Xb$ can
  also be viewed as $\sym_n \backslash \Xb^n$, the orbits in $\Xb^n$
  under the coordinate permuting symmetric group action, since being
  able to permute the coordinates makes their order irrelevant.  Given
  our assumption on spaces (Remark~\ref{rem:spaces}) $\Xb^n$ a locally
  Euclidean manifold of dimension~$mn$ where $m = \dim(\Xb)$, possibly with boundary and/or
  corners.  The topology on $\mult_n(\Xb)$ comes from the quotient map
  $\Xb^n \to \mult_n(\Xb) = \Xb^n / \sym_n$.
  In this article, $\Xb$ is always a subset of $\C$, so all multisets have
  elements in $\C$, and $\mult_n$ is used as shorthand for
  $\mult_n(\C)$. In general, we omit arguments in this manner only
  when the argument is all of $\C$.
\end{defn}

\begin{defn}[Integer partitions]\label{def:int-part}
  A \emph{partition of a positive integer $n$} can be viewed as a
  multiset of positive integers whose sum is $n$, but care is required
  since partitions and multisets have distinct standard terminology.
  Let $\lambda$ be a multiset of $k$ positive integers whose sum is
  $n$.  The elements of $\lambda$ are its \emph{parts}, the number $k$
  is its \emph{length} and the number $n$ is its \emph{size}.  We
  write $\len(\lambda)=k$ and $\lambda \vdash n$.  Note that the
  length of $\lambda$ (as a partition) is its size as a multiset, and
  its size (as a partition) is its sum.  The parts of a partition are
  typically listed in weakly decreasing order, summarized using
  exponents, and visualized using Young diagrams as in
  Example~\ref{ex:part-mult}.  Concretely, a typical partition is of
  the form $\lambda = \lambda_1^{a_1} \lambda_2^{a_2} \cdots
  \lambda_\ell^{a_\ell}$, where $\lambda_1 > \lambda_2 > \cdots >
  \lambda_\ell$ are its $\ell$ distinct parts, the number $a_i$
  indicates the number of times the part $\lambda_i$ occurs, its
  length $k = a_1 + \cdots + a_\ell$ is the total number of parts, and
  its size $n = \sum_{i=1}^\ell a_i \cdot \lambda_i$ is the sum of its
  parts.  We write $\ipart_{n}$ to denote the collection of all
  integer partitions of $n$.
\end{defn}

\begin{figure}
  \begin{tikzcd}[scale=.95]
    \bx=(a,b,a,c,d,c,e) \arrow[r,rightarrow] \arrow[d,rightarrow] &
    \textrm{$[\lambda] = 13|2|46|5|7$} \arrow[d,rightarrow] & \Xb^n
    \arrow[r,rightarrow] \arrow[d,rightarrow] & \spart_n
    \arrow[d,rightarrow] \\ M=a^2b^1c^2d^1e^1 \arrow[r,rightarrow] &
    \lambda=2^21^3 & \mult_n(\Xb) \arrow[r,rightarrow] & \ipart_n 
  \end{tikzcd}
  \caption{An $n$-tuple $\bx \in \Xb^n$ determines a set partition
    $[\lambda] \vdash [n]$ in $\spart_n$, an $n$-element multiset
    $M \in \mult_n(\Xb)$, and an integer partition $\lambda \vdash n$ in
    $\ipart_n$. In the example $n=7$. \label{fig:part-mult}}
\end{figure}

The maps in Figure~\ref{fig:part-mult} form a commuting
square, and they are defined as follows.
  
\begin{defn}[Maps]\label{def:set-maps}
  The map $\Xb^n \to \spart_n$ sends each $n$-tuple $\bx$ to the set
  partition $[\lambda]$ recording which of its coordinates are equal.
  The map $\Xb^n \to \mult_n(\Xb)$ sends each $n$-tuple $\bx$ to its
  multiset $M$ of coordinates.  The map $\mult_n(\Xb) \to \ipart_n$
  sends a multiset $M$ to its multiset of multiplicities.  And the map
  $\spart_n \to \ipart_n$ sends the set partition $[\lambda]$ to the
  multiset of its blocks sizes. The vertical maps are always onto and
  the horizontal maps are onto since $\Xb$ has infinitely many elements
  (Remark~\ref{rem:spaces}). For the top and left maps, we name the
  maps by their output, writing $M = \mult(\bx)$ and $[\lambda] =
  \spart(\bx)$.  For the three maps (horizontal, vertical and implicit
  diagonal) that end at $\ipart_n$, we describe $\lambda$ as the
  \emph{shape of $M$}, the \emph{shape of $[\lambda]$} and the
  \emph{shape of $\bx$}. and write $\lambda = \shape(M) =
  \shape([\lambda]) = \shape(\bx)$.  The common name for these maps is
  unambiguous in practice since the object being evaluated determines
  the domain.
\end{defn}

\begin{rem}[Functions]\label{rem:functions}
  The elements of the four spaces in Figure~\ref{fig:part-mult} can
  also be interpreted as functions.  A point $\bx \in \Xb^n$ corresponds
  to a function $f\colon [n] \to \Xb$, and note that both the domain and
  range have distinguishable elements.  If we make the elements of the
  domain / the range / both the domain and the range
  indistinguishable, then $f$ carries less information, and the
  remaining information is captured by the multiset $M$ / the set
  partition $[\lambda]$ / the shape $\lambda$.
\end{rem}

The integer partitions of $n$ can be viewed as a poset or as an
acyclic category.

\begin{rem}[Acyclic categories]
  An \emph{acyclic category} is a category where (1) the only arrows
  starting and ending at the same object are the identity arrows, and
  (2) if there is a nonidentity arrow $f$ from $a$ to $b$, there is no
  arrow from $b$ to $a$.  Acyclic categories are to posets as
  Hatcher's $\Delta$-complexes are to simplicial complexes. Every
  poset can be viewed as an acyclic category by using the elements as
  objects and drawing a single arrow $p \to q$ if and only if $p \leq
  q$.  In the other direction, every acyclic category has an
  underlying poset structure by defining $p \leq q$ if there exists an
  arrow $p \to q$.  Note that sequentially applying both constructions
  to a poset reproduces the origin poset, but applying both
  constructions to an acyclic category produces a quotient category
  where all arrows with the same endpoints have been identified.
  Acyclic categories were introduced by Andre Haefliger under the name
  ``small categories without loops'' to study complexes of groups
  \cite{Haefliger-91}. See also \cite{bridson-haefliger}.  The
  ``acyclic category'' terminology is from \cite{Kozlov-08}.
\end{rem}

\begin{defn}[Ordering integer partitions]\label{def:ordering-int-part}
  If $\lambda$ is an integer partition and $\pi$ is a partition of the
  parts of $\lambda$ then there is a new partition $\mu$ whose parts
  are the sums of the numbers in the blocks of $\pi$. We write 
  $\lambda \stackrel{\pi}{\to} \mu$, or simply $\lambda \to \mu$ when
  the partition $\pi$ is implicitly understood. If $\lambda =
  3^12^11^2 \vdash 7$ and $\pi = \{\{3,2\},\{1,1\}\}$, for example,
  then $\mu = 5^12^1$.  Note that multiple arrows can exist with the
  same endpoints.  In our example, $\pi' = \{\{3,1,1\},\{2\}\}$ is a
  distinct partition that also turns $\lambda$ into $\mu$.  The
  collection of all arrows of this form turn the set $\ipart_n$ into
  an acyclic category.  And it can be simplified to its underlying
  poset structure by defining $\lambda \leq \mu$ if there exists a
  partition $\pi$ with $\lambda \stackrel{\pi}{\to} \mu$.  With this
  partial order $\ipart_n$ becomes a bounded graded poset.  It is
  bounded below by the \emph{discrete partition} $\lambda = 1^n$,
  bounded above by the \emph{indiscrete partition} $\lambda = n^1$ and
  the grading is determined by the number of parts.  The map $\shape$
  from $\spart_n$ to $\ipart_n$ can be viewed as a functor between
  acyclic categories, which simplifies to an order-preserving
  rank-preserving poset homomorphism if the acyclic category structure
  on $\ipart_n$ is replaced with its underlying poset structure.  See
  Figure~\ref{fig:spart4-ipart4}.
\end{defn}

The acyclic category of integer partitions is used to index
stratifications of polynomial spaces. See Section~\ref{sec:poly-ll-maps}.

\begin{figure}
  \includegraphics[width=\textwidth]{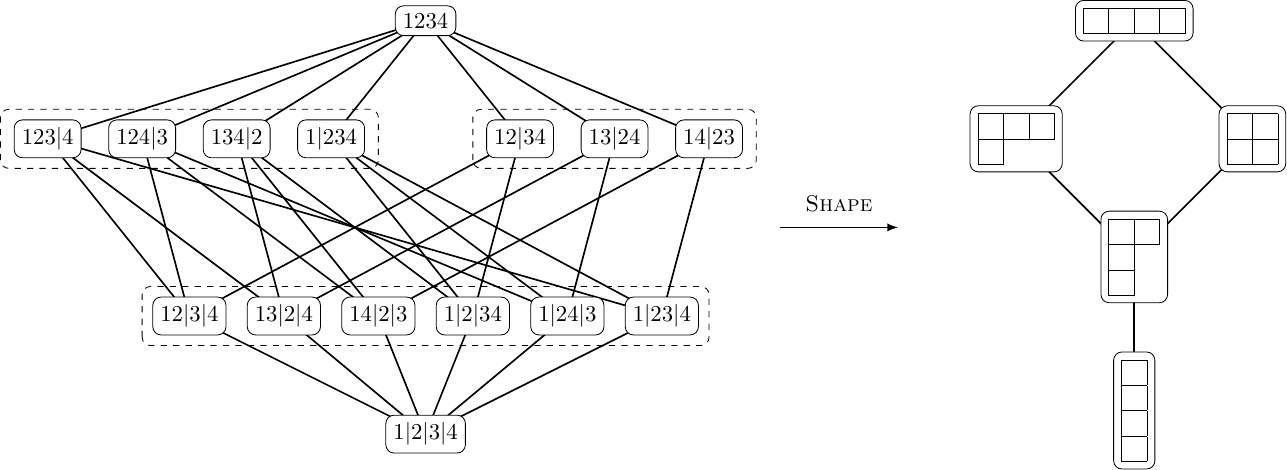}
  \caption{The poset homomorphism $\shape$ from the $15$ element poset
    $\spart_4$ to the $5$ element poset $\ipart_4$. The dashed
    lines enclose set partitions sent to the same integer
    partition.\label{fig:spart4-ipart4}}
\end{figure}

\begin{rem}[Set partition notation]\label{rem:int-set}
  Although we write $[\lambda] \vdash [n]$ to denote a set partition
  of shape $\lambda \vdash n$, the reader should note that this
  notation is ambiguous. In Figure~\ref{fig:spart4-ipart4}, for
  example, there are $4$ distinct set partitions of shape $3^11^1$,
  and $[3^11^1]$ could refer to any one of them. There is no such
  ambiguity at either extreme since the discrete set partition
  $1|2|\cdots|n$, with $n$ blocks of size $1$, is the unique set
  partition of shape $1^n$ and the indiscrete set partition $12\cdots
  n$, with $1$ block of size $n$, is the unique set partition of shape
  $n^1$.
\end{rem}

\section{Polynomials and Branched Covers}\label{sec:poly-brcover} 

This section records basic facts about polynomials
(\ref{subsec:poly}), branched covers between
surfaces~(\ref{subsec:branched-covers}), and planar branched covers
between disks (\ref{subsec:planar-covers}).

\subsection{Polynomials}\label{subsec:poly}
For any $U \subset \C$, let $\C_U$ be the complement $\C \setminus U$,
so that $\C_\zer$ is $\C^\ast$ with $\zer=\{0\}$. Let $\D$ be the
closed unit disk with unit circle boundary $\T = \partial \D$.

\begin{defn}[Polynomials]\label{def:poly}
  For each $d\in \N$, let $\poly_d \subset \C[z]$ be the collection of
  \emph{complex polynomials of degree $d$}.  Concretely, for $p \in
  \poly_d$, we write $p(z) = c_0 z^d + c_1 z^{d-1} + \cdots + c_{d-1}
  z + c_d$ with $c_0 \in \C_\zer$ and $c_i \in \C$ for $i \in [d]$.
  The polynomial $p$ is \emph{monic} if $c_0 =1$ and \emph{centered}
  if $c_1=0$. Let $\poly_d^m$ be the subspace of monic polynomials and
  let $\poly_d^{mc}$ be the subspace of monic, centered polynomials.
  Using \emph{coefficient coordinates}, there are natural
  homeomorphisms $\poly_d^{m} \cong \C^d$ and $\poly_d^{mc} \cong
  \C^{n}$ where $n = d-1$.
\end{defn}

\begin{defn}[Points and values]\label{def:cpt-cvl}
  Let $p \colon \C \to \C$ be a degree-$d$ polynomial. 
  The roots of the derivative $p'$ form an $n$-element 
  multiset (where $n = d-1$) called the \emph{critical points} of $p$.
  In
  symbols, if $p'(z) = d \cdot c_0 \cdot (z-z_1)^{n_1} \cdots
  (z-z_k)^{n_k}$, then $\cpt = \cpt(p) = \{ z_1^{n_1}, \ldots,
  z_k^{n_k} \}$. The $z_i$ are the $k$ critical points, the $n_i =
  n_{z_i}$ are their multiplicities, and we typically order the $z_i$
  so that the $n_i$ are weakly decreasing.  Note that $n = n_1 +
  \cdots + n_k 
  = d-1$ is the size of $\cpt$.  The \emph{critical
  values} of $p$ are the images of the critical points, and the
  $n$-element multiset $\cvl = \cvl(p)$ is $p(\cpt)$. The points in
  the domain and the values in the range that are not critical are
  \emph{regular points} and \emph{regular values}. A subset is
  \emph{regular} if every element is regular.  Finally, denote the
  full preimage of $\cvl$ by $\cpt^+ = \{z_1, \dots, z_m\}=
  p^{-1}(\cvl)$, where $m\geq k$. If $m>k$, the regular points in
  $\cpt^+$ are listed at the end.  We still have $n_1 + \cdots + n_m =
  n$ since $n_i =0$ for regular points.
\end{defn}

The following polynomial is going to be our running example throughout
Part~\ref{part:single-poly}.

\begin{example}\label{ex:deg5}
  Let $p$ be the complex polynomial
  \[z^5 + \left(\frac{-17+6i}{4}\right)z^4 +
  \left(\frac{73-63i}{15}\right)z^3 +
  \left(\frac{34-12i}{25}\right)z^2 +
  \left(\frac{-308+252i}{125}\right)z.\] Its critical points are
  $\cpt(p) =\{-\frac25, \frac25, \frac{7-7i}{5}, \frac{10+i}{5}\}$,
  all with multiplicity~$1$, and its (rounded) critical values are
  $\cvl(p) \approx \{.8-.6i, -.6+.5i, -8.5-4.3i, 3.6-6.9i\}$, again
  all with multiplicity~$1$.  The critical values are listed in the
  same order as the critical points: for example, $p(\frac25) \approx
  -.6+.5i$.  Figure~\ref{fig:deg5-original-range} on
  page~\pageref{fig:deg5-original-range} shows the range of $p$ with
  its $4$ critical values marked as yellow dots.
  Figure~\ref{fig:deg5-original-domain} on
  page~\pageref{fig:deg5-original-domain} shows the domain of $p$ with
  its $4$ critical points marked as yellow dots.
\end{example}

Polynomials are examples of planar branched covers, and general
branched covers are locally modeled on degenerate polynomials.

\begin{defn}[Degenerate polynomials]\label{def:degenerate}
  The polynomial $p(z) = a\cdot (z-b)^d +c$ of degree $d>1$ is
  \emph{degenerate} in the sense that both its critical point multiset
  $\cpt(p) = \{b^n\}$ and its critical value multiset $\cvl(p) =
  \{c^n\}$ are indiscrete, and as far from generic as possible.
  Regular values have $d$ point preimages under $p$, but the critical
  value $c$ has only one, so $c$ has $n=d-1$ missing preimages. 
  The \emph{power polynomial} $p \colon \C \to \C$ with
  $p(z) = z^d$ is the identity when $d=1$ and degenerate for $d>1$.
\end{defn}

\begin{defn}[Branched even coverings]\label{def:br-even-cover}
  Let $\db = (d_1,\ldots, d_m) \in \N$ be an $m$-tuple of natural
  numbers with sum $d = d_1 + \cdots + d_m$.  Let $\pb_\db \colon \Zb
  \to \Wb$ be the map where $\Wb = \D$, $\Zb = \bigsqcup_{i\in[m]} \D$
  is a disjoint union of $m$ copies of the closed unit disk $\D$, and
  the $i^{th}$ component map $\pb_i \colon \D \to \D$ is the power map
  $\pb_i(z) = z^{d_i}$.  And let $p_\db \colon Z \to W$ be the map from 
  the interior of the domain to the interior of the range.  When 
  $\db = \one = (1,1,\ldots,1)$, each $p_i$ is a homeomorphism and 
  $p_\db$ is the prototypical \emph{even covering} used to define a
  (finite sheeted) covering map.  For general $\db$ we say that
  $p_\db$ is a \emph{branched even covering}. When a degenerate 
  component map exists, $0 \in W$ is a (unique) critical value
  and its multiplicity is $d-m$. The critical points are at the origin 
  in components where the component map is degenerate.
\end{defn}

\subsection{Branched covers}\label{subsec:branched-covers}
Branched covers are surface maps that locally look like branched even
covers.  Here is what we mean by ``surface'' and ``looks like''.

\begin{defn}[Surfaces]\label{def:surfaces}
  Unless otherwise stated, all surfaces in this article are either
  closed compact oriented $2$-dimensional manifolds, possibly with
  boundary, or open surfaces homeomorphic to the interior of one of
  our closed surfaces.  We use boldface letters for closed surfaces,
  non-bold letters for open surfaces, and the same letter for a closed
  surface and its open interior, i.e. a closed surface $\Zb$ has
  interior $Z$.  For a connected closed surface $\Zb$, let $\wh \Zb$
  be the closed surface without boundary obtained by attaching a disk
  to each circle of $\partial \Zb$.  If $\wh \Zb$ has genus $g$, then
  we say $\Zb$ has genus $g$.  The \emph{Euler characteristics} are
  $\euler(\wh \Zb) = 2- 2g$ and $\euler(\Zb) =\euler(Z) = 2-2g-b$
  where $b$ is the number of components of $\partial \Zb$. In
  particular, for connected closed surfaces, $\euler(\Zb) =1$ if and
  only if  $\Zb$ is a closed disk with $g=0$, $b=1$.
\end{defn}

\begin{defn}[Maps]\label{def:surface-maps}
  For maps between surfaces we make a standing assumption that the
  range is always connected, but the domain is allowed to be
  disconnected. Note, however, that the number of connected components
  for any surface considered here is finite.  Let $\qb \colon \Zb \to \Wb$ be
  a map between surfaces.  A \emph{component map of $\qb$} is 
  a restriction of $\qb$ to a component of its domain.  If these have been
  indexed, $\qb_i$ is the restriction of $\qb$ to the $i^{th}$ component
  $\Zb_i$. 
\end{defn}

\begin{figure}
  \begin{tikzcd}[scale=1]
    X \arrow[r,rightarrow,"p"]  \arrow[d,leftarrow,"f"] & 
    Y \arrow[d,rightarrow,"g"] &
    Z \arrow[r,rightarrow,"p"]  \arrow[d,leftarrow,"h"] & 
    W \arrow[d,equal] \\ 
    Z \arrow[r,rightarrow,"q"] & W &
    Z \arrow[r,rightarrow,"q"] & W 
  \end{tikzcd}
  \caption{The maps $p$ and $q$ on the left look alike.  The maps $p$
    and $q$ on the right are topologically equivalent. The three
    labeled vertical maps are homeomorphisms, and both squares commute.
  \label{fig:equivalent}}
\end{figure}

\begin{defn}[Equivalent maps]\label{def:equivalent}
  We say that a map $q \colon Z \to W$ \emph{looks like} a map $p
  \colon X \to Y$ if there are homeomorphisms $f \colon Z \to X$ and
  $g \colon Y \to W$ such that $q = g \circ p \circ f$. There is also
  a more restricted notion when $X=Z$ and $Y=W$ and the homeomorphism
  of the range must be the identity map.  We say that $q$ is
  \emph{topologically equivalent} to $p$ if there is a
  self-homeomorphism $h \colon Z \to Z$ of the domain so that $q = p
  \circ h$.  In the language of \cite{LaZv04}, the first version is a
  \emph{flexible} notion of equivalence, and the second is a
  \emph{rigid} one.  See Figure~\ref{fig:equivalent}.
\end{defn}

\begin{defn}[Branched covers]\label{def:br-covers}
  Let $q \colon Z \to W$ be a map between open surfaces
  (Definition~\ref{def:surfaces}).  Let $D \subset W$ be an open disk
  and let $C = q^{-1}(D)$ be its preimage.  We say that $D$ is
  \emph{branched evenly covered} if the restricted map $q \colon C \to
  D$ looks like (Definition~\ref{def:equivalent}) a branched even
  covering map $p_\db$ for some tuple $\db$
  (Definition~\ref{def:br-even-cover}).  The map $q$ is an \emph{(open)
  branched covering map} if every point $w\in W$ has a neighborhood
  $D$ which is branched evenly covered.  For maps $\qb \colon \Zb \to
  \Wb$ between closed surfaces (Definition~\ref{def:surfaces}), we say
  that $\qb$ is a \emph{(closed) branched covering map} if it
  restricts to an open branched cover $q \colon Z \to W$ on the
  interiors and an ordinary covering map $\partial \qb \colon
  \partial \Zb \to \partial \Wb$ on the boundaries. Not only do
  closed branched covers restrict to open branched covers, but every
  open branched cover (of the open surfaces being considered) 
  extends to a closed branched cover.
\end{defn}

Branched covering maps retain many properties of ordinary covering
maps.

\begin{lem}[Surjective maps]\label{lem:surj-maps}
  If $\qb \colon \Zb \to \Wb$ is a branched cover, then $\qb$ is
  surjective. In particular, for any $w\in \Wb$, the natural map that
  sends a preimage of $w$ to the (index of the) component containing
  it is onto.
\end{lem}

\begin{proof}
  Let $\Zb_1$ be a component of $\Zb$ with $z_1 \in \Zb_1$ and $w_1 =
  \qb(z_1)$.  Since $\Wb$ is connected
  (Definition~\ref{def:surface-maps}), for any $w_2 \in \Wb$ there is
  a regular path $\beta$ from a value near $w_1$ to a value near
  $w_2$, where ``near'' means in a branched evenly covered neighborhood
  $D_i$ of $w_i$.  Once $\beta$ is lifted to start in the component of
  $\qb^{-1}(D_1)$ containing $z_1$, it ends in a component of
  $\qb^{-1}(D_2)$ containing a lift of $w_2$.
\end{proof}

\begin{defn}[Branch points]\label{def:br-points}
  Let $q \colon Z \to W$ be an open branched cover. By definition,
  every $z\in Z$ has an open disk neighborhood sent to an open disk
  neighborhood containing $w= q(z)$ that looks like a power map $p(z)
  = z^d$ for some unique $d = d_z$. The positive integer $d_z$ is the
  \emph{degree of $z$} and the nonnegative integer $n_z = d_z-1$ is
  the \emph{multiplicity of $z$}. When $d_z=1$ and $n_z=0$, $p(z)=z$
  is the identity map and $q$ is a local homeomorphism near $z$. When
  $d_z>1$ and $n_z > 0$, the point $z$ is a \emph{branch point of
  multiplicity $n_z$}.  Branch points are also called \emph{critical
  points} or \emph{ramification points}.  The critical values, regular
  values, regular points, and the preimage set $\cpt^+$ of $q$ are
  defined exactly as in Definition~\ref{def:cpt-cvl}.  For closed
  branched covers $\qb \colon \Zb \to \Wb$, the points $z \in \partial
  \Zb$, being regular, have degree $d_z=1$ and multiplicity $n_z=0$.
\end{defn}

A polynomial is a branched cover, its critical points are its branch
points, and the two notions of multiplicity agree, so the rest of the
terminology is also consistent.  Globally, a branched cover is a
covering map away from finitely many points.

\begin{defn}[Degree and metric]\label{def:degree-metric}
  Let $\qb \colon \Zb \to \Wb$ be a closed branched cover.  If $V =
  \set(\cvl) \subset W$ is the set of critical values and $U =
  \qb^{-1}(V) = \cpt^+ \subset Z$ is the full set of preimages, then
  the restricted map $\qb \colon (\Zb \setminus U) \to (\Wb \setminus
  V)$ is a covering map. Moreover, since $\Wb$ is connected
  (Definition~\ref{def:surface-maps}), $\Wb \setminus V$ is connected,
  and $\qb$ has a constant \emph{global degree} $d$, making $\qb$ a
  \emph{$d$-branched cover}.  Branched covers are also called
  \emph{ramified covers}.  And note that if $\Wb$ has a metric, there
  is a unique \emph{induced metric} on $\Zb$ so that $\qb$ is a local
  isometry except at $\cpt$.
\end{defn}

\begin{rem}[Riemann--Hurwitz formula]\label{rem:riemann-hurwitz}
  Let $\qb \colon \Zb \to \Wb$ be a $d$-branched cover. If $m =
  \size{\cpt^+}$ and $\ell = \size{\cvl}_\set$, then $\qb$ satisfies
  the \emph{Riemann--Hurwitz formula}: $(\euler(\Zb) - m) = d \cdot
  (\euler(\Wb) - \ell)$, because Euler characteristic is
  multiplicative for covering maps.
\end{rem}

\begin{lem}[Disk preimages]\label{lem:disk-preimage}
  If $\qb \colon \Zb \to \Wb$ is a closed branched cover where $\Zb$
  is connected and $\Wb$ is a closed disk with one critical value, then
  $\Zb$ is a closed disk with one critical point.  Similarly, if $q
  \colon Z \to W$ is an open branched cover where $Z$ is connected and
  $W$ is an open disk with one critical value, then $Z$ is an open disk
  with one critical point.
\end{lem}

\begin{proof}
  Open branched covers extend to closed branched covers in our
  setting, so we only need to prove the closed version. 
  Let $g$ be the genus of $\Zb$, let
  $b$ be the number of components of $\partial \Zb$ and let $m
  =\size{\cpt^+}$.  Since $\euler(\Wb)=1$ and
  $\size{\cvl}_\set=1$, we have $\euler(\Zb) = \size{\cpt^+}$ by
  Remark~\ref{rem:riemann-hurwitz}, so $2-2g-b = m$ and therefore $2 =
  2g+b+m$.  We know that $g \geq 0$ by definition, $b>0$ because
  $\partial \Zb = \qb^{-1}(\partial \Wb)$ is nonempty, and $m >0$
  because $\qb$ has at least one critical point, so the only solution
  is $g=0$ and $b=m=1$, making $\Zb$ a closed disk.
\end{proof}

\begin{lem}[Degree and preimages]\label{lem:degree-pre}
  Let $\qb \colon \Zb \to \Wb$ be a $d$-branched cover.  For any $w
  \in \Wb$, the sum of the degrees of its preimages is $d$.
\end{lem}

\begin{proof}
  Let $\qb^{-1}(w) = \{z_1,\ldots, z_m\}$ and let $d_i$ be the degree
  of $z_i$.  Let $D$ be a branched evenly covered neighborhood of $w$
  with preimage $C = \qb^{-1}(D)$.  The $\sum d_i$ must be equal to
  $d$ since any regular point in $D$ has $d_i$ preimages in the
  component of $C$ containing $z_i$, and $d$ preimages total.
\end{proof}
  
\begin{lem}[Degree and multiplicity]\label{lem:degree-mult}
  If $\qb \colon \Zb \to \Wb$ is a $d$-branched cover and $n$ is the
  common size of the multisets $\cpt$ and $\cvl$, then $n+ \euler(\Zb)
  = d \cdot \euler(\Wb)$.  In particular, when $\Wb$ is a disk and $\Zb$ is a
  disjoint union of $b$ disks, $n+b =d$.
\end{lem}

\begin{proof}
  Suppose $U = \cpt^+ = \{z_1,\ldots, z_m\}$ and $V = \set(\cvl) =
  \{w_1,\ldots,w_\ell\}$.  The sum of the degrees of the $z_i$ is $\sum_i
  d_i = \sum_i (n_i+1) = n+m$. By Lemma~\ref{lem:degree-pre}, the sum
  of degrees of $q^{-1}(w_j) = d$ for each $j \in [\ell]$, so the sum of
  the degrees of the $z_i$ is $d\cdot \ell$. Thus $n+m = d\cdot \ell$. By
  Remark~\ref{rem:riemann-hurwitz}, $(\euler(\Zb) - m) = d\cdot
  (\euler(\Wb) - \ell)$.  Adding these two equations completes the proof.
\end{proof}

These results hold, of course, for our running example.
  
\begin{example}\label{ex:riemann-hurwitz}
  The polynomial $p$ of Example~\ref{ex:deg5} has degree $d=5$ and
  multiplicity $n=4$ (Lemma~\ref{lem:degree-mult}).  It has $4$
  critical values, so $\size{V}=\size{\cvl}_\set = 4$.  Each critical
  value has $4$ preimages, one critical point of degree $2$ and three
  regular points of degree $1$, so the sum of the degrees of every
  value is $5$ (Lemma~\ref{lem:degree-pre}) and $\size{U} =
  \size{\cpt^+} = 16$.  Finally, the domain and range are $\C$ with
  $\euler(\C)=1$, so the Riemann--Hurwitz formula
  (Remark~\ref{rem:riemann-hurwitz}) is satisfied since $(1-16) = 5
  \cdot (1-4)$.
\end{example}

\subsection{Planar branched covers}\label{subsec:planar-covers}
We now restrict attention to branched covers between subsurfaces of spheres.

\begin{defn}[Surfaces in spheres]\label{def:disks-curves}
  For any closed surface $\Wb \subset \C \subset \Ch = \Cp^1$ with
  interior $W$, let $\beta = \partial \Wb$ be the union of its simple
  closed boundary curves, let $W^c$ be the complement of $\Wb$ in $\Ch$
  and let $\Wb^c = W^c \sqcup \beta$ be the closed surface which is the
  complement of $W$ in $\Ch$.  We call $\Wb^c$ the \emph{closed
  complement of $\Wb$}. Note that $\Ch = W \sqcup \beta \sqcup W^c =
  \Wb \cup \Wb^c$, and $\Wb$ is a closed disk if and only if
  $\beta$ is a Jordan curve if and only if $\Wb^c$ is a closed disk.
  More generally, $\Wb$ is a disjoint union of closed disks if and only
  if its closed complement $\Wb^c$ is connected. 
\end{defn}
 
\begin{defn}[Planar branched covers]\label{def:planes}
  A \emph{(closed) planar branched cover} is a branched cover $\qb
  \colon \Zb \to \Wb$ where $\Zb$ and $\Wb$ are closed disks, and an
  \emph{(open) planar branched cover} is a branched cover $q \colon Z
  \to W$ where $Z$ and $W$ are open disks.  Closed planar branched
  covers restrict to open planar branched covers and every open planar
  branched cover extends to a closed planar branched cover.  Also,
  open planar branched covers look like branched covers $p\colon \C
  \to \C$, hence the name.  Closed planar branched covers look like 
  branched covers
  $\pb \colon \D \to \D$. A planar branched cover with only one
  critical point is \emph{degenerate}.
\end{defn}

\begin{example}[Degenerate branched covers]\label{ex:degenerate-cover}
  The map $p \colon \C \to \C$ defined by $p(z)= z^d$ extends to a map
  $\wh p \colon \Ch \to \Ch$, where $\Ch = \Cp^1$. The map $\wh p$
  preserves the decomposition $\Ch = \D \cup \D^c$, so it splits into
  two degenerate planar $d$-branched covers $\pb \colon \D \to \D$ and
  $\pb^c \colon \D^c \to \D^c$ with critical points $0$ and $\infty$,
  respectively, that overlap on the $d$-fold cover $\partial \pb
  \colon \T \to \T$.
\end{example}

Complex polynomials form the most natural examples of open planar
branched covers, and in the sense described below, the two 
notions actually coincide.

\begin{rem}[Polynomials and Branched Covers]\label{rem:poly-branch}
  Let $[p]_\topolo$ denote the topological equivalence class of a 
  planar branched cover $p \colon \C \to \C$ presented 
  in Definition~\ref{def:equivalent}. When $p,q \colon \C \to \C$ are 
  complex polynomials, we say that $p$ and $q$ are \emph{linearly 
  equivalent} if there is a linear function $h(z) = az +b$ so that 
  $p = q \circ h$, i.e. $p(z) = q(az+b)$.  Let $[p]_\linear$ 
  denote the linear equivalence class of a complex polynomial $p$.
  Since polynomials are planar branched covers and linear maps 
  are homeomorphisms, there is a well defined function $[p]_\linear 
  \mapsto [p]_\topolo$ from polynomials up to linear equivalence to
  planar branched covers up to topological equivalence.  In fact, this
  function is a bijection: every planar branched cover $p \colon \C
  \to \C$ is topologically equivalent to a polynomial (surjectivity), and
  two polynomials that are topologically equivalent are linearly
  equivalent (injectivity).  See \cite[Chapter~$1$]{LaZv04}.
\end{rem}

\begin{defn}[Spherical branched covers]\label{def:spheres}
  A \emph{spherical branched cover} is a branched cover $\wh \qb
  \colon \wh \Zb \to \wh \Wb$ where $\wh \Zb$ and $\wh \Wb$ are
  $2$-spheres.  Every planar $d$-branched cover $\qb \colon \Zb \to
  \Wb$ extends to a spherical $d$-branched cover $\wh q \colon \wh \Zb
  \to \wh \Wb$ by attaching the degenerate planar $d$-branched cover
  $\pb^c \colon \D^c \to \D^c$ (Example~\ref{ex:degenerate-cover}) to the
  boundaries.  This completes $\Zb$ to $\wh \Zb$ and $\Wb$ to $\wh
  \Wb$, and $\pb^c$ agrees with $\qb$ on the boundary map $\partial \qb
  \colon \partial \Zb \to \partial \Wb$. In the other direction, if
  $\wh \qb \colon \wh \Zb \to \wh \Wb$ is a spherical branched cover
  and there is a point $w^c \in \wh \Wb$ with only one preimage $z^c \in
  \wh \Zb$, then removing a branched evenly covered neighborhood $W^c$
  of $w^c$ and its preimage $Z^c$ containing $z^c$ leaves a planar
  $d$-branched cover.
\end{defn}

\begin{example}[Non-disk preimages]\label{ex:non-disks}
  Let $\Wb$ be a closed disk in $\Ch$ that does not contain $0$ or
  $\infty$, and let $\Wb^c$ be its closed disk complement.  Under $\wh
  p$, the spherical $d$-branched cover of Example~\ref{ex:degenerate-cover},
  the preimage $\Zb$ of $\Wb$ is $d$ disjoint topological disks since
  $\Wb$ is regular and the preimage $\Zb^c$ of $\Wb^c$ is a connected
  surface with $d$ boundary components.  In particular, $\Zb^c$ is not
  a disk, even though $\Wb^c$ is a disk. 
\end{example}

The preimage of a disk under a spherical branched cover need not be a
disk, but the preimage under a planar branched cover is a union of
disks.

\begin{prop}[Disks and preimages]\label{prop:disk-pre}
  Let $\pb \colon \D \to \D$ be a closed planar branched cover and let
  $\Wb$ be a closed disk in the range.  If $\partial \Wb$ is regular,
  then $\Zb = \pb^{-1}(\Wb)$ is a disjoint union of $b$ closed disks
  $\Zb_i$, its closed complement $\Zb^c$ is connected, and its
  boundary $\partial \Zb$ is a set of disjoint nonnested simple closed
  curves. Moreover, the component maps $\pb_i \colon \Zb_i \to \Wb$
  are closed planar branched covers, and the number of components 
  $b$ is equal to $d-n_w$, where $n_w$ is the total multiplicity of the 
  critical values in $\Wb$.
\end{prop}

\begin{proof}
  Since $\partial \Wb$ is regular, its preimage $\partial \Zb$ is a
  collection of simple closed curves, and $\Zb^c={\wh
    \pb}^{-1}(\Wb^c)$ is a closed submanifold in $\Ch$ with $\infty
  \in \Zb^c$. Since $\infty$ is the only preimage of $\infty$
  (Definition~\ref{def:spheres}), $\Zb^c$ is connected
  (Lemma~\ref{lem:surj-maps}), so $\Zb = \Zb_1 \sqcup \cdots \sqcup
  \Zb_k$ is disjoint union of disks
  (Definition~\ref{def:disks-curves}).  The restricted maps $\pb_i
  \colon \Zb_i \to \Wb$ between spaces homeomorphic to $\D$ is itself
  a branched cover since it still satisfies the required conditions:
  local conditions are local, a connected component of a cover is a
  cover, and the boundary map is a covering.  And 
  $b=d-n_w$ by Lemma~\ref{lem:degree-mult}.
\end{proof}

\section{Cell Complexes and Cellular Maps}\label{sec:complex-map}

Given a branched cover $p \colon \C \to \C$ and a rectangle
$\closedsquare$ in the range, we use the critical values $\cvl(p)$ to
define $4$ complexes: two coordinate complexes $\Qb_p$ and $\Qb'_p$ in
the range and two branched coordinate complexes $\Pb_p$ and $\Pb'_p$
in the domain.  The focus here is on rectangular coordinates, but
similar constructions using polar coordinates were given in the first
paper in this series \cite{DoMc22}.  After defining the coordinate
complexes $\Qb_p$ and $\Qb'_p$ in the range
(\ref{subsec:cplx-range}), we record the properties of branched
cellular maps (\ref{subsec:br-cell-maps}) that make it possible to
define the branched coordinate complexes $\Pb_p$ and $\Pb'_p$ in the
domain (\ref{subsec:cplx-domain}).

\subsection{Complexes in the range}\label{subsec:cplx-range}
The constructions in the range are very straightforward since we are
simply subdividing a rectangle $\closedsquare$ vertically and
horizontally through the critical values of $p$ in the case of
$\Qb'_p$ and at representative regular values in the case of $\Qb_p$.
The two constructions are not quite ``dual'' to each other, since taking
cellular duals is not an involution on planar cell complexes with
boundary.  We begin with an example of two cell complexes in $\R$,
followed by the definitions to make it precise.

\begin{example}[$\Ib_p$ and $\Ib'_p$]
  The real critical value interval $\Ib'_p$ for the polynomial of
  Example~\ref{ex:deg5} is the interval $[-10,5]$ subdivided at
  $C'=\{x'_1,x'_2,x'_3,x'_4\}$, the real parts of $\cvl(p)$. The
  real regular value interval $\Ib_p$ is an interval that has $\Ib'_p$
  as its cellular dual.  Figure~\ref{fig:deg5-interval-original-dual}
  shows $\Ib_p$ on the top and $\Ib'_p$ on the bottom.  In the
  ``critical'' interval $\Ib'_p$ the $x'_i$ label vertices and in the
  ``regular'' interval $\Ib_p$ they label edges.
\end{example}

\begin{figure}
    \centering
    \includegraphics[scale=.5]{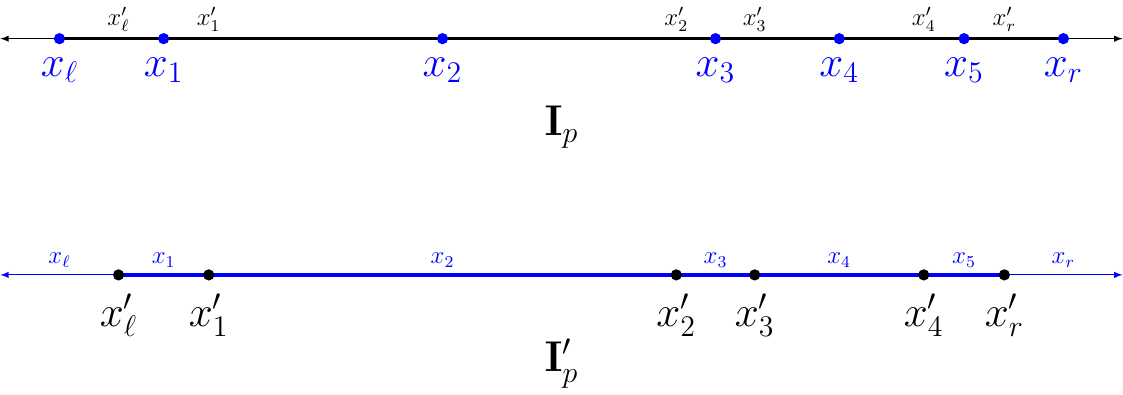}
    \caption{The real critical value interval $\Ib'_p$, shown on the
      bottom, is the interval $[-10,5]$ subdivided at
      $C'=\{x'_1,x'_2,x'_3,x'_4\}$, the real parts of the critical
      values of our running example.  The real regular value interval
      $\Ib_p$, shown on the top, is an undual of $\Ib'_p$.  It is a
      subdivision of $[-11,6]$, subdivided at the midpoints of the
      midpoints of the edges of $\Ib'_p$.  In the ``critical'' interval
      $\Ib'_p$ the points in $C'$ label vertices and in the ``regular''
      interval $\Ib_p$ they label edges.
    \label{fig:deg5-interval-original-dual}}
\end{figure}

\begin{defn}[Intervals]\label{def:int-subdivide}
  Let $[x_l,x_r] \subset \R$ be a compact interval and let $\Ib$ be
  $[x_l,x_r]$ with the structure of a cell complex.  The \emph{basepoint of $\Ib$} 
  is $x_\ell$. The \emph{combinatorial cell structure} of $\Ib$ is completely determined
  by the number $k$ of vertices in its interior.  The cell structure
  of the original interval has $1$ edge and $2$ vertices and we say it
  has been $0$-subdivided. More generally we say that $\Ib$ has been
  \emph{$k$-subdivided} when it has $k+2$ vertices $x_\ell=v_0 < v_1 <
  \cdots < v_k < v_{k+1} = x_r$ indexed in the order they occur in
  $\R$, and $k+1$ open edges $e_i$ with endpoints $v_{i-1}$ and $v_i$,
  $i \in [k+1]$.  A \emph{metric cell structure} on $\Ib$ corresponds
  to a \emph{metric $k$-subdivision}.  It has an \emph{absolute}
  description that records the locations $x_i \in (x_l,x_r)=I =
  \inter(\Ib)$ of the vertices $v_i$, $i\in [k]$.  This is a
  $k$-element set $C = \{x_1, \ldots, x_k\} \in \set_k(I)$, and we
  write $\Ib_C$ for $\Ib$ with this metric cell
  structure. Alternatively, it is sufficient to give a \emph{relative}
  description, by listing the relative widths of the $k+1$
  subintervals.  The \emph{weight} or \emph{relative width} $\width_i$ of the $i^{th}$
  subinterval $(x_{i-1},x_i)$ is the positive real $\width_i =
  \frac{\size{x_i-x_{i-1}}}{\size{x_r-x_\ell}}$. Note that $\width_1 +
  \cdots + \width_{k+1} =1$.  The metric cell structures on an interval
  $\Ib_C$ that has been $k$-subdivided are in bijection with
  $(k+1)$-tuples $\bary(\Ib_C) = (\width_1,\ldots, \width_{k+1})$ of
  positive reals with sum $1$.  In Definition~\ref{def:geom-comb},
  the $\width_i$ are used as the barycentric coordinates of a point in an
  open $k$-simplex, hence the name.
\end{defn}

\begin{defn}[Dual intervals]\label{def:dual-intervals}
  Let $\Ib, \Ib' \subset \R$ be subdivided intervals $\Ib=\Ib_k$ and
  $\Ib'=\Ib'_{k'}$.  We say that $\Ib' = (\Ib)'$ is a \emph{dual} of $\Ib$ and
  $\Ib = \int(\Ib')$ is an \emph{undual} of $\Ib'$ if the $k'+2$ vertices of $\Ib'$
  are in bijection with the $k+1$ edges of $\Ib$ and each vertex of
  $\Ib'$ is contained in the corresponding edge of $\Ib$.  Concretely,
  $k'= k-1$ and the vertices $V$ of $\Ib$ and the vertices $V'$ of
  $\Ib'$ strictly alternative in $\R$, with $x_i < x'_i < x_{i+1}$ for
  $i = 0,\ldots,k'+2$.  Given $\Ib$, creating a dual $\Ib' = (\Ib)'$ involves choosing a point from
  each bounded component of $\R \setminus V$.  Given $\Ib'$, creating an undual 
  $\Ib = \int(\Ib')$ involves choosing a point from each component of
  $\R \setminus V'$.
\end{defn}    

The derivative and integral metaphor of Definition~\ref{def:dual-intervals} is 
inspired by Rolle's Theorem.  If the vertices $x_i$ of $\Ib$ are the distinct simple roots of 
real polynomial $p$, then the roots of its derivative $p'$ are one choice for the vertices 
$x'_i$ of its dual $\Ib'$. 

\begin{figure}
  \centering
  \includegraphics[scale=.45]{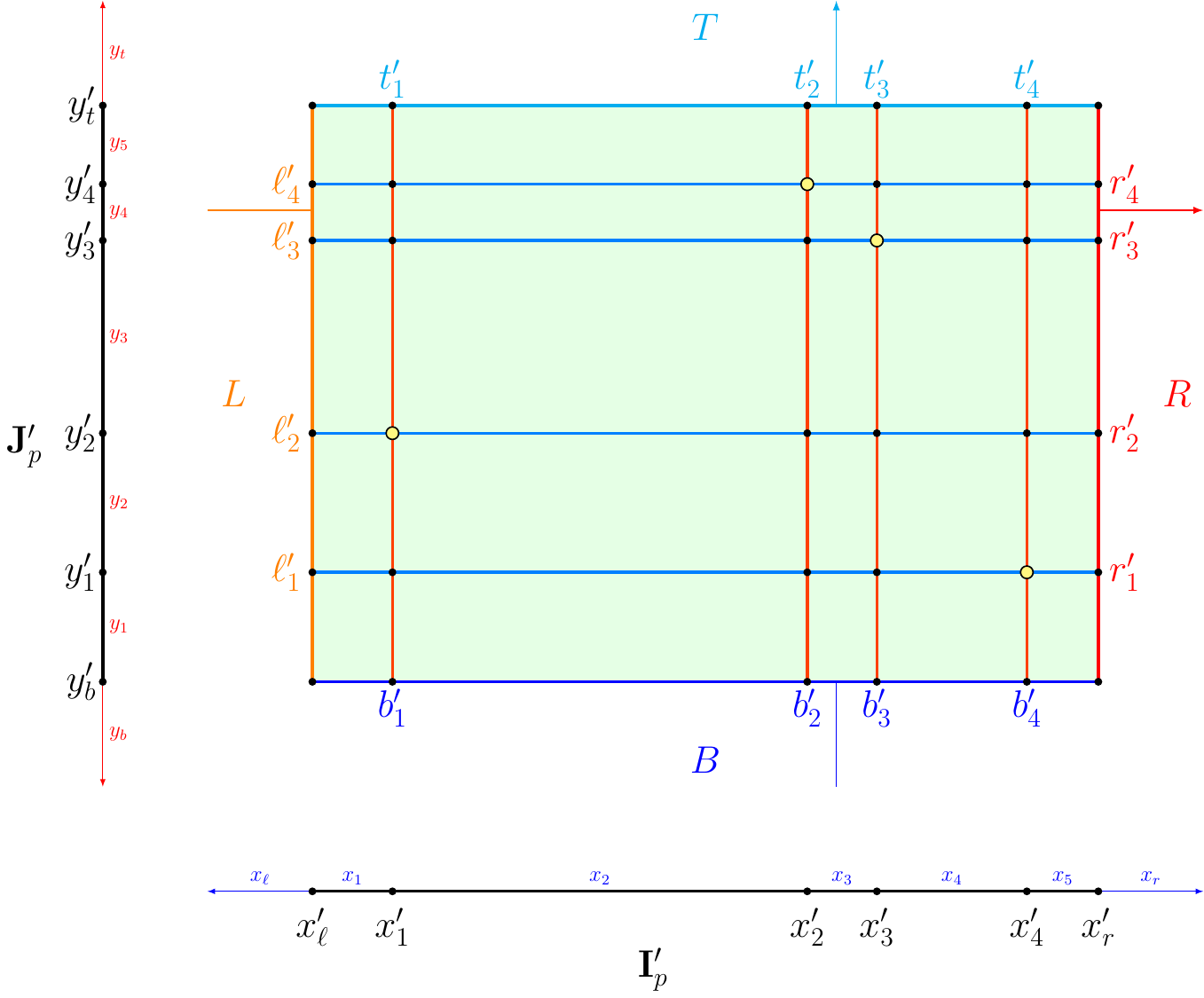}
  \caption{The range of the polynomial $p$ of Example~\ref{ex:deg5}
    with its $4$ critical values marked as yellow dots.  They are
    shown inside a rectangle $[-10,5] \times [-9,2] \subset \R^2 = \C$
    that has been subdivided into the critical value complex
    $\Qb'_p$.}
  \label{fig:deg5-original-range}
\end{figure}

\begin{example}[Critical value complex $\Qb'_p$]\label{ex:Q'_p}
  Let $p$ be the polynomial of Example~\ref{ex:deg5} and let
  $\closedsquare = [-10,5] \times [-9,2] \subset \R^2 = \C$ be a
  rectangle in its range, chosen to contain $\cvl(p)$. The
  \emph{critical value complex $\Qb'_p$ of $p$ with rectangle
  $\closedsquare$} is the rectangle $\closedsquare$ after it has been
  subdivided vertically and horizontally through the critical values
  of $p$. Figure~\ref{fig:deg5-original-range} shows the range of $p$ with
  its $4$ critical values marked as yellow dots.  The metric cell
  complex $\Qb'_p$ is the product of factor metric cell complexes
  $\Ib'_p$ and $\Jb'_p$ which subdivide the factors of $\closedsquare$
  at the real and imaginary parts of $\cvl(p)$.  The real parts of
  $\cvl(p)$, listed in increasing order, are
  $C'=\{x'_1,x'_2,x'_3,x'_4) \approx \{-8.5,-.6,.8,3.6\}$, and the
  imaginary parts in increasing order, are $D' =
  \{y'_1,y'_2,y'_3,y'_4\} \approx \{-6.9,-4.3,-.6,.5\}$. The metric on
  $\Qb'_p$ can be recorded by listing the relative widths of the open
  subintervals of $\Ib'_p$ and $\Jb'_p$.  In this case, $\bary(\Ib'_p)
  \approx (.102,.528,.089,.191,.091)$ and $\bary(\Jb'_p) \approx
  (.190,.241,.334,.098,.136)$.  Since we are recording relative
  widths, the sum of each list is $1$. They are used as barycentric coordinates in 
  Definition~\ref{def:geom-comb}. 
\end{example}
  
\begin{figure}
  \centering
  \includegraphics[scale=.45]{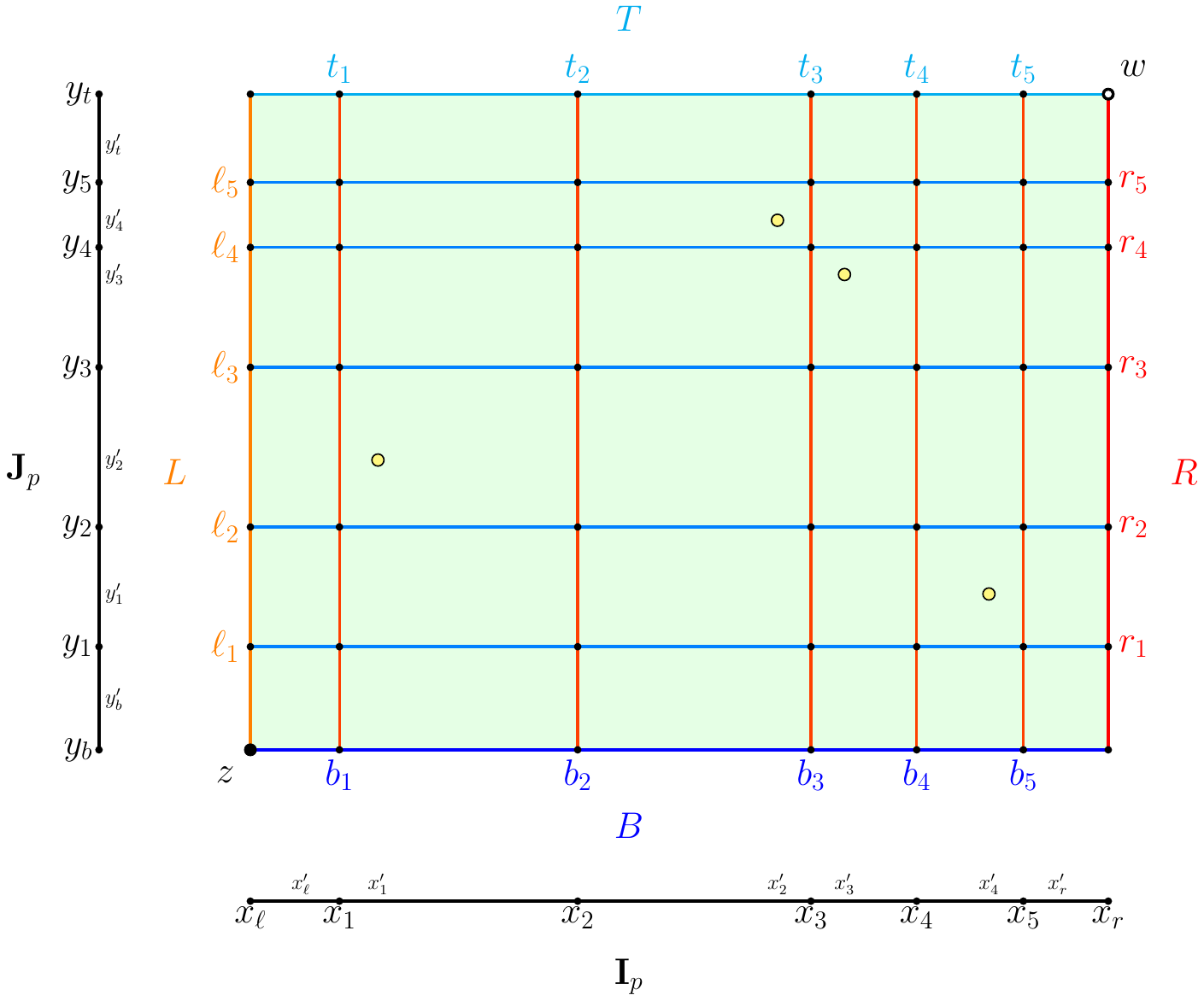}
  \caption{The range of the polynomial of Example~\ref{ex:deg5} with
    its $4$ critical values marked as yellow dots.  They are shown
    inside a rectangle $[-11,6] \times [-10,3] \subset \R^2 = \C$ that
    has been subdivided into the regular value complex
    $\Qb_p$.  The basepoint $z$ and breakpoint $w$ are also marked.
    \label{fig:deg5-original-dual-range}}
\end{figure}

\begin{example}[Regular value complex $\Qb_p$]\label{ex:Q_p}
  The regular value complex $\Qb_p$ for the polynomial $p$ of
  Example~\ref{ex:deg5} with rectangle $\closedsquare$ is constructed
  from the critical value complex $\Qb'_p$ and its factor complexes
  $\Ib'_p$ and $\Jb'_p$.  Specifically, if $\Ib_p$ is any interval
  with a cell structure whose dual is $\Ib'_p$ and $\Jb_p$ is any interval with a cell structure
  whose dual is $\Jb'_p$, then the complex $\Qb_p = \Ib_p \times \Jb_p$ whose
  dual is $\Qb'_p$ is the \emph{regular value complex for $p$ with rectangle 
  $\closedsquare$}.  The choices made when constructing $\Qb_p$ from $\Qb'_p$ 
  impact the metric on $\Qb_p$, but its combinatorial cell structure is well-defined.  
  Let $\Ib$ and $\Jb$ be the intervals that are subdivided to form $\Ib_p$ and $\Jb_p$, 
  and let $\Qb$ be the rectangle that is subdivided to form $\Qb_p$.  
  Figure~\ref{fig:deg5-original-range} shows one choice for $\Qb = \Ib
  \times \Jb$.  The enlarged rectangle is $[-11,6] \times [-10,3]$,
  and points $C = \{x_1,\ldots, x_5\}$ and $D = \{y_1,\ldots,y_5\}$
  are chosen to be the midpoints of the corresponding intervals in
  $\Ib'_p$ and $\Jb'_p$, respectively.  The vertex $z=(x_\ell,y_b)$ is 
  its \emph{basepoint} and the opposite vertex $w=(x_r,y_t)$ is its \emph{breakpoint}.
\end{example}

\begin{defn}[Coordinate rectangles]\label{def:coord-rectangle}
  Let $\Re\colon \C \to \R$ and $\Im\colon \C \to \R$ be maps so that
  for any $z = x+yi \in \C$, $\Re(z)=x$ is its \emph{real part}, and
  $\Im(z) = y$ is its \emph{imaginary part}.  Let $I = (x_\ell,x_r)$
  and $J = (y_b,y_t)$ be two non-empty open intervals in $\R$.  Their
  closures $\Ib = [x_\ell,x_r]$ and $\Jb = [y_b,y_t]$ have a natural
  cell structure with open edges $I$ and $J$, vertices $x_\ell$,
  $x_r$, $y_b$ and $y_t$, and \emph{basepoints} $x_\ell$ and $y_b$.  
  They determine a \emph{coordinate rectangle} $\Qb = \Ib +\Jb i = \{ z \mid \Re(z) \in
  \Ib, \Im(z) \in \Jb \}$.  Equivalently, $\Qb = \Ib \times \Jb$ under
  the natural identification of $\C$ with $\R^2$.
\end{defn}
  
\begin{figure}
  \centering
  \begin{tabular}{ccc}
    \includegraphics[width=0.4\linewidth]{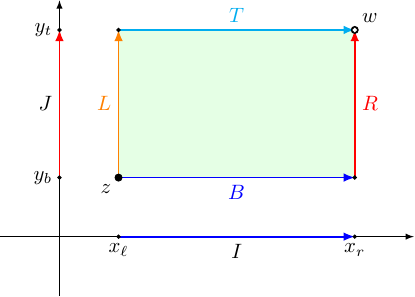}
    & \hspace*{1em} &
    \includegraphics[width=0.4\linewidth]{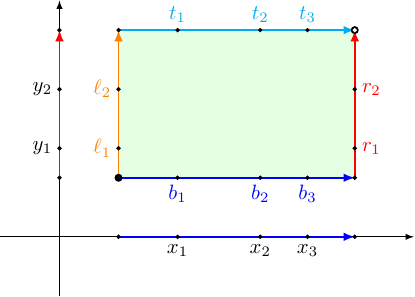}
  \end{tabular}
  \caption{The vertices and sides of $\Qb = \Ib + \Jb i$ are shown on
    the left, including its basepoint $z$ and its breakpoint $w$. 
    The $10$ points of a $(3,2)$-configuration induced by $C
    = \{x_1,x_2,x_3\} \subset I$ and $D = \{y_1, y_2\} \subset J$ are
    shown on the right.}
    \label{fig:rectangle}
\end{figure}

\begin{defn}[Sides and cells]\label{def:side-cell}
  The rectangle $\Qb$ has a natural cell structure.  The \emph{sides}
  of $\Qb$ are the open edges of the \emph{top side} $\Tb = \Ib \times
  \{y_t\}$, \emph{bottom side} $\Bb = \Ib \times \{y_b\}$, 
 \emph{left side} $\Lb = \{x_\ell\} \times \Jb$, and \emph{right side} $\Rb = \{x_r\} \times \Jb$.   
 The vertex $z=(x_\ell,y_b)$ is its 
 \emph{basepoint} and the opposite vertex $w=(x_r,y_t)$ is its \emph{breakpoint}.
  The isometric \emph{side identifications}  $f_\Bb \colon \Bb \to \Ib$, 
  $g_\Rb \colon \Rb \to \Jb$,  $f_\Tb \colon \Tb \to \Ib$, and
  $g_\Lb \colon \Lb \to \Jb$, drop the fixed coordinate. The intervals
  $\Ib$ and $\Jb$ are oriented left-to-right as subsets of $\R$ and
  the sides of $\Qb$ are oriented via their side identifications:
  $\Bb$ and $\Tb$ are oriented left-to-right; $\Lb$ and $\Rb$ are
  oriented bottom-to-top.  In the boundary $\partial \Qb$, $\Lb$ and
  $\Tb$ are oriented clockwise while $\Bb$ and $\Rb$ are oriented
  counterclockwise.
\end{defn}

Cell structures on $\Ib$ and $\Jb$ produce a product cell structure on
$\Qb = \Ib \times \Jb$.

\begin{defn}[Rectangles and subdivisions]\label{def:rect-metric}
  Let $\Qb = \Ib \times \Jb$ be a coordinate rectangle
  (Definition~\ref{def:coord-rectangle}). When $\Ib$ is $k$-subdivided
  and $\Jb$ is $l$-subdivided, we say $\Qb$ has been
  \emph{$(k,l)$-subdivided} and has a \emph{$(k,l)$-structure}.  It
  has $(k+1)(l+1)$ $2$-cells, $(k+1)(l+2)$ horizontal edges,
  $(k+2)(l+1)$ vertical edges, and $(k+2)(l+2)$ vertices $v_{i,j}
  =(x_i,y_j)$.  The $k\cdot l$ new vertices in the interior are those
  with $i \in [k]$ and $j\in [l]$.
  When $\Ib$ and $\Jb$ have been given metric cell structures $\Ib_C$
  and $\Jb_D$, the rectangle $\Qb$ has a metric cell structure
  $\Qb_{C,D}$. Using the side identifications, the sides of
  $\Qb_{C,D}$ have metric cell structures  $\Tb_C$, $\Bb_C$, $\Lb_D$, and $\Rb_D$.
  When $C = \{x_1, x_2, \ldots, x_k\}$ and $D = \{y_1,
  y_2, \ldots, y_l\}$, we write 
  $C_T = \{t_1, \ldots, t_k\}$,  
  $C_B = \{b_1,\ldots, b_k\}$, 
  $D_L = \{\ell_1, \ldots, \ell_l\}$,
  and $D_R = \{r_1,\ldots, r_l\}$, 
  for the interior vertices of 
  $\Tb_C$, $\Bb_C$, $\Lb_D$, and $\Rb_D$, respectively.
  The $2(k+l)$ new vertices in $\partial \Qb_{C,D}$ are a metric
  \emph{$(C,D)$-configuration} and a combinatorial
  \emph{$(k,l)$-configuration}. The right hand side of
  Figure~\ref{fig:rectangle} shows the $10$ points of a
  $(3,2)$-configuration. 
\end{defn}
  
\begin{defn}[Critical value complex $\Qb'_p$]\label{def:coord-cplx}
  Let $p \colon \C \to \C$ be a planar branched cover, and let
  $\closedsquare = [x'_\ell,x'_r]\times [y'_b,y'_t]$ be a coordinate
  rectangle in the range with $\cvl(p) \subset \closedsquare$.  Define
  $E' = \set(\cvl(p))$ and let $C' = \Re(E') \cap I'$ and $D' =
  \Im(E') \cap J'$, the real and imaginary parts of $E'$ in the open
  intervals $I'$ and $J'$, respectively.  Let $\Ib'_p = \Ib'_{C'}$,
  $\Jb'_p = \Jb'_{D'}$ and $\Qb'_p = \Qb'_{C',D'}$
  (Definitions~\ref{def:int-subdivide} and~\ref{def:rect-metric}).
  This is the unique minimal horizontal and vertical subdivision of
  $\closedsquare$ so that $\cvl(p)$ is in the $0$-skeleton. We call $\Qb'_p$ the \emph{critical value complex} on
  $\closedsquare$, whose factors are the \emph{real critical value
  complex} $\Ib'_p$ and the \emph{imaginary critical value complex}
  $\Jb'_p$ of $p$ on $\closedsquare$.  
\end{defn}

\begin{rem}[Cells and sides]\label{rem:cells-sides}
  The cell structure $\Qb'_p$ is almost entirely independent of the choice of $\closedsquare$, but
  changes occur depending on whether or not the sides of $\closedsquare$ are regular.
  For example, if a rectangle $\closedsquare \supset \cvl(p)$ contains a critical value in 
  its right side and we slightly extend $\Ib'$ to the right so that this doesn't happen, then 
  the new extended $\Ib'_p$ has one more subdivision than the old $\Ib'_p$ and the new extended 
  version of $\Qb_p'$ has an extra row of $2$-cells on the right compared to the previous $\Qb'_p$.
  Similar comments apply to the other sides.  Given a branched cover $p$, one could simply choose a rectangle 
  $\closedsquare$ large enough so that its sides are regular, eliminating this dependency, 
  but we need to allow critical values in the boundary of $\closedsquare$ in order to create compact 
  spaces of polynomials.  See Lemma~\ref{lem:closed-bounded} in Part~\ref{part:poly-spaces}.
\end{rem}

\begin{defn}[Regular value complex $\Qb_p$]\label{def:reg-val-cplx}
  Let $\Ib_p$ be a subdivision of an interval
  $\Ib=[x_\ell,x_r]$ and let $\Jb_p$ be a subdivision of an interval $\Jb=[y_b,y_t]$ so that
  $\Ib'_p = (\Ib_p)'$ and $\Jb'_p = (\Jb_p)'$ as cellular duals.  In other words
  $\Ib_p = \int(\Ib'_p)$ is an undual of $\Ib'_p$ and $\Jb_p = \int(\Jb'_p)$ 
  is an undual of $\Jb'_p$. The product of the
  \emph{real regular value complex} $\Ib_p$ and the \emph{imaginary
  regular value complex} $\Jb_p$ is the \emph{regular value complex}
  $\Qb_p = \Ib_p \times \Jb_p$ for $p$ based on $\closedsquare$.  It is based on 
  $\closedsquare$ in the sense of Remark~\ref{rem:cells-sides}, but it is drawn on 
  the larger rectangle $\Qb$ with $\closedsquare = \Qb' \subset \inter(\Qb)$ 
  by construction.  Extending the notation of Definition~\ref{def:dual-intervals}, 
  we write $\Qb_p = \int(\Qb'_p)$ since $\Qb_p$ is an undual of $\Qb'_p$. 
  Note that if $\Qb'_p$ is $(k',l')$-subdivided, then $\Qb_p$ is
  $(k,l)$-subdivided with $k=k'+1$ and $l=l'+1$. 
  The basepoint $z=(x_\ell,y_b)$ and breakpoint $w=(x_r,y_t)$ are marked because they are 
  needed. The basepoint is used when computing fundamental groups (Definition~\ref{def:monodromy}), and
  the breakpoint is used when indexing the sides of a branched rectangle 
  (Definition~\ref{def:monic-poly-labels}).  Note that we mark a basepoint and a breakpoint
  only on the regular value complex $\Qb_p$ where both are regular, and not on the 
  critical value complex $\Qb'_p$ where they may not be regular.
\end{defn}  

\subsection{Branched cellular maps}\label{subsec:br-cell-maps}
A cell complex in the range of a planar branched cover with regular
edges induces a cell structure on its preimage and there is a branched
cellular map between them.  A cellular map between cell complexes is
one that maps each cell in the domain homeomorphically to a cell of
the same dimension in the range.  A branched cellular map allows
branching in $2$-cells.

\begin{defn}[Branched cellular maps]\label{def:br-cellular}
  Let $\pb \colon \Zb \to \Wb$ be a branched covering map between
  surfaces that restricts to a map $\pb \colon \Xb \to \Yb$ from a
  cell complex $\Xb \subset \Zb$ to a cell complex $\Yb \subset \Wb$.
  We say that the restricted map $\pb$ is a \emph{branched cellular
  map} when (1) for each open $i$-cell in $\Xb$ there is an open
  $i$-cell in $\Yb$ that contains its image under $\pb$, (2) the maps
  between $1$-cells are homeomorphisms and (3) the maps between open
  $2$-cells are open planar branched covers
  (Definition~\ref{def:planes}). In particular, every branched
  cellular map looks like a cellular map between cell complexes if you
  only look near the $1$-skeleton.
\end{defn}

\begin{lem}[Disk preimages]\label{lem:2-cells}
  Let $\pb \colon \Zb \to \Wb$ be a planar branched cover and let $\Yb
  \subset \Wb$ be a cell complex.  For every open $2$-cell $F$ in
  $\Yb$, the preimage $\pb^{-1}(F)$ is a disjoint union of open
  $2$-cells and the component maps are planar branched covers.
\end{lem}

\begin{proof}
  Let $\Wb_F \subset F$ be a closed disk such that all of the
  critical values of $p$ in $F$ are in the interior of $\Wb_F$.  By
  Proposition~\ref{prop:disk-pre}, its preimage $\Zb_F =
  \pb^{-1}(\Wb_F)$ is a union of disks, and the preimage of the open
  annulus region $F \setminus \Wb_F$, being regular, is a union of
  open annuli which provide an annular padding around the components
  of $\Zb_F$, showing that the components of $\pb^{-1}(F)$ are open
  disks.
\end{proof}

Let $\Yb \subset \Wb$ be a cell complex in the range of a branched
cover.  We say that $\Yb$ is \emph{$1$-regular} if the open $1$-cells
of $\Yb$ are regular.  Note that the regular value complex $\Qb_p$ and
the critical value complex $\Qb'_p$ are both $1$-regular by
construction: in the regular case because the critical values lie in
$2$-cells and in the critical case because the critical values are
$0$-cells.  For $1$-regular cell complexes in the range of a branched
cover, the preimage has a cell structure.

\begin{defn}[Induced cell structures]\label{def:induced-cell}
  Let $\pb \colon \Zb \to \Wb$ be a planar branched cover. If $\Yb
  \subset \Wb$ is a $1$-regular cell complex, then its preimage $\Xb =
  \qb^{-1}(\Yb)$ has an \emph{induced cell structure from $\Yb$ via
  $\qb$} and the restricted map $\pb \colon \Xb \to \Yb$ is a branched
  cellular map. The \emph{$i$-regions} of $\Xb$ are the connected
  components of the preimages of the open $i$-cells of $\Yb$ with the
  obvious attaching maps.  The $1$-regions in $\Xb$ are open $1$-cells
  (since the open $1$-cells are evenly covered), and the $2$-regions
  are open $2$-cells by Lemma~\ref{lem:2-cells}. The assumption that
  $\pb$ is a \emph{planar} branched cover is crucial here, since
  Example~\ref{ex:non-disks} gives an example where a preimage of a
  closed disk is a closed surface with multiple boundary
  cycles. Finally, if $\Yb$ has a metric and $\Xb$ has the induced
  metric, then $\pb$ is a local isometry away from the critical points
  in $\Xb$.
\end{defn}

For our purposes, the most important cell complexes are disk diagrams.

\begin{defn}[Disk diagrams]\label{def:disk-diagrams}
  A \emph{disk diagram} $\Yb$ is a compact contractible cell complex
  embedded in a surface $\Wb$.  If $\Yb$ is homeomorphic to a closed
  disk, it is \emph{nonsingular}, otherwise it is \emph{singular}.
  For disk diagrams, being singular is equivalent to having a cut
  point, or even a local cut point, whose removal (locally)
  disconnects $\Yb$.  Disk diagrams can also be characterized as cell
  complexes embedded in a surface where there are arbitrarily small
  neighborhoods that are topological disks.  When the surface $\Wb$
  containing the (singular or nonsingular) disk diagram $\Yb$ is
  itself a disk, $\Wb$ deformation retracts to $\Yb$.
\end{defn}

\begin{rem}[Components and critical complexes]\label{rem:comp-crit-cplx}
  Let $\pb \colon \Zb \to \Wb$ be a planar branched cover, let $\Yb
  \subset \Wb$ be a disk diagram and let $\Wb^\Yb$ be a small closed
  disk neighborhood of $\Yb$ so that any critical values of $\pb$ in
  $\Wb^\Yb$ are in $\Yb$.  By Proposition~\ref{prop:disk-pre}, the
  preimage $\Zb^\Yb = \pb^{-1}(\Wb^\Yb)$ is a union of disks and the
  component maps $\pb^\Yb_i \colon \Zb_i^\Yb \to \Wb^\Yb$ are planar
  branched covers.  In particular, working component-by-component,
  restricting $\Zb$ to $\Zb^\Yb_i$, $\Wb$ to $\Wb^\Yb$, and $\pb$ to
  $\pb_i^\Yb$, reduces the general case to a set of special cases in
  with $\cvl(\qb_i^\Yb) \subset \Yb$.
\end{rem}

\begin{cor}[Disk diagrams]\label{cor:disk-diagrams}
  If $\pb \colon \Zb \to \Wb$ is a planar branched cover and $\Yb
  \subset \Wb$ is a $1$-regular disk diagram, then $\Xb =
  \pb^{-1}(\Yb)$ is a disjoint union of disk diagrams.  In particular,
  if $\cvl(\pb) \subset \Yb$, then $\Xb$ itself is a disk diagram.
\end{cor}

\begin{proof}
  By Remark~\ref{rem:comp-crit-cplx} it is sufficient to prove this
  when $\cvl(\pb) \subset \Yb$ and $\Zb$ is a disk. In this case, 
  $\Xb$ is a cell complex and the map $p \colon \Xb \to \Yb$ is a branched 
  cellular map (Definition~\ref{def:br-cellular}).  Moreover, the deformation
  retraction from $\Wb$ to $\Yb$ (Definition~\ref{def:disk-diagrams})
  lifts to a deformation retraction from $\Zb$ to $\Xb$ (since the
  complement of $\Yb$ is covered by the complement of $\Xb$).  Since
  $\Xb$ and $\Zb$ are homotopy equivalent, $\Xb$ is contractible.
\end{proof}

It is easy to determine whether or not the preimage disk diagram is
singular.

\begin{cor}[Singular diagrams]\label{cor:singular}
  Let $\pb \colon \Zb \to \Wb$ be a planar branched cover that
  restricts to a map $\Xb \to \Yb$ between disk diagrams.  When $\Yb$
  is singular, $\Xb$ is singular.  And when $\Yb$ is nonsingular,
  $\Xb$ is singular if and only if $\partial \Xb$ contains a critical
  point of $p$.
\end{cor}

\begin{proof}
  Suppose $\Yb$ is singular and let $v$ be a local cut vertex in
  $\partial \Yb$ (Definition~\ref{def:disk-diagrams}).  Any $u$ in
  $\partial \Xb$ sent to $v$ is also a local cut vertex, regardless of
  whether the local model is branched, so $\Xb$ is also singular.
  Next, suppose $\Yb$ is nonsingular and there are no critical points
  in the boundary of $\Xb$.  Then the local models are homeomorphisms,
  there are no local cut points in the boundary of $\Xb$, so $\Xb$ is also
  nonsingular.
  Finally, suppose $\Yb$ is nonsingular and there is a critical point
  $u$ in the boundary of $\Xb$.  Its image $v$ is a critical value in
  the boundary of $\Yb$, the branched local model near $u$ shows that
  $u$ is a local cut vertex, and $\Xb$ is singular.
\end{proof}

\begin{example}[Branched lines and banyans]\label{ex:banyans}
  Let $\qb \colon \Zb \to \Wb$ be a planar branched cover, and let
  $\Yb \subset \Wb$ be a closed interval.  If $\Yb$ is subdivided so
  that any critical values in $\Yb$ are vertices, then its preimage
  $\Xb = \qb^{-1}(\Yb)$ is a $1$-complex
  (Definition~\ref{def:induced-cell}) with components $\Xb = \Xb_1
  \sqcup \cdots \sqcup \Xb_k$.  The $1$-complexes $\Xb_i$ are trees
  because they are contractible (Corollary~\ref{cor:disk-diagrams}).
  We call them \emph{branched lines} or \emph{banyan trees}, a type of
  tree with multiple branches and multiple roots.  The cell complex
  $\Xb$ is a \emph{banyan grove}, a type of forest.
\end{example}

\begin{example}[Branched disks and cacti]\label{ex:cacti}
  Let $\qb \colon \Zb \to \Wb$ be a planar branched cover, and let
  $\Yb \subset \Wb$ be a closed disk.  If $\Yb$ is given the cell
  structure of a nonsingular $2$-complex with only one $2$-cell and a
  subdivided boundary $\partial \Yb$ with regular edges, then its
  preimage $\Xb = \qb^{-1}(\Yb)$ is a $2$-complex
  (Definition~\ref{def:induced-cell}) with components $\Xb = \Xb_1
  \sqcup \cdots \sqcup \Xb_k$.  The $2$-dimensional cell complexes
  $\Xb_i$ are disk diagrams (Corollary~\ref{cor:disk-diagrams}) called
  \emph{branched disks} or \emph{cactus diagrams}, since they resemble
  prickly pear cacti.  The cell complex $\Xb$ is a \emph{garden of
  cactus diagrams}.  The components $\partial \Xb_i$ of $\partial \Xb$
  are the boundaries of the components $\Xb_i$ of $\Xb$.  The
  component $\partial \Xb_i$ is a \emph{branched circle} or
  \emph{cactus graph}, and the full $1$-skeleton $\partial \Xb$ is a
  \emph{garden of cactus graphs}.
\end{example}

Metric banyans and metric cacti appeared in the first paper in
this series \cite{DoMc22}, but the definitions given here are more
concise.

\subsection{Complexes in the domain}\label{subsec:cplx-domain}
We now define the regular point complex $\Pb_p$ and the critical point
complex $\Pb'_p$ using Definition~\ref{def:induced-cell}.

\begin{figure}
  \centering   
  \includegraphics[scale=.8]{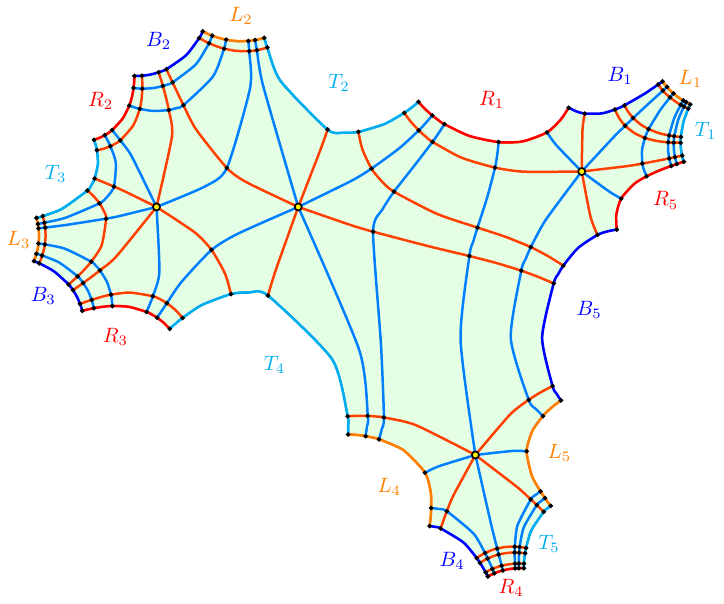}
  \caption{The domain of the polynomial $p$ of Example~\ref{ex:deg5}
    with its $4$ critical points marked as yellow dots.  They are
    shown inside the critical point complex $\Pb'_p$, which is the
    preimage of the critical value complex $\Qb'_p$ shown in
    Figure~\ref{fig:deg5-original-range}.
  \label{fig:deg5-original-domain}}
\end{figure}

\begin{example}[Point complexes $\Pb'_p$ and $\Pb_p$]\label{ex:pt-cplx}
  Let $p$ be the polynomial of Example~\ref{ex:deg5}, let $\Qb'_p$
  be the critical value complex in the range of $p$ shown in
  Figure~\ref{fig:deg5-original-range}, and let $\Qb_p$ be the regular 
  value complex in the range of $p$ shown in Figure~\ref{fig:deg5-original-dual-range}. 
  The \emph{critical point complex} $\Pb'_p$ in the domain of $p$ is shown in
  Figure~\ref{fig:deg5-original-domain}.  The pullback metric on
  $\Pb'_p$ makes each of the $125$ topological $2$-cells into a
  Euclidean rectangle with a metric determined by its image in
  $\Qb'_p$.  The disk diagram $\Pb'_p$ is nonsingular, in this case,
  because $\partial \Qb'_p$ is regular (Corollary~\ref{cor:singular}).
  The \emph{regular point complex} $\Pb_p$ in the domain of $p$ is shown in
  Figure~\ref{fig:deg5-original-dual-domain}. Since the branching
  occurs in the interior of $2$-cells, the $1$-skeleton of $\Pb_p$ is
  a $5$-sheeted cover of the $1$-skeleton of $\Qb_p$, and it has a $5$-branched 
  $(5,5)^5$-structure.  The side labels are explained in Section~\ref{sec:geo-comb}.
\end{example}

\begin{figure}
  \centering
  \includegraphics[scale=.8]{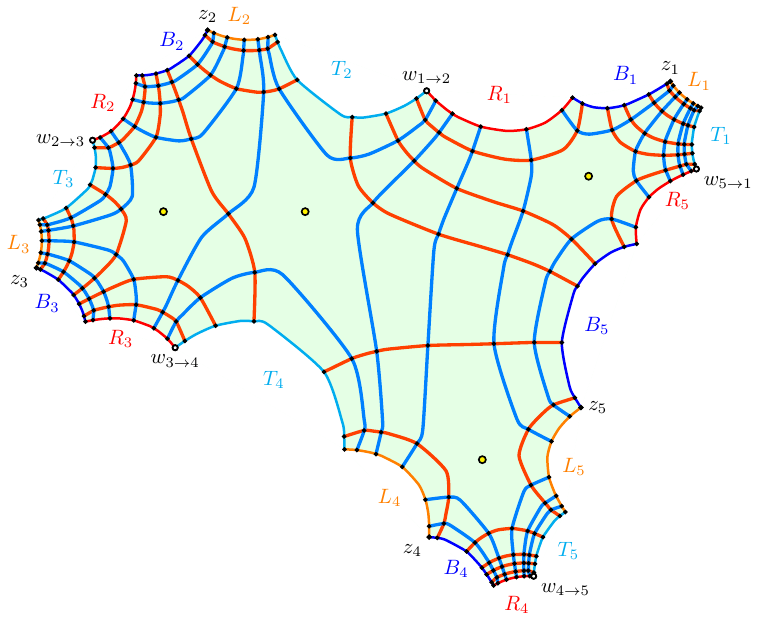}
  \caption{The domain of the polynomial $p$ of Example~\ref{ex:deg5}
    with its $4$ critical points marked as yellow dots.  They are
    shown inside the regular point complex $\Pb_p$, which is the
    preimage of the regular value complex $\Qb_p$ shown in
    Figure~\ref{fig:deg5-original-range}.  The breakpoint preimages 
    $w_{m\to m+1}$ and basepoint preimages $z_i$ are also marked.
 \label{fig:deg5-original-dual-domain}}
\end{figure}

\begin{defn}[Point complexes $\Pb_p$ and $\Pb'_p$]\label{def:br-coord-cplx}
  Let $p \colon \C \to \C$ be a planar branched cover, and let
  $\closedsquare = [x'_\ell,x'_r]\times [y'_b,y'_t]$ be a coordinate
  rectangle in the range with $\cvl(p) \subset \closedsquare$.  The
  coordinate complexes $\Qb_p$ and $\Qb'_p$ in the range are both
  $1$-regular, and by Definition~\ref{def:induced-cell} their
  preimages under $p$ have induced cell structures.  We call these
  cell complexes in the domain the \emph{regular point complex
  $\Pb_p$} and the \emph{critical point complex $\Pb'_p$}.  The map
  $p$ restricts to  the \emph{regular complex map} $\Pb_p \to \Qb_p$ and
  the \emph{critical complex map} $\Pb'_p \to \Qb'_p$.  Note that the cell 
  structure on $\Qb'_p$ and the critical complex map $\Pb'_p \to \Qb'_p$
  can be reconstructed simply from the cell structure of $\Pb'_p$.
  The regular complex map is a branched cellular map, and the the critical 
  complex map is a cellular map since the rectangular $2$-cells of the critical value complex $\Qb_p$ are 
  regular by construction. The regular complex map is a $d$-sheeted covering map
  between their $1$-skeletons because the $1$-skeleton of $\Qb_p$ is
  regular.  If $\Qb_p$ has a $(k,l)$-structure (Definition~\ref{def:rect-metric}), 
  then we say that $\Pb_p$ has a $d$-branched \emph{$(k,l)^d$-structure}.
\end{defn}

\begin{rem}[Unique geodesics]\label{rem:unique-geod}
  The critical map from $\Pb'_p$ to $\Qb'_p$ is a cellular map between
  piecewise Euclidean complexes built out of Euclidean rectangles.
  The complex $\Qb'_p$, as a subdivided rectangle, is obviously a
  $\cat(0)$ metric space, and so is $\Pb'_p$ since every interior
  vertex has an integer multiple of $2\pi$ in angle.  In particular,
  both spaces have unique geodesics \cite{bridson-haefliger}. 
\end{rem}

\begin{rem}[Combinatorial invariance]\label{rem:comb-inv}
  Technically speaking, the construction of the regular point complex
  $\Pb_p$ depends on the regular value complex $\Qb_p$ which in turn
  depends on a choice of subdivided intervals $\Ib_p$ and $\Jb_p$
  that have $\Ib'_p$ and $\Jb'_p$ as cellular duals, but the
  combinatorial structures of $\Ib_p$, $\Jb_p$, $\Qb_p$ and $\Pb_p$ are
  independent of these choices.  In particular, continously varying the choice of vertices 
  of $\Ib_p$ and $\Jb_p$, continuously varies the $1$-skeleton of $\Qb_p$
  and continuously varies the $1$-skeleton of $\Pb_p$ (since this is taking place in the 
  regular portion of the range) without changing any of their cell structures.  It does, of course, 
  still depend on the choice of rectangle $\closedsquare$ since this changes the cell 
  structure of the critical value complex $\Qb'_p$ (Remark~\ref{rem:cells-sides}).
\end{rem}

\section{Noncrossing Combinatorics}\label{sec:nccomb}

The combinatorial structure of the regular point complex $\Pb_p$ can
be described using planar noncrossing combinatorics. This section
establishes conventions for noncrossing partitions
(\ref{subsec:ncpart}), drawn on planar branched rectangles
(\ref{subsec:brrect}), and it connects them to noncrossing matchings
(\ref{subsec:ncmatch}) and noncrossing permutations
(\ref{subsec:ncperm}).  Since these structures are associated to
planar $d$-branched covers and degree-$d$ polynomials, we use $[d]$ as
our indexing set.

\subsection{Noncrossing partitions}\label{subsec:ncpart}
Noncrossing partitions can be defined combinatorially, metrically, and
topologically.  The combinatorial version is the simplest, the metric
version explains the name, and the topological version, introduced
here, is the most convenient to use in our context.  We begin with
some basic notation for points and subsets of the complex plane.

\begin{defn}[Points and subsets]\label{def:pts-subsets}
  Recall that $\D$ is the closed unit disk and $\T = \partial \D$ is
  its unit circle boundary. Let $e \colon \C \to \C_\zer$ be the
  function $e(z) = \exp(2\pi i z)$ that sends $\R$ to $\T$ with kernel
  $\Z$.  The $d^{th}$ roots of unity $\sqrt[d]{1} = \{e(j/d) \mid j \in
  [d]\}$ are indexed by $j \in [d]$, viewed as residue classes of the
  integers mod $d$.  More generally, every $d$-element subset of $\T$
  can be identified with the integers mod $d$ by sending $1$ to the
  element with smallest positive argument, and proceeding in a
  counterclockwise fashion around the circle.
\end{defn}

\begin{defn}[Combinatorial noncrossing partitions]\label{def:nc-comb}
  A set partition $[\lambda] \vdash [d]$ is a (combinatorial)
  \emph{noncrossing partition} if whenever there is a $4$-element
  subset $i < j < k < \ell$ in $[d]$ with $i$ and $k$ in the same
  block, and $j$ and $\ell$ in the same block, then all four are in
  the same block. The collection $\ncpart_d$ of all noncrossing
  partitions of $[d]$, ordered by refinement, is an induced subposet
  of $\spart_d$.
\end{defn}

\begin{defn}[Metric noncrossing partitions]\label{def:nc-metric}
  Let $\sqrt[d]{1} \subset \C$ be the $d^{th}$ roots of unity with $j\in
  [d]$ labeling $e(j/d)$ (Definition~\ref{def:pts-subsets}).  A
  set partition $[\lambda] \vdash [d]$ is a (metric) \emph{noncrossing
  partition}, if the convex hulls of the $d^{th}$ roots labeled by the
  numbers in each block of $[\lambda]$ form pairwise disjoint
  subspaces of~$\C$.
\end{defn}

\begin{defn}[Topological noncrossing partitions]\label{def:nc-top}
  Let $\Xb$ be a closed disk in the complex plane, let $S
  \subset \partial \Xb$ be a subset of its boundary with $d$ path
  components and fix a bijective labeling of the components of $S$ by
  the numbers in $[d]$ in the counterclockwise order they occur in the
  boundary of $\Xb$.  For every subspace $\Ub$ (typically a closed
  subsurface) with $S \subset \Ub \subset \Xb$ we define a set
  partition $\ncpart(\Ub)$ where $i$ and $j$ are in the same block of
  $\ncpart(\Ub)$ if and only if the path components of $S$ labeled $i$
  and $j$ are in the same path component of $\Ub$.  A set partition
  $[\lambda] \vdash [d]$ is a (topological) \emph{noncrossing
  partition} if there exists a subspace $\Ub$ with $S \subset \Ub
  \subset \Xb$ such that $\ncpart(\Ub) = [\lambda]$.  Note that if we
  view the subspaces of $\Xb$ containing $S$ as a poset under
  inclusion, then the map to $\ncpart_d$ is weakly order-preserving.
  In other words, if $S \subset \Ub_1 \subset \Ub_2 \subset \Xb$ as
  subspaces, then $\ncpart(\Ub_1) \leq \ncpart(\Ub_2)$ as set partitions.
\end{defn}

The three definitions are, of course, equivalent.

\begin{prop}[Noncrossing partitions]\label{prop:ncpart}
  The combinatorial, metric and topological definitions of noncrossing
  partitions define the same collection of set partitions.
\end{prop}

\begin{proof}
  (C $\Leftrightarrow$ M) The equivalence of the combinatorial and
  metric definitions is classical \cite{kreweras72}. (M $\Rightarrow$
  T) Given $\Xb$ and $S$ as in Definition~\ref{def:nc-top}, fix an
  identification of $\Xb$ with closed unit disk $\D$ so that the path
  component of $S$ labeled $j$ contains $e(j/d) \in \sqrt[d]{1}$.  For
  each metric noncrossing partition $[\lambda]$, the union of its
  convex hulls (union $S$), viewed as a subspace $\Ub \subset \Xb =
  \D$, shows that $[\lambda]$ is a topological noncrossing
  partition. If we need $\Ub$ to be a subsurface, simply replace the
  subspace with a small closed neighborhood in $\D$. (T $\Rightarrow$
  C) Let $[\lambda]$ be a topological noncrossing partition in a disk
  $\Xb$, let $\Ub$ be a subspace with $\ncpart(\Ub) = [\lambda]$ and
  let $i < j < k < \ell$ be numbers in $[n]$ with $i$ and $k$ in the
  same block and $j$ and $\ell$ in the same block.  If $\alpha\colon
  \Ib \to \Ub$ is a path in $\Ub$ from the $i^{th}$ component of $S$ 
  to the $k^{th}$ component of $S$, and $\beta\colon \Ib \to \Ub$ is 
  a path in $\Ub$ from the $j^{th}$ component of $S$ to the $\ell^{th}$ 
  component of $S$, then $\alpha$ and $\beta$ intersect.\footnote{If not, then
  $(s,t) \mapsto (\beta(t) - \alpha(s))$ is a contractible map $\Ib^2
  \to \C^\ast$ whose boundary cycle is sent to a loop of winding
  number $1$, contradiction. The winding number can be computed by
  assuming $\Xb$ is a rectangle and $\alpha$ and $\beta$ connect
  opposite corners.  The map in this case sends the vertices of
  $\Ib^2$ to the coordinate axes and its sides to paths in specific
  quadrants.}  Thus all four points are in the same path component of
  $\Ub$, the numbers are in the same block of $\ncpart(\Ub)=
      [\lambda]$, and $[\lambda]$ is a combinatorial noncrossing
      partition.
\end{proof}

\subsection{Branched rectangles}\label{subsec:brrect}
The topological noncrossing partitions of interest here are induced by
subspaces of planar branched covers of rectangles.  This subsection
establishes our conventions for such spaces.  We first specify a
standard rectangle.

\begin{defn}[Standard rectangle]\label{def:std-rectangle}
  The \emph{standard rectangle} is $\Qb = \Ib +\Jb i \subset \C$ where
  both $\Ib$ and $\Jb$ are the \emph{standard interval} $[-1,1]$.  The
  sides of $\Qb$ touch the unit circle at their midpoints $\sqrt[4]{1}$ with
  $i \in T$, $-1 \in L$, $-i \in B$ and $1 \in R$, and these are
  their \emph{representatives}.  The vertices of
  $\Qb$ are $\{\pm (1 \pm i)\} = \sqrt[4]{-4}$.  The breakpoint is $w=1+i$ and the basepoint is $z=-(1+i)$.
\end{defn}

Similar conventions apply to the power map preimages of the standard
rectangle.

\begin{defn}[Standard branched rectangles]\label{def:std-brrectangles}
  Let $p(z) = z^d$ be the power map (Definition~\ref{def:degenerate})
  and let $\Qb$ is the standard rectangle
  (Definition~\ref{def:std-rectangle}).  The \emph{standard
  $d$-branched rectangle} is $\Pb = p^{-1}(\Qb)$. The $d$-fold
  covering map $\partial \pb \colon \partial \Pb \to \partial \Qb$
  between their boundaries is used to partition the $4d$ open edges, or \emph{sides}, 
  $S$, into $d$ \emph{top sides
  $S_T$} , $d$ \emph{left sides $S_L$}, $d$ \emph{bottom sides $S_B$},
  and $d$ \emph{right sides $S_R$}. For example $S_T = p^{-1}(T)$.
  The midpoints of the sides of $\Pb$ touch the unit circle
  at the $4d$ points $\sqrt[4d]{1}$, and these are their \emph{standard
  representatives}.  More precisely, the top sides are
  represented by the points $\sqrt[d]{i}$, the left sides by
  the points $\sqrt[d]{-1}$, the bottom sides by the
  points $\sqrt[d]{-i}$ and right sides are by the points
  $\sqrt[d]{1}$.  Sides are indexed by the indexing of their representatives 
  (Definition~\ref{def:pts-subsets}).  For example, $T_m$ 
  contains $e(\frac{4m-3}{4d})$, the $m^{th}$ element of $\sqrt[d]{i} = \{e(\frac{4m-3}{4d}) \mid m \in [d]\}$.  
  The  \emph{basepoint preimage} between sides $L_m$ and $B_m$
  is $z_m$, and the \emph{breakpoint preimage} between sides $R_m$ and $T_{m+1}$
  is $w_{m\to m+1}$.  See Figure~\ref{fig:br-rectangle}.
\end{defn}

\begin{figure}
    \centering
    \includegraphics{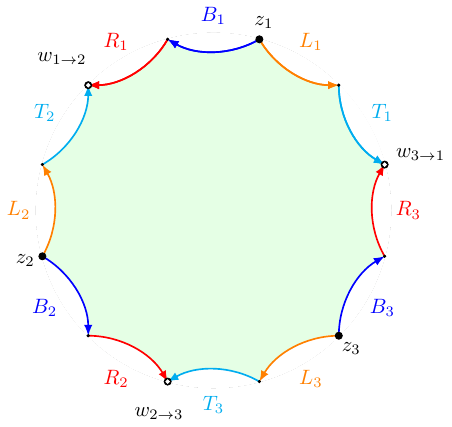}
    \caption{A standard $3$-branched rectangle with labeled cells.  The $3$ large white dots
    are the indexed breakpoints $w_{m\to m+1}$ and the $3$ large black dots are the
    indexed basepoints $z_i$.
    \label{fig:br-rectangle}}
\end{figure}

The standard $4d$-gon $\Pb$ roughly looks like a regular hyperbolic
polygon, but more precisely, the sides are portions of the curves $r^d
\cos(d \theta) = \pm 1$ (for the preimages of the right and left
sides) and $r^d \sin(d \theta) = \pm 1$ (for the preimages of the top
and bottom sides).  Nevertheless, we often treat $\Pb$ as though it
was in the Poincar\'e disk model with a hyperbolic metric. Arcs and
convex hulls are drawn with this hyperbolic approximation in mind. See
Figure~\ref{fig:br-rectangle} for an example when $d=3$.

The standard branched rectangle has side-based noncrossing partitions.

\begin{defn}[Noncrossing partitions of sides]\label{def:ncpart-sides}
  Let $\ncpart_d^T$ be the (topological) noncrossing partitions of the
  top sides $S_T$ in $\Pb$, with $\ncpart_d^L$, $\ncpart_d^B$, and
  $\ncpart_d^R$ defined similarly.  There are canonical isomorphisms
  $\ncpart_d \cong \ncpart_d^T \cong \ncpart_d^L \cong \ncpart_d^B
  \cong \ncpart_d^R$ based on the indexing of the sides, and we write
  $[\lambda]^T$, $[\lambda]^L$, $[\lambda]^B$, $[\lambda]^R$, for the
  image of a noncrossing partition $[\lambda] \in \ncpart_d$ in these
  other copies.
\end{defn}

Since the identifications in Definition~\ref{def:ncpart-sides} are
based on subscripting conventions, they necessarily break certain
symmetries.

\begin{rem}[Broken symmetries]\label{rem:broken}
  Iteratively rotating $\Pb$ through a counterclockwise angle of
  $2\pi/4d$ sends a subsurface defining $[\lambda]^T$ to a subsurface
  defining $[\lambda]^L$, then to a subsurface defining $[\lambda]^B$,
  then to a subsurface defining $[\lambda]^R$, then to a subsurface
  defining $[\lambda^+]^T$.  The plus indicates that the
  (combinatorial) noncrossing partition $[\lambda] \vdash [d]$ stays
  the same until the last step, when $[\lambda^+]$ is $[\lambda]$ with
  every number increased by one (mod $d$).  In
  Figure~\ref{fig:br-rectangle}, for example, a component of a
  subsurface containing $T_2$ is sent to one containing $L_2$, then
  $B_2$, then $R_2$, then $T_3$.
\end{rem}

The map $p(z)=z^d$ lifts subsets of the sides of $\Qb$ to subsets of
the sides of $\Pb$.

\begin{defn}[Points in branched sides]\label{def:pts-in-brsides}
  Given subsets $C \subset I$ and $D \subset J$, we define
  corresponding subsets of the sides of $\Pb$ using the map $p(z) =
  z^d$.  For example, $S_{T,C} = p^{-1}(C_T) = \bigcup_{m \in [d]}
  C_{T_m}$ where $C_{T_m} = p^{-1}(C_T) \cap T_m$ and $C_T$ is given
  as in Definition~\ref{def:rect-metric}. When $C$ and $D$ are finite,
  we index the points in the sides of $\Pb$ by indicating the side in
  the first subscript and the point in the second subscript.  For
  example, we write $t_{m,i}$ for the point in $C_{T_m}$ with
  $p(t_{m,i}) = t_i \in C_T$, where $f_T(t_i) =x_i \in C$.  If
  $\size{C} = k$ and $\size{D} = l$, then there are $2d(k+l)$ lifted
  points in the sides $S$, with $\size{S_{T,C}} = \size{S_{B,C}} = dk$
  and $\size{S_{L,D}} = \size{S_{R,D}} = dl$. We call this a
  \emph{$(k,l)^d$-configuration of points} in the boundary of the
  $d$-branched rectangle $\Pb$.
\end{defn}

If we select a point in $I$ and a point in $J$, these lift to a point
in each side of~$\Pb$.

\begin{defn}[Representative points]\label{def:rep-pts}
  A point $x\in I$ with $C = \{x\}$ determines a
  $(1,0)^d$-configuration of $2d$ points, one in each of the $d$ top
  sides and $d$ bottom sides.  In particular, there is $t_{m,x} \in
  T_m$ and $b_{m,x} \in B_m$.
  Similarly, a point $y \in J$ with $D = \{y\}$ determines a
  $(0,1)^d$-configuration of $2d$ points, one in each of the $d$ left
  sides and $d$ right sides.  In particular, there is $\ell_{m,y} \in
  L_m$ and $r_{m,y} \in R_m$.
  The combination of these two is a $(1,1)^d$-configuration with $4d$
  points, one in each side.  In all three cases, we say these points
  \emph{represent} the corresponding side.
\end{defn}

\subsection{Noncrossing matchings}\label{subsec:ncmatch}
A noncrossing matching is both a special type of noncrossing partition
on $[2d]$ points and an encoding of a noncrossing partition on $[d]$
points.  Both viewpoints can be illustrated on $d$-branched
rectangles.

\begin{defn}[Noncrossing matchings]\label{def:ncmatch}
  A (combinatorial) \emph{noncrossing matching} $[\mu]$ is a
  noncrossing partition of the set $[2d]$, necessarily of even size,
  where every block has size~$2$. In other words, $\shape([\mu]) =
  2^d$.  Note that the numbers in a block are forced to have opposite
  parity since they are separated by an even number of ends of
  noncrossing arcs.  In particular, a (metric) noncrossing matching
  can be viewed as a noncrossing bijection between the $d$ points in
  $\sqrt[d]{1}$ with even labels and the $d$ points in $\sqrt[d]{-1}$
  with odd labels.  A bijection between two sets is sometimes called a
  \emph{matching}, hence the name.  Let $\ncmatch_{2d}$ denote the set
  of all noncrossing matchings of $[2d]$.
\end{defn}

There is also a topological definition.

\begin{defn}[Side matchings]\label{def:side-matchings}
  The $d$-branched rectangle $\Pb$ contains $2d$ top and bottom sides,
  $T_1, B_1, T_2, B_2, \ldots, T_d, B_d$ in counterclockwise order,
  starting at the first breakpoint $w_{d\to1}$, containing the $2d$ points $\sqrt[2d]{1}$ as
  their midpoints. A (topological noncrossing) \emph{top-bottom
  matching} is a topological noncrossing partition of $S_T \cup S_B$
  where every block contains exactly $2$ sides. Let
  $\ncmatch_{2d}^{TB}$ denote the collection of all such top-bottom
  matchings.  The bijection $S_T \cup S_B \to [2d]$ sending $T_m$ to
  $2m-1$ and $B_m$ to $2m$ extends to a bijection $\ncmatch_{2d}^{TB}
  \to \ncmatch_{2d}$. Alternatively, a top-bottom matching can be
  viewed as a noncrossing matching of the $2d$ points in a
  $(1,0)^d$-configuration (Definition~\ref{def:rep-pts}).
  Similarly, $\Pb$ contains $2d$ left and right sides,
  $L_1, R_1, L_2, R_2, \ldots, L_d, R_d$ in counterclockwise order,
  starting at the first breakpoint $w_{d\to1}$. A (topological noncrossing) \emph{left-right
  matching} is a topological noncrossing partition of $S_L \cup S_R$
  where every block contains exactly $2$ sides. Let
  $\ncmatch_{2d}^{LR}$ denote the collection of all such left-right
  matchings.  The bijection $S_L \cup S_R \to [2d]$ sending $L_m$ to
  $2m-1$ and $R_m$ to $2m$ extends to a bijection $\ncmatch_{2d}^{LR}
  \to \ncmatch_{2d}$.   Alternatively, a left-right
  matching can be viewed as a noncrossing matching of the $2d$ points in a 
  $(0,1)^d$-configuration (Definition~\ref{def:rep-pts}).
  We write $[\mu]^{LR}$ and $[\mu]^{TB}$ for the
  left-right and top-bottom matchings that correspond to $[\mu] \in
  \ncmatch_{2d}$.
\end{defn}

\begin{figure}
  \bt{cc}
    \includegraphics[width=5.9cm]{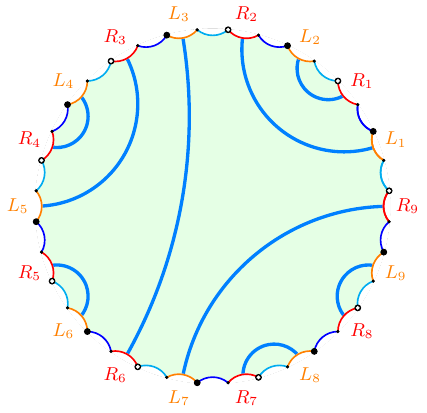} &
    \includegraphics[width=5.9cm]{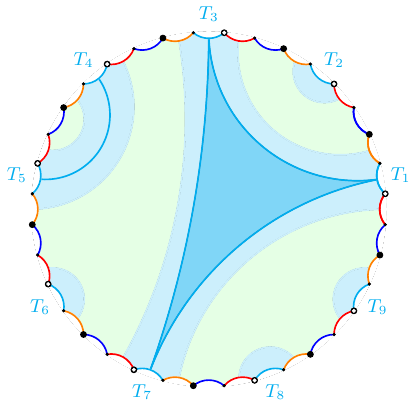}\\
    \includegraphics[width=5.9cm]{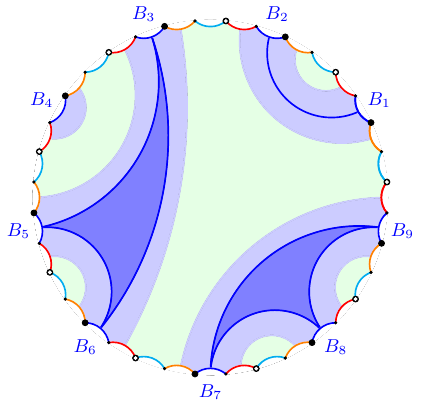} &
    \includegraphics[width=5.9cm]{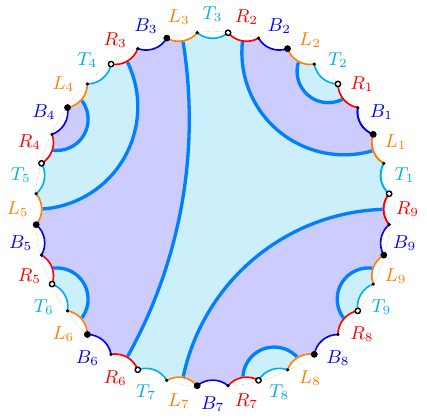}\\
  \et
  \caption{The noncrossing left-right matching $[\mu]^{LR}$ (upper
    left), the noncrossing top partition $[\lambda]^T =
    137|2|45|6|8|9$ (upper right), the noncrossing bottom partition
    $[\lambda]^B = 12|356|4|789$ (lower left), and the $2$-coloring
    (lower right), all contain the same combinatorial information. The
    partition figures include the convex hulls of representative
    points to make them easier to see.
    \label{fig:ncpart-ncmatch}}
\end{figure}

There are bijections between $\ncmatch_{2d}^{LR}$, $\ncpart_d^T$ and
$\ncpart_d^B$, and between $\ncmatch_{2d}^{TB}$, $\ncpart_d^L$ and
$\ncpart_d^R$.  We begin with an example.

\begin{exmp}\label{ex:ncpart-ncmatch}
  Let $\Pb$ be a standard $9$-branched rectangle and let $[\mu]^{LR}$
  be a left-right matching in $\ncmatch_{18}^{LR}$ with blocks
  $\{L_1,R_2\}$, $\{L_2,R_1\}$, $\{L_3,R_6\}$, $\{L_4,R_4\}$,
  $\{L_5,R_3\}$, $\{L_6,R_5\}$, $\{L_7,R_9\}$, $\{L_8,R_7\}$, and
  $\{L_9,R_8\}$.  The upper left corner of
  Figure~\ref{fig:ncpart-ncmatch} shows the multiarc whose arcs are convex 
  hulls of the representative points in a $(0,1)^9$-configuration (e.g. the
  ``hyperbolic'' arc from $\ell_{1,y}$ to $r_{2,y}$). The complementary
  regions of $\Pb$ with these arcs removed can be $2$-colored: a
  lighter sky blue if it contains a top side and a darker sea blue if
  it contains a bottom side (lower right).  The sky blue subsurface
  $\Ub_T$ contains $S_T$ and the corresponding (topological)
  noncrossing partition of top sides is $[\lambda]^T = 137|2|45|6|8|9
  \in \ncpart_9^T$ (upper right).  The sea blue subsurface $\Ub_B$
  contains $S_B$ and the corresponding (topological) noncrossing
  partition of bottom sides is $[\lambda]^B = 12|356|4|789 \in
  \ncpart_9^B$ (lower left). The convex hulls of representative points
  in each block (e.g. the ``hyperbolic'' triangle connecting $b_{3,x}$,
  $b_{5,x}$, and $b_{6,x}$), are also included.  This process is also
  reversible in the sense that the original left-right matching can be
  recovered from either the top or the bottom noncrossing partition.
  Given $[\lambda]^T$, for example, we can define $\Ub_T$ to be an
  $\epsilon$-neighborhood of convex hulls of the representative points
  in each block union the sides $S_T$.  The $9$ arcs of $\partial
  \Ub_T$ in the interior of $\Pb$ are a multiarc that defines the 
  left-right matching~$[\mu]^{LR}$.
\end{exmp}

Arguing as in Example~\ref{ex:ncpart-ncmatch} establishes the following result.

\begin{prop}[Matchings and partitions]\label{prop:match-part}
  There are bijections 
  \[
    \ncpart_d^T \leftrightarrow \ncmatch_{2d}^{LR} \leftrightarrow \ncpart_d^B
  \]
  and
  \[
    \ncpart_d^L \leftrightarrow \ncmatch_{2d}^{TB} \leftrightarrow \ncpart_d^R.
  \]
\end{prop}

The bijections between different types of noncrossing partitions 
are related to the Kreweras complement map.

\begin{defn}[Kreweras complement maps]\label{def:kreweras}
  A \emph{topological Kreweras complement} is a bijection $\krew^{XY}
  \colon \ncpart_d^X \to \ncpart_d^Y$ from
  Proposition~\ref{prop:match-part}, where $X$ is $T$, $L$, $B$ or $R$
  and $Y = X^{\textrm{op}}$ is the \emph{opposite side} $B$, $R$, $T$
  or $L$, respectively.  Combining a topological Kreweras complement
  map with the bijections of Definition~\ref{def:ncpart-sides}
  produces the classical order-reversing \emph{Kreweras complement
  map} (see \cite{kreweras72}) $\krew \colon \ncpart_d \to \ncpart_d$,
  or its inverse.  The bottom-to-top and right-to-left maps produce
  the Kreweras complement, and the top-to-bottom and left-to-right
  maps produce its inverse.  Figure~\ref{fig:kreweras} illustrates the
  process.
\end{defn}

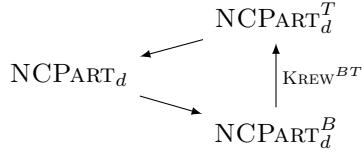
\begin{figure}
    \centering
    \begin{tikzcd}[row sep=tiny]
      & \ncpart_d^T \arrow[dd,latex-,"\krew^{BT}"] \\
      \ncpart_d \arrow[ur,latex-] & \\
      & \ncpart_d^B \arrow[ul,latex-]
    \end{tikzcd}
    \caption{Three bijections that do not form a commutative diagram, since  
    the composition of all three starting at $\ncpart_d$ is the Kreweras 
    complement map $\krew$, a non-trivial order-reversing automorphism.}
    \label{fig:kreweras}
\end{figure}

\subsection{Noncrossing permutations}\label{subsec:ncperm}
When noncrossing partitions are converted to noncrossing permutations, 
the Kreweras map has an algebraic description.

\begin{defn}[Noncrossing permutations]\label{def:ncperm}
  A set partition $[\lambda] \vdash [d]$ can be converted into a
  permutation $\pi \in \sym_d$ by turning each block of $[\lambda]$ of
  size $k$ into a $k$-cycle in which the elements of the block are
  listed in increasing order. For example, if $[\lambda] =
  137|2|45|6|8|9$, then $\pi = \perm([\lambda]) =(1\ 3\ 7)(4\ 5)$ in
  $\sym_9$. This defines a map $\perm\colon \spart_d \into \sym_d$
  which is injective because the set partition can be reconstructed
  from its permutation image.  When $[\lambda]$ is a noncrossing
  partition, we call $\pi = \perm([\lambda])$ a \emph{noncrossing
  permutation}.  Let $\ncperm_d$ denote the collection of noncrossing
  permutations in $\sym_d$. The restriction $\perm \colon \ncpart_d \to 
  \ncperm_d$ is a bijection as shown by Biane in \cite{biane97}.
  In the metric version of noncrossing partitions, the disjoint cycles of
  $\perm([\lambda])$ record the order in which vertices occur in the
  counterclockwise boundary cycle of the convex hull of each block of
  $[\lambda]$. 
\end{defn}

The bijection from $\ncpart_d$ to $\ncperm_d$ makes the latter into
a partially ordered set, and this has a useful interpretation which
is purely algebraic.

\begin{defn}[Absolute order]\label{def:abs-order}
    Let $T$ be the set of all transpositions in the symmetric group $\sym_d$
    and, for each $\sigma \in \sym_d$, define the \emph{absolute reflection length} 
    $\ell(\sigma)$ to be the length of a minimal factorization of $\sigma$
    into elements of $T$. Declaring $\sigma \leq \tau$ if 
    $\ell(\sigma) + \ell(\sigma^{-1}\tau) = \ell(\tau)$ makes $\sym_d$ into
    a partially ordered set which happens to be a lattice. 
    If we define $\delta$ to be the $d$-cycle $(1\ 2\ \cdots\ d)$, then the 
    interval $[1,\delta]$ in $\sym_d$ is the set $\ncperm_d$, so the set
    of noncrossing permutations is partially ordered with the 
    \emph{absolute order}. One should think of the interval 
    $[1,\delta]$ as a union of geodesics from the identity
    to $\delta$ in the Cayley graph of $\sym_d$ with respect 
    to the generating set $T$. Moreover, the map $\perm \colon \ncpart_d 
    \to \ncperm_d$ is a poset isomorphism.
\end{defn}

\begin{defn}[Noncrossing permutations of sides]\label{def:ncperm-sides}
  Composing the identifications in Definition~\ref{def:ncpart-sides}
  with the function in Definition~\ref{def:ncperm} produces new
  functions $\perm\colon \ncpart_d^X \to \ncperm_d$ where $X$ is $T$,
  $L$, $B$, or $R$.  We write $\pi^X$ when $\pi^X =
  \perm([\lambda]^X)$ for some $[\lambda]^X \in \ncpart_d^X$.  Alternatively, 
  let $\Ub$ be a closed, possibly disconnected, subsurface of $\Pb$ with contractible 
  components and $\ncpart(\Ub) = [\lambda]$.  The sides $S_X$ occur in the boundary 
  $\Ub$ and the disjoint cycles of $\pi^X$ record the counterclockwise ordering of 
  the sides in the boundaries of each component.
\end{defn}

\begin{defn}[Permutations and matchings]\label{def:ncperm-ncmatch}
  Let $[\mu] \vdash [2d]$ be a noncrossing matching.  By
  Proposition~\ref{prop:match-part}, the left-right matching
  $[\mu]^{LR}$ determines noncrossing partitions $[\lambda]^T$ and
  $[\lambda]^B$.  The corresponding noncrossing permutations, $\pi^T$
  and $\pi^B$, are the \emph{top and bottom permutations} associated
  with $[\mu]^{LR}$.  Similarly, the top-bottom matching $[\mu]^{TB}$
  determines noncrossing partitions $[\lambda]^L$ and $[\lambda]^R$,
  and the corresponding noncrossing permutations, $\pi^L$ and $\pi^R$,
  are the \emph{left and right permutations} associated with
  $[\mu]^{TB}$.
\end{defn}

There is a close algebraic connection between the two noncrossing
permutations associated with a noncrossing side-to-side matching.

\begin{example}\label{ex:factors}
 The noncrossing matching $[\mu]^{LR}$ in
 Example~\ref{ex:ncpart-ncmatch} determines the noncrossing partitions
 $[\lambda]^T = 137|2|45|6|8|9$ and $[\lambda]^B = 12|356|4|789$,
 which become the noncrossing permutations $\pi^T = (1\ 3\ 7)(4\ 5)$
 and $\pi^B = (1\ 2)(3\ 5\ 6)(7\ 8\ 9)$ in $\sym_9$.  Note that the
 composition $\pi^T \cdot \pi^B = (1\ 2\ \cdots\ 9)$. This can be
 understood geometrically.  The permutation $\pi^B$ is applied first.
 It identifies $6$ with the side $B_6$ which is connected to $B_3$ in the
 counterclockwise order in the boundary of the block $\{3,5,6\}$ and
 this side is identified with the number $3$.  The permutation $\pi^T$
 identifies $3$ with the side $T_3$ which is connected to $T_7$ in the
 counterclockwise order in the boundary of the block $\{1,3,7\}$ and
 this is identified with the number $7$.  This can be visualized as a
 path from $B_6$ to $B_3$ to $T_3$ (passing through $L_3$) to $T_7$.
 This is nearly a full circuit around the arc from $R_6$ to $L_3$,
 from the bottom side $B_6$ adjacent to $R_6$ to the top side $T_7$,
 also adjacent to $R_6$.  Between the start and end sides there is
 exactly one vertex where the indexing changes
 (Remark~\ref{rem:broken}), so $6$ goes to $7$ in the composition.
 Similarly, $i$ goes to $i+1$ mod $9$ for every $i \in [9]$.
\end{example}

Arguing as in Example~\ref{ex:factors} establishes the following.

\begin{prop}[Factors]\label{prop:factors}
  Let $[\mu] \vdash [2d]$ be a noncrossing matching and let $\delta$
  be the $d$-cycle $(1\ 2\ \cdots\ d)$. The top and bottom
  permutations associated to the left-right matching $[\mu]^{LR}$
  satisfy the equation $\pi^T \cdot \pi^B = \delta$, and the left and
  right permutations associated to the top-bottom matching
  $[\mu]^{TB}$ satisfy the equation $\pi^L \cdot \pi^R = \delta$.
\end{prop}

Proposition~\ref{prop:factors} gives an algebraic description of the
Kreweras complement map.

\begin{rem}[Kreweras complements]\label{rem:kreweras}
  Proposition~\ref{prop:factors} shows, for example, that the
  permutation version of $\krew^{BT}$ sends $\pi^B$ to $\pi^T = \delta
  \cdot (\pi^B)^{-1}$, its \emph{left complement} with respect to
  $\delta = (1\ 2\ \cdots\ d)$.  Similarly, the permutation version of
  $\krew^{TB}$ sends $\pi^T$ to $\pi^B = (\pi^T)^{-1} \cdot \delta$,
  its \emph{right complement} with respect to $\delta$. 
\end{rem}

\section{Geometric Combinatorics}\label{sec:geo-comb}

As we transition from looking at a single polynomial (Part~\ref{part:single-poly})
to looking at spaces of polynomials (Part~\ref{part:poly-spaces}), we introduce a way of sending
each monic polynomial to a point in a product of order complexes. 
After turning the combinatorics of the regular complex $\Pb_p$ into two chains in 
noncrossing partition lattices (\ref{subsec:side-chains}), we combine this with the 
metric of the critical value complex $\Qb'_p$ to create the $\geocomb$ map 
from a polynomial space to a product of two simplicial complexes (\ref{subsec:geocomb}).
 
\subsection{Chains of side partitions}\label{subsec:side-chains}
The combinatorial structure of the regular point complex $\Pb_p$ encodes two 
chains of noncrossing partitions.  We begin by listing our standing assumptions.

\begin{rem}[Standing assumptions]
  In this section $p\colon \C \to \C$ is a monic complex polynomial of
  degree $d$ with fixed coordinate rectangle $\closedsquare \supset
  \cvl(p)$.  The critical value complex $\Qb'_p$, regular value
  complex $\Qb_p$ and regular point complex $\Pb_p$ are as defined in
  Section~\ref{sec:complex-map}.  We assume that $\Qb'_p$ is
  $(k',l')$-subdivided and $\Qb_p$ is $(k,l)$-subdivided, with
  $k=k'+1$ and $l=l'+1$.  All illustrations in this section use the
  degree $5$ polynomial $p$ of Example~\ref{ex:deg5}.
\end{rem}

\begin{figure}
  \centering
  \begin{tabular}{ccc}
    \bt{c}\includegraphics[scale=.75]{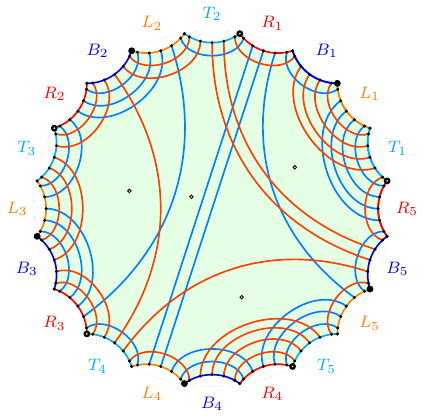}\et &
    $\longrightarrow$ &
    \bt{c}\includegraphics[scale=.35]{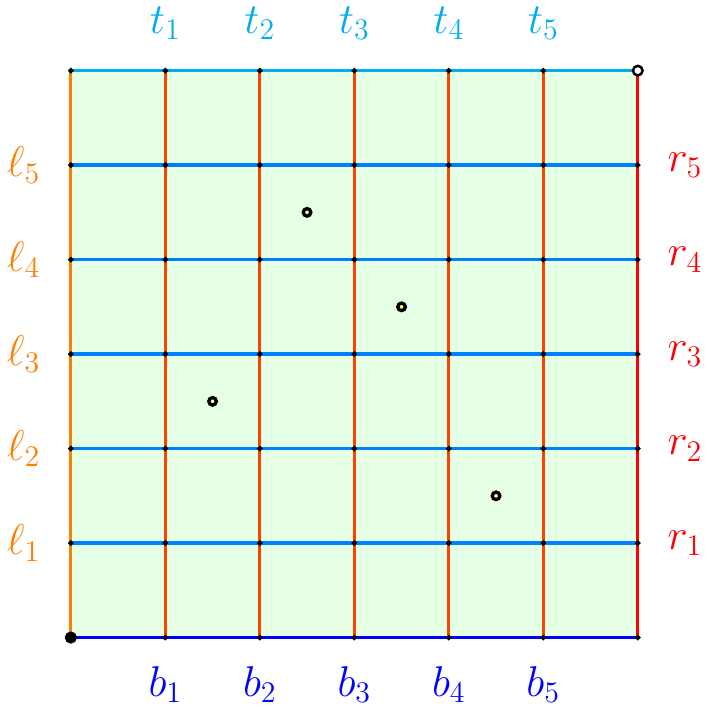}\et
  \end{tabular}
  \caption{The ``regular'' map $\pb \colon \Pb_p \to \Qb_p$ for the
    polynomial $p$ of Example~\ref{ex:deg5} is shown in standardized
    form. The regular value complex $\Qb_p$ of
    Figure~\ref{fig:deg5-original-dual-range} is shown here as a
    standard rectangle with an equally spaced $(5,5)$-subdivision.
    Similarly, the regular point complex $\Pb_p$ of
    Figure~\ref{fig:deg5-original-dual-domain} is shown here in a
    $20$-sided standard $5$-branched rectangle.
  \label{fig:deg5-stylized-dual}}
\end{figure}

Since we are only interested in its combinatorial structure, we redraw 
the regular point complex $\Pb_p$ as a cell structure on a standard 
branched rectangle.  We do this using a standard labeling that 
comes from the fact that $p$ is monic.

\begin{defn}[Labeling $\Pb_p$]\label{def:monic-poly-labels}
  Recall that cell structure of $\Qb'_p$ depends on the choice of rectangle $\closedsquare$ (Remark~\ref{rem:cells-sides}), 
  but once that choice has been made, the cell structure of $\Qb_p$ (and $\Pb_p$) is unaffected by the
  size of the larger rectangle on which $\Qb$ is built (Remark~\ref{rem:comb-inv}).  Thus we are 
  free to assume that interval $\Jb$ includes $0$.  Let $\gamma$ be
  the portion of the $x$-axis that starts in the right side $\Rb$ of
  $\Qb_p$ and extends to the right. Because $p$ is monic, the preimage
  of this regular ray lifts to $d$ topological rays that start in the
  right sides of $\Pb_p$ and end up asymptotic to the rays in the
  $d^{th}$ root of unity directions.  The side $\Rb_d = \Rb_0$ in
  $\Pb_p$ is the one whose ray ends up asymptotic to the positive
  $x$-axis, and the remaining sides and vertices are indexed
  so that its boundary labels look like a standard $d$-branched rectangle 
  (Definition~\ref{def:std-brrectangles}). See Figure~\ref{fig:deg5-original-dual-domain}.
\end{defn}

\begin{defn}[Standard representations]\label{def:std-spaces}
  A \emph{standard representation of $\Qb_p$} is one where the small
  rectangles in its $(k,l)$-subdivision are the same size, and a
  \emph{standard representation of $\Pb_p$} is one where its cell
  structure is drawn on the $4d$-gon that is the standard $d$-branched
  rectangle (Definition~\ref{def:std-brrectangles}), according to the
  monic polynomial labels (Definition~\ref{def:monic-poly-labels}).
\end{defn}

Figure~\ref{fig:deg5-stylized-dual} shows standard representations for
our running example.  The representations of
Definition~\ref{def:std-spaces} are possible because the regular point
complex $\Pb_p$ is always a nonsingular disk diagram.  This would not always be
possible for the critical point complex $\Pb'_p$ since it is singular
when the critical value complex $\Qb'_p$ has critical values in its
boundary (Corollary~\ref{cor:singular}).  The subdivision of $\Qb_p$
can be viewed as a horizontal subdivision and a vertical subdivision
that have been superimposed.

\begin{figure}
  \bt{ccc}
  \bt{c}\includegraphics[scale=.75]{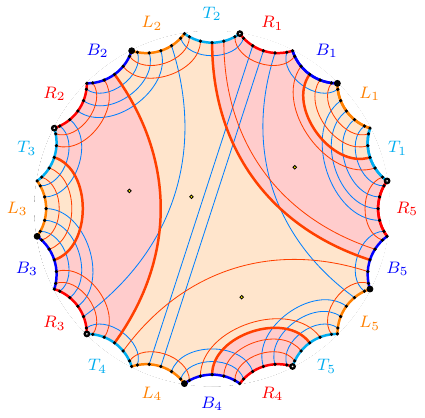}\et &
  $\longrightarrow$ &
  \bt{c}\includegraphics[scale=.35]{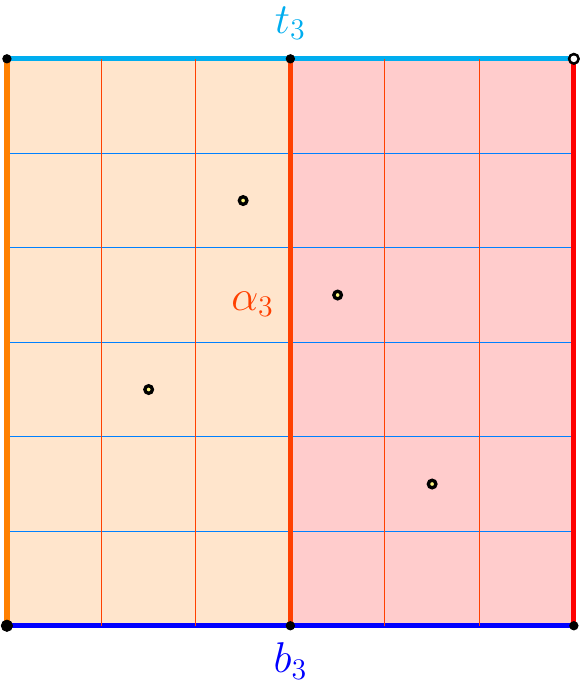}\et
  \et
  \caption{The vertical arc $\alpha_3$ from $b_3$ to $t_3$ in the
    range divides the regular value complex $\Qb_p$ into an orange
    left subsurface $\Vb_3^L$ and a red right subsurface $\Vb_3^R$.
    Their preimages divide the regular point complex $\Pb_p$ into an
    orange ``left'' subsurface $\Ub_3^L = p^{-1}(\Vb_3^L)$ and a red
    ``right'' subsurface $\Ub_3^R = p^{-1}(\Vb_3^R)$ separated by the
    multiarc $\widetilde \alpha_3 = p^{-1}(\alpha_3)$.  The corresponding
    topological noncrossing partitions are $\ncpart^L(\Ub_3^L) =
    1|245|3$ of the $5$ left sides of $\Pb_p$, and $\ncpart^R(\Ub_3^R)
    = 15|23|4$ of the $5$ right sides of
    $\Pb_p$.\label{fig:deg5-subsurfaces}}
\end{figure}

\begin{defn}[Arcs]\label{def:arcs}  
  The regular value complex $\Qb_p = \Ib_p \times \Jb_p$ is a cell
  structure on the rectangle $\Qb = \Ib \times \Jb$.  As intermediate
  steps we have $\Qb_{\Re(p)} = \Ib_p \times \Jb$ and let
  $\Qb_{\Im(p)} = \Ib \times \Jb_p$, which have a $(k,0)$-subdivision
  and a $(0,l)$-subdivision. Let $\alpha_i$ be the edge in
  $\Qb_{\Re(p)}$ from $b_i$ in the bottom side to $t_i$ in the top
  side, with $i \in [k]$. Let $\beta_j$ be the edge in $\Qb_{\Im(p)}$
  from $\ell_j$ in the left side to $r_j$ in the right side, with $j
  \in [l]$.  These persist as subdivided arcs in the
  $(k,l)$-subdivided $\Qb_p$ and we call them the \emph{vertical arcs
  $\alpha_i$} and the \emph{horizontal arcs $\beta_j$}.
\end{defn}
  
For our standard example, Figure~\ref{fig:deg5-subsurfaces} shows the
structures associated with a single vertical arc.  Here are the
definitions to make this precise.

\begin{defn}[Subsurfaces of $\Pb_p$ and $\Qb_p$]\label{def:subsurfaces}
  The closed complementary regions 
  (as in Definition~\ref{def:disks-curves})
  of the vertical arc $\alpha_i$ from $b_i$
  to $t_i$ are the \emph{left subsurface} $\Vb^L_i$ and the
  \emph{right subsurface} $\Vb^R_i$. The complementary regions of a
  horizontal regular arc $\beta_j$ from $\ell_j$ to $r_j$ are the
  \emph{top subsurface} $\Vb^T_j$ and the \emph{bottom subsurface}
  $\Vb^B_j$.  Collectively these the $2(k+l)$ \emph{side subsurfaces
  of $\Qb_p$}.  Their preimages are the \emph{left subsurface}
  $\Ub^L_i = p^{-1}(\Vb^L_i)$, the \emph{right subsurface} $\Ub^R_i =
  p^{-1}(\Vb^R_i)$, the \emph{top subsurface} $\Ub^T_j =
  p^{-1}(\Vb^T_j)$, and the \emph{bottom subsurface} $\Ub^B_j =
  p^{-1}(\Vb^B_j)$.  These are the $2(k+l)$ \emph{side subsurfaces of
  $\Pb_p$}.  See Table~\ref{tbl:partition-table}. 
\end{defn}

As in Example~\ref{ex:ncpart-ncmatch}, these subsurfaces produce
noncrossing partitions.

\begin{defn}[Side partitions and side matchings]\label{def:side-part-match}
  For each vertical arc $\alpha_i$ from $b_i$ to $t_i$, we have $L
  \subset \Vb^L_i \subset \Qb_p$, so $S_L \subset \Ub^L_i \subset
  \Pb_p$.  The \emph{$i^{th}$ left partition} $[\lambda_i]^L =
  \ncpart^L(\Ub^L_i)$ (Definition~\ref{def:ncpart-sides}) is the
  corresponding noncrossing partition of the left sides of $\Pb_p$.
  The preimage $\widetilde \alpha_i =p^{-1}(\alpha_i)$ of $\alpha_i$ is
  a multiarc, a collection of $d$ noncrossing arcs that connect the $d$ top sides
  and the $d$ bottom sides. The mulitarc $\widetilde \alpha_i$ defines a \emph{top-bottom 
  matching} $[\mu_i]^{TB}$ and forms the common boundary between the two subsurfaces $\Ub_i^L$ and $\Ub_i^R$.  
  Similarly, the \emph{$i^{th}$ right partition} is $[\lambda_i]^R =
  \ncpart^R(\Ub^R_i)$.  And for any horizontal arc $\beta_j$, the
  \emph{$j^{th}$ top partition} is $[\lambda_j]^T =
  \ncpart^T(\Ub^T_j)$ and the \emph{$j^{th}$ bottom partition} is
  $[\lambda_j]^B = \ncpart^B(\Ub^B_j)$.  The preimage 
  $\widetilde \beta_j = p^{-1}(\beta_j)$ of $\beta_j$ is a multiarc, a collection of
  $d$ noncrossing arcs that connect the $d$ left sides and the $d$
  right sides. The multiarc $\widetilde \beta_j$ defines a \emph{left-right matching} $[\mu_j]^{LR}$ and form
  the common boundary between the two subsurfaces $\Ub_j^T$ and $\Ub_j^B$.  These are the 
  \emph{side partitions} and \emph{side matchings} of $\Pb_p$. 
\end{defn}

\begin{table}
\newcommand\Tstrut{\rule{0pt}{2.6ex}}
\newcommand\Bstrut{\rule[-1.4ex]{0pt}{0pt}}
\bgroup
\def\arraystretch{1.5}
\begin{center}
\begin{tabular}{|c|c|c|c|c|}
    \hline
    name & partition & permutation & in $\Pb_p$ & in $\Qb_p$ 
    \Tstrut\Bstrut\\
    \hline
    $TB$ matching & $[\mu_i]^{TB} \in \ncmatch_{2d}^{TB}$ & --- & $\widetilde \alpha_i$ & $\alpha_i$ \Bstrut \\
    $LR$ matching & $[\mu_j]^{LR} \in \ncmatch_{2d}^{LR}$ & --- & $\widetilde \beta_j$ & $\beta_j$ \\
    \hline
    top partition & $[\lambda_j]^T \in \ncpart_d^T$ & $\pi_j^T$ & 
    $\Ub_j^T$ & $\Vb_j^T$\\
    left partition & $[\lambda_i]^L \in \ncpart_d^L$ & $\pi_i^L$ & 
    $\Ub_i^L$ & $\Vb_i^L$   \Tstrut \\
    bottom partition & $[\lambda_j]^B \in \ncpart_d^B$ & $\pi_j^B$ & 
    $\Ub_j^B$ & $\Vb_j^B$ \\
    right partition & $[\lambda_i]^R \in \ncpart_d^R$ & $\pi_i^R$ & 
    $\Ub_i^R$ & $\Vb_i^R$ \\
    \hline
\end{tabular}
\end{center}
\egroup
\caption{Notations for the side matchings and side partitions of the regular point complex $\Pb_p$ 
(Definition~\ref{def:side-part-match}).
\label{tbl:partition-table}}
\end{table}

\begin{figure}
  \bt{c}
  \includegraphics[scale=.3]{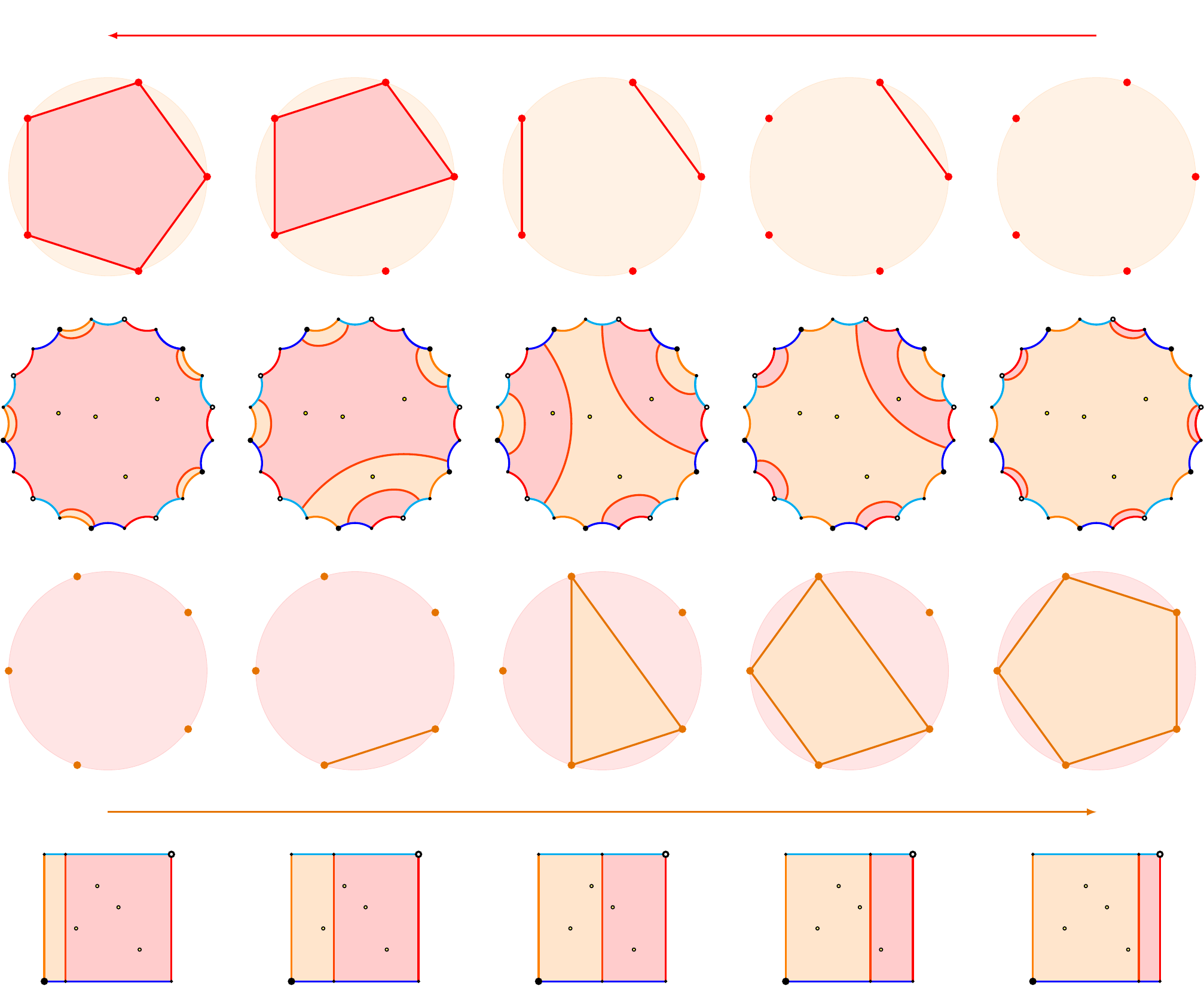}
  \et
  \caption{The bottom row shows the left and right subsurfaces
    $\Vb^L_i, \Vb^R_i \subset \Qb_p$ separated by the vertical arcs
    $\alpha_i$. The second row shows the corresponding left and right
    subsurfaces $\Ub^L_i, \Ub^R_i \subset \Pb_p$ separated by the $5$ 
    arcs of the multi arc $\widetilde \alpha_i$ which defines the top-bottom matching 
    $[\mu_i]^{TB}$. The third row shows 
    the chain $1|2|3|4|5 < 1|2|3|45 < 1|245|3 < 1|2345 < 12345$ of left 
    noncrossing partitions $[\lambda_i]^L \in \ncpart_5^L$. 
    And the top row shows the chain $12345 > 1235|4 > 15|23|4 > 15|2|3|4 >
    1|2|3|4|5$ of right noncrossing partitions $[\lambda_i]^R \in
    \ncpart_5^R$.
    \label{fig:bt-matchings}}
\end{figure}

Nested subsurfaces in $\Qb_p$ lift to nested subsurfaces in $\Pb_p$
and comparable noncrossing partitions.  For our standard example
Figure~\ref{fig:bt-matchings} shows the structures associated with the
vertical subdivisions of $\Qb_p$, and Figure~\ref{fig:lr-matchings}
shows the structures associated with the horizontal subdivisions of
$\Qb_p$. The side chains are strictly monotonic because of the way in
which the cell complexes for $p$ are defined.

\begin{lem}[Distinct]\label{lem:increasing}
  Nested side surfaces of $\Pb_p$ have distinct partitions.
\end{lem}

\begin{proof}
  We prove this for nested left subsurfaces. For any
  $i_1,i_2 \in [k]$ with $i_1<i_2$, we have $\Ub_{i_1}^L \subset \Ub_{i_2}^L$ and
  $[\lambda_{i_1}]^L \leq [\lambda_{i_2}]^L$ (Definition~\ref{def:nc-top}).
  We know that there is at least one critical value of $p$ in
  $\Vb_{i_2}^L \setminus \Vb_{i_1}^L$ by construction.  As a consequence
  the subsurface $\Ub^L_{i_1}$ has more connected components than
  $\Ub^L_{i_2}$ (Proposition~\ref{prop:disk-pre}) and the noncrossing
  permutation $\perm([\lambda_{i_1}]^L)$ has more disjoint cycles than
  $\perm([\lambda_{i_2}]^L)$ (Definition~\ref{def:ncmatch}), so the
  inequality is strict.
\end{proof}

\begin{defn}[Side chains of $\Pb_p$]\label{def:side-chains}
  The $k$ vertical arcs $\alpha_i$ from $b_i$ to $t_i$ define nested
  left subsurfaces $L \subset \Vb^L_1 \subset \Vb^L_2 \subset \cdots
  \subset \Vb^L_k \subset \Qb_p$ which lift to $S_L \subset \Ub^L_1
  \subset \Ub^L_2 \subset \cdots \subset \Ub^L_k \subset \Pb_p$ which
  in turn define an increasing chain $[\lambda_1]^L < [\lambda_2]^L <
  \cdots < [\lambda_k]^L$ of left noncrossing partitions in
  $\ncpart_d^L$ (Definition~\ref{def:nc-top} and
  Lemma~\ref{lem:increasing}).  In the other direction, the nested
  right subsurfaces $\Qb_p \supset \Vb^R_1 \supset \Vb^R_2 \supset
  \cdots \supset \Vb^R_k \supset R$ lift to $\Pb_p \supset \Ub^R_1
  \supset \Ub^R_2 \supset \cdots \supset \Ub^R_k \supset S_R$ which
  define a decreasing chain $[\lambda_1]^R > [\lambda_2]^R > \cdots >
  [\lambda_k]^R$ of right noncrossing partitions in $\ncpart_d^R$.
  Similarly, the $l$ horizontal arcs $\beta_j$ from $\ell_j$ to $r_j$
  define an increasing chain $[\lambda_1]^B < [\lambda_2]^B < \cdots <
  [\lambda_l]^B$ of bottom noncrossing partitions in $\ncpart_d^B$,
  and a decreasing chain $[\lambda_1]^T > [\lambda_2]^T > \cdots >
  [\lambda_j]^T$ of top noncrossing partitions in $\ncpart_d^T$.
  These are the $4$ \emph{side chains of $\Pb_p$}.  Since the right
  and left chains are related by the topological Kreweras complement
  map (Definition~\ref{def:kreweras}), they contain the same
  information, and the same is true for the top and bottom chains.  We
  focus on the increasing left and bottom chains.
\end{defn}
 
\begin{figure}
  \bt{c}
  \includegraphics[scale=.3]{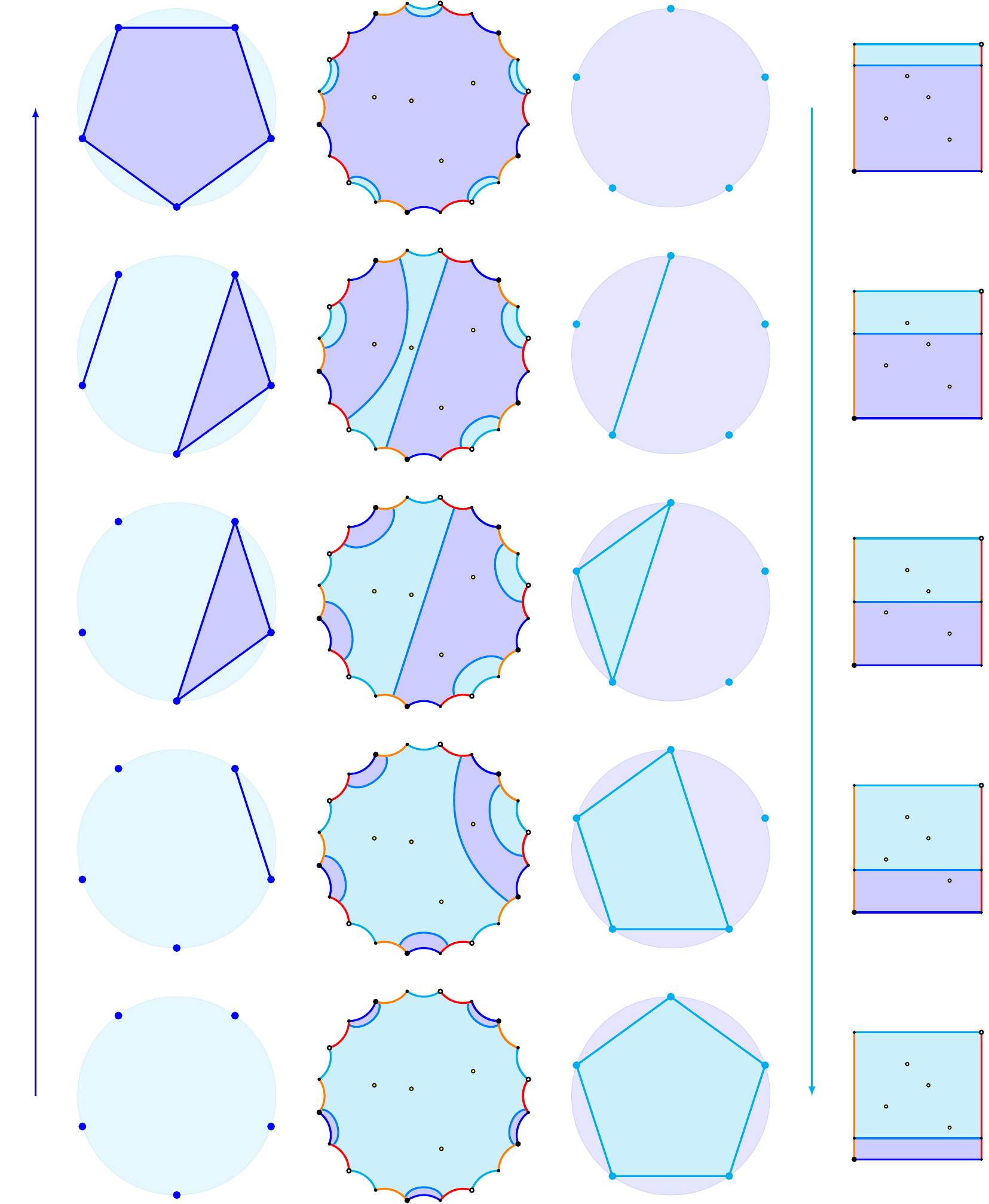}
  \et
  \caption{The last column shows the top and bottom subsurfaces
    $\Vb^T_j, \Vb^B_j \subset \Qb_p$ separated by the horizontal arcs
    $\beta_j$.  The second column shows the corresponding top and
    bottom subsurfaces $\Ub^T_j, \Ub^B_j \subset \Pb_p$ separated by the $5$ arcs
    of the multiarc $\widetilde \beta_j$ which defines the left-right matching
    $[\mu_j]^{LR}$. The first column shows the
    chain $1|2|3|4|5 15|2|3|4 < 145|2|3 < 145|23 < 12345$ of bottom
    noncrossing partitions $[\lambda_j]^B \in \ncpart_5^B$. And the
    third column shows the chain $1|2|3|4|5 < 1|24|3|5 < 1|234|5 <
    1|2345 < 12345$ of top noncrossing partitions $[\lambda_j]^T \in
    \ncpart_5^T$.
    \label{fig:lr-matchings}}
\end{figure}

In our example, the side chains start at the discrete partition and
end at the indiscrete partition because $\partial \Qb'_p$ is regular.

\begin{lem}[Discrete and indiscrete]\label{lem:discrete-indiscrete}
  The side chains of $\Pb_p$ start at the discrete partition if and
  only if the corresponding start side of $\Qb'_p$ is regular, and
  they end at the indiscrete partition if and only if the
  corresponding end side of $\Qb'_p$ is regular.
\end{lem}

\begin{proof}
  It is sufficient to prove this for the start of the left chain.  By
  the same reasoning as in the proof of Lemma~\ref{lem:increasing},
  the first partition $[\lambda_1]^L$ is discrete if and only if the
  first column of $\Qb_p$ is regular, which is true if and only if the
  left side of $\Qb'_p = \closedsquare$ has no critical values.
\end{proof}

\begin{rem}[Real and imaginary Morse functions]\label{rem:morse}
  The left chain of $\Pb_p$ essentially encodes the changes in the
  topology of the preimages of lower intervals for the \emph{real
  Morse function} $\Re \circ p \colon \C \to \R$, but with slight
  modifications because of the rectangle $\closedsquare =
  [x'_\ell,x'_r] \times [y'_b,y'_t]$ surrounding the critical
  values. The critical values of this Morse function are the points
  $C' = \{x'_1,\ldots,x'_k\}$ with the possible addition of $x'_\ell$
  and $x'_r$, if there is a critical value in the left and/or right
  side of $\closedsquare$.  The values $x_i$ used to define $\Ib_p$
  are representatives of the equivalence classes of regular levels of
  this Morse function, with possible duplication at the extremes if
  there are no critical values in the left and/or right side of
  $\closedsquare$.  To see the connection with Morse level sets, note
  that the regular point $x_i$ used to define the vertical arc
  $\alpha_i$ and the left subsurface $\Vb_i^L$, also defines the left
  half-space $\Hb_i^L$ with the vertical line $\{x_i\} \times \R$ as
  its boundary.  And since each $\Hb^L_i \setminus \Vb^L_i$ is regular,
  the preimages under $p$ have the same homotopy type.  In
  particular, the portion of the domain shown in the diagrams in the
  second row of Figure~\ref{fig:bt-matchings} contains all of the
  interesting information. Similar comments hold for the bottom chain
  of $\Pb_p$ and the \emph{imaginary Morse function} $\Im \circ p
  \colon \C \to \R$.
\end{rem}

\subsection{The geometric combinatorics map}\label{subsec:geocomb}
We now assemble all of the geometric combinatorial information about 
monic polynomials with critical values in a fixed closed rectangle $\closedsquare$
into a single continuous map from a polynomial space to a bisimplicial cell complex.
The range is built out of order complexes.

\begin{defn}[Order complex]\label{def:ord-cplx}
  Let $\size{\ncpart_d}_\simp$ be the ordered simplicial complex that is the 
  \emph{order complex} of $\ncpart_d$.  Its vertices are labeled by the elements of $\ncpart_d$ and 
    it has an ordered $k$-simplex on vertices $[\lambda_1],[\lambda_2],\ldots,[\lambda_k]$
    for every chain $[\lambda_1] < [\lambda_2] < \cdots < [\lambda_k]$ in $\ncpart_d$. 
    A point in $|\ncpart_d|_{\Delta}$ is a formal sum $\sum_{i\in [k]} \width_i \cdot [\lambda_i]$, 
    where the $[\lambda_i]$ label the vertices of an ordered simplex in $\size{\ncpart_d}_\simp$
    and the real numbers $\width_i \in (0,1)$ with $\sum_{i\in [k]} \width_i = 1$ are the barycentric 
    coordinates of a point in the interior of that simplex. 
\end{defn}

\begin{defn}[Geometric combinatorics map]\label{def:geom-comb}
  Let $\poly_d^{mc}(\closedsquare)$ be the set of monic centered complex
  polynomials of degree $d$ with $\cvl(p)$ in a fixed coordinate
  rectangle $\closedsquare$.  There is a well-defined
  \emph{geometric combinatorics map} 
  \[
    \geocomb \colon \poly_d^{mc}(\closedsquare) \to \size{\ncpart^L_d}_\simp \times
    \size{\ncpart^B_d}_\simp
  \] 
  defined by sending a polynomial $p$ to
  \[
    \geocomb(p) = \left(\sum_{i\in [k]} \width_i^I \cdot
    [\lambda_i]^L,\sum_{j\in [l]} \width_j^J \cdot [\lambda_j]^B \right).
  \]
  The left side partitions $[\lambda_i]^L$ of Definition~\ref{def:side-chains}
  label the vertices of a simplex in the order complex $\size{\ncpart_d^L}_\simp$, 
  and the numbers $(\width_1^I,\ldots,\width_k^I) = \bary(\Ib'_p)$ of 
  Definition~\ref{def:int-subdivide} are the barycentric coordinates of a point in 
  the interior of that simplex (Definition~\ref{def:ord-cplx}).
  Similarly, the bottom side partitions $[\lambda_i]^B$ label the vertices of a 
  simplex in the order complex $\size{\ncpart_d^B}_\simp$ and the relative widths 
  $(\width_1^J,\ldots,\width_l^J) = \bary(\Jb'_p)$ are the barycentric coordinates 
  of a point in in the interior of that simplex.  The map $\perm \colon 
  \ncpart_d \to \ncperm_d$ (Definition~\ref{def:ncperm}) 
   sends side partitions such as $[\lambda_i]^L$ to \emph{side permutations}
  $\pi_i^L = \perm([\lambda_i]^L)$. See Table~\ref{tbl:partition-table} on 
  page~\pageref{tbl:partition-table}.  This means that $\size{\ncpart_d}_\simp = 
  \size{\ncperm_d}_\simp$ with relabeled vertices.  There is an alternative version 
  of the geometric combinatorics map 
  \[
    \geocomb \colon \poly_d^{mc}(\closedsquare) \to \size{\ncperm^L_d}_\simp \times \size{\ncperm^B_d}_\simp
  \] 
  using noncrossing permutations, and defined by the formula
  \[
    \geocomb(p) = \left(\sum_{i\in [k]} \width_{i}^I \cdot
    (\pi_i^L), \sum_{j\in [l]} \width_{j}^J \cdot (\pi_j^B) \right).
  \]
In both formulations, the simplices are determined by the combinatorial structure of the 
regular point complex $\Pb_p$, and the points in these simplices are selected
using the metric structure of the critical value complex $\Qb'_p$.
\end{defn} 

\newpage
\part{Spaces of Polynomials}\label{part:poly-spaces}

In this part the focus shifts from a single polynomial to spaces of
polynomials.  The Lyashko--Looijenga map, or $\LL$ map, is a stratified
covering map from a polynomial space to a multiset space. Its
structure makes it possible to lift certain types of multiset paths 
(a continuously varying families of multisets) to a polynomial path (a
continuously varying family of polynomials).  And lifted polynomial
paths lead to polynomial homotopies that continuously modify spaces of
polynomials.  The goal of Part~\ref{part:poly-spaces} is to prove
Theorems~\ref{mainthm:homeomorphisms}, \ref{mainthm:compactifications}
and~\ref{mainthm:deformations} using these polynomial homotopies.

Part~\ref{part:poly-spaces} is structured as follows.
Section~\ref{sec:prod-mult} discusses the $\mult$ map as a stratified
covering $\mult \colon \Xb^n \to \mult_n(\Xb)$.
Section~\ref{sec:poly-ll-maps} discusses the Lyashko--Looijenga map,
or $\LL$ map, as a stratified covering $\LL \colon
\poly_d^{mt}(\Xb) \to \mult_n(\Xb)$.  Section~\ref{sec:path-lift}
shows that paths in $\mult_n(\Xb)$ with weakly increasing shapes can
be uniquely lifted to $\poly_d^{mt}(\Xb)$ through the $\LL$ map, using 
path lifting to $\Xb^n$ through the $\mult$ map as a model.
Section~\ref{sec:poly-homotopy} turns these unique polynomial lifts
into polynomial homotopies, and Section~\ref{sec:poly-spaces} uses
these to establish relationships between spaces of polynomials,
including homeomorphisms, compactifications, deformations and
quotients.  

\section{Products and Multisets}\label{sec:prod-mult}

This section describes stratifications of the product space $\Xb^n$
and the multiset space $\mult_n(\Xb)$ that turn the $\mult$ map into a
stratified covering.  After discussing product strata
(\ref{subsec:prod-strata}) we discuss multiset
strata~(\ref{subsec:mult-strata}).

\subsection{Product strata}\label{subsec:prod-strata}
The horizontal map $\spart \colon \Xb^n \to \spart_{[n]}$ in
Figure~\ref{fig:part-mult} on page~\pageref{fig:part-mult} subdivides
$\Xb^n$ into strata we call $\Xb_{[\lambda]}$ and into subspaces we
call $\Xb^{[\lambda]}$. We begin with an example and then define the
notation.

\begin{figure}
  \centering
  \tikzstyle{GreenPoly}=[thick,color=green!50!black,fill=green!30,join=bevel]

\begin{tikzpicture}
    \begin{scope}[scale=1,x={(4cm,0cm)},y={(1.8cm,1.2cm)},z={(0cm,4cm)}, shift = {(-1,0)}]
        \draw[thin,-latex] (0,0,0) -> (0,0,1.5);
    
        \filldraw[GreenPoly] (0,0,0) -- (1,1,0) -- (1,1,1) -- (0,0,1) -- cycle;
    
        \draw[thin,-latex] (0,0,0) -> (1.5,0,0);
        \draw[thin,-latex] (0,0,0) -> (0,1.5,0);
    
        \node[anchor=north] () at (1.5,0,0) {$x$};
        \node[anchor=west] () at (0,1.5,0) {$y$};
        \node[anchor=east] () at (0,0,1.5) {$z$};
        
        \draw[] (0,0,0) -- (1,0,0) -- (1,1,0) -- (0,1,0) -- (0,0,0);
        \draw[] (0,0,1) -- (1,0,1) -- (1,1,1) -- (0,1,1) -- (0,0,1);
        \draw[] (1,0,0) -- (1,0,1);
        \draw[] (0,1,0) -- (0,1,1);
    
        \node () at (0.75,0,-0.25) {$\mathbf{I}^{12|3}$};
        
    \end{scope}

    \begin{scope}[scale=1,x={(4cm,0cm)},y={(1.8cm,1.2cm)},z={(0cm,4cm)}, shift = {(1,0)}]

        \draw[thin,-latex] (0,0,0) -> (0,0,1.5);
        \draw[] (0,0,0) -- (0,0,1);
        \draw[] (1,1,0) -- (1,1,1);
        
        \filldraw[GreenPoly] (0,0,0) -- (1,1,0) -- (1,1,1) -- (0,0,1) -- cycle;
        \draw[ultra thick, color=white] (0,0,0) -- (1,1,1);
        \draw[GreenPoly, ultra thick, dashed] (0,0,0) -- (1,1,1);
    
        \draw[thin,-latex] (0,0,0) -> (1.5,0,0);
        \draw[thin,-latex] (0,0,0) -> (0,1.5,0);
    
        \node[anchor=north] () at (1.5,0,0) {$x$};
        \node[anchor=west] () at (0,1.5,0) {$y$};
        \node[anchor=east] () at (0,0,1.5) {$z$};
        
        \draw[] (0,0,0) -- (1,0,0) -- (1,1,0) -- (0,1,0) -- (0,0,0);
        \draw[] (0,0,1) -- (1,0,1) -- (1,1,1) -- (0,1,1) -- (0,0,1);
        \draw[] (1,0,0) -- (1,0,1);
        \draw[] (0,1,0) -- (0,1,1);

        \node[shape=circle,draw,color=green!50!black,fill=white,inner sep=1.5pt] () at (0,0,0) {};
        \node[shape=circle,draw,color=green!50!black,fill=white,inner sep=1.5pt] () at (1,1,1) {};
    
        \node () at (0.75,0,-0.25) {$\mathbf{I}_{12|3}$};
        
    \end{scope}
\end{tikzpicture}
  \caption{On the left, the subspace $\Ib^{12|3}$ forms a rectangular
    subspace with side lengths of $\sqrt{2}$ and $1$ inside the
    $3$-cube $\Ib^3$. On the right, the stratum $\Ib_{12|3}$ is the
    same rectangle, but with the long diagonal $\Ib^{123} = \Ib_{123}$
    removed.}
  \label{fig:cube-subspaces}
\end{figure}
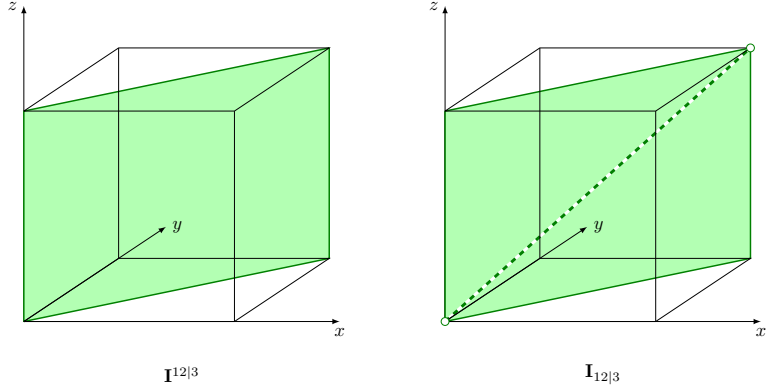

\begin{exmp}[Cubes]\label{ex:cubes}
  Let $\Ib = [0,1]$, or more generally let $\Ib = [x_\ell,x_r]$ be an 
  interval of length $s$.  The subspace $\Ib^{12|3}$ inside the $3$-cube $\Ib^3$ is an $(\sqrt{2})s \times
  s$ rectangle containing the long diagonal subspace
  $\Ib^{123}$.  The stratum $\Ib_{12|3}$ is the same rectangle with
  this long diagonal removed. Note that the closure of the stratum
  $\Ib_{12|3}$ is the subspace $\Ib^{12|3}$.  The stratum
  $\Ib_{12|3}$ is homeomorphic but not isometric to $\conf_2(\Ib)$,
  the space of two distinct unlabeled points in $\Ib$. See
  Figure~\ref{fig:cube-subspaces}.  More generally, $\Ib^n$ is an
  $n$-dimensional cube of side length~$s$.  For each set partition
  $[\lambda]\vdash [n]$ with $k$ blocks of size $a_1, \ldots, a_k$,
  the subspace $\Ib^{[\lambda]}$ is a product of $k$ line segments
  with lengths $(\sqrt{a_1})s, \ldots, (\sqrt{a_k})s$. There is a
  unique minimal simplicial cell structure on $\Ib^n$ that contains
  all of these subspaces $\Ib^{[\lambda]}$ as subcomplexes.
\end{exmp}

\begin{defn}[Product strata]\label{def:prod-strata}
  Let $\Xb_{[\lambda]} = \{ \bx \in \Xb^n \mid \spart(\bx) =
  [\lambda]\}$ for each $[\lambda] \in \spart_n$. These are the
  \emph{strata of $\Xb^n$}, and $\Xb^n$ is a disjoint union of these
  nonempty subspaces (Definition~\ref{def:set-maps}). The indiscrete
  stratum $\Xb_{[n^1]}$, where all coordinates are equal, is the
  \emph{long thin diagonal} of $\Xb^n$.  It is a topological diagonal
  copy of $\Xb$ with its metric dilated by a factor of $\sqrt{n}$.
  The discrete stratum $\Xb_{[1^n]}$ is $\conf_n(\Xb)$, the
  \emph{configuration space} of all collections of $n$ distinct
  labeled points.
\end{defn}

In addition to the strata $\Xb_{[\lambda]}$, we also define a
different, overlapping collection of subspaces $\Xb^{[\lambda]}$, also
indexed by set partitions, whose structure is easier to describe.

\begin{defn}[Product subspaces]\label{def:prod-subspaces}
  For a set partition $[\lambda] \vdash [n]$ let $\Xb^{[\lambda]}$ be
  the collection of points $\bx$ where coordinates belonging to the
  same block of $[\lambda]$ are required to be equal.  Unlike the
  definition of $\Xb_{[\lambda]}$, we do \emph{not} require that
  distinct blocks have distinct coordinate values.  For example,
  $\Xb^{13|2|4} = \{ (a,b,a,c)\in X^4\}$ while $\Xb_{13|2|4} = \{
  (a,b,a,c)\in \Xb^4 \mid a,b,c \textrm{ distinct}\}$.  Since an
  $n$-tuple $\bx$ is in $\Xb^{[\lambda]}$ if and only if $\shape(\bx)
  \geq [\lambda]$, $\Xb^{[\lambda]}$ is a disjoint union of
  strata: \[\Xb^{[\lambda]} = \bigsqcup_{[\mu] \geq [\lambda]}
  \Xb_{[\mu]}.\]
\end{defn}
  
The rectangular boxes in the $n$-cube (Example~\ref{ex:cubes})
illustrate the general properties of product subspaces and their
strata.  Let $[\lambda] \vdash [n]$ be a set partition with $k$
blocks.  We show that the subspace $\Xb^{[\lambda]}$ is a rescaled
version of $\Xb^k$ (Proposition~\ref{prop:subsp-metrics}), the stratum
$\Xb_{[\lambda]}$ is a rescaled version of $\conf_k(\Xb)$
(Proposition~\ref{prop:strata-metrics}), and the topological closure
of $\Xb_{[\lambda]}$ is $\Xb^{[\lambda]}$
(Proposition~\ref{prop:strata-closure}).  The proofs use natural maps
between set partitions and indexing sets.

\begin{defn}[Set partition maps]\label{def:spart-maps}
  Let $[\lambda] \vdash [n]$ be a set partition with $k$ blocks and
  let $\lambda = \lambda_1^{a_1}\cdots \lambda_\ell^{a_\ell}$ be its
  shape. There are canonical maps $[n] \to [\lambda] \to [\ell]$.  The
  first sends $i \in [n]$ to the block containing $i$ in $[\lambda]$,
  The second sends a block of $[\lambda]$ of size $\lambda_j$ to $j
  \in [\ell]$.  If we pick a bijection $[\lambda] \to [k]$ fixing an
  ordering of the blocks, we get maps $[n] \to [k] \to [\ell]$ between
  standard sets of integers.  Note that functions to and from
  $[\lambda]$ are technically functions to and from the set of blocks of
  $[\lambda]$.
\end{defn}

\begin{rem}[Points as functions]\label{rem:pt-fn}
  In the language of Remark~\ref{rem:functions}, the points in
  $X^{[\lambda]}$ bijectively correspond to functions $[\lambda]\to
  \Xb$, whereas the points in $\Xb_{[\lambda]}$ correspond to the
  injective functions $[\lambda] \into \Xb$.  A ``point'' $[n]\to \Xb$
  in $\Xb^{[\lambda]}$ factors through the set partition map $[n]\to
  [\lambda]$ (Definition~\ref{def:spart-maps}) to yield a ``point''
  $[\lambda]\to \Xb$, and composition with this set partition map
  reverses the process.  If $[\lambda]$ has $k$ blocks and we fix a
  bijection $[k] \to [\lambda]$ ordering the blocks of $[\lambda]$,
  then composition establishes an additional bijection, in fact a
  homeomorphism, between ``points'' in $\Xb^{[\lambda]}$ and ``points'' in
  $\Xb^k$.
\end{rem}

\begin{prop}[Subspace metrics]\label{prop:subsp-metrics}
  Let $[\lambda] \vdash [n]$ be a set partition with $k$ blocks.
  Ordering of the blocks of $[\lambda]$ produces a homeomorphism
  between $\Xb^{[\lambda]} \subset \Xb^n$ and $\Xb^k$.  Moreover, this
  becomes an isometry by rescaling the $i^{th}$ coordinate of $\Xb^k$ by
  $\sqrt{j}$, where $j$ is the size of the corresponding block in
  $[\lambda]$.
\end{prop}

\begin{proof}
  Pick a bijection $[k] \to [\lambda]$ ordering the blocks of
  $[\lambda]$ and define a bijection between $\Xb^k$ and
  $\Xb^{[\lambda]}$ as described in Remark~\ref{rem:pt-fn}.  This
  function is continuous with a continuous inverse.  Moreover, when
  restricted to a single coordinate of $\Xb^k$ and the corresponding
  block of coordinates in $\Xb^{[\lambda]}$, this is the long thin
  diagonal embedding of Definition~\ref{def:prod-strata}, which
  accounts for the stretch factors.
\end{proof}

\begin{prop}[Stratum metrics]\label{prop:strata-metrics}
  If $[\lambda] \vdash [n]$ is a set partition with $k$ blocks, then
  the stratum $\Xb_{[\lambda]} \subset \Xb^n$ is homeomorphic to
  $\conf_k(\Xb)$.  Moreover, this homeomorphism can be promoted to an
  isometry by rescaling the $i^{th}$ coordinate of $\conf_k(\Xb)$ by
  $\sqrt{j}$, where $j$ is the size of the corresponding block in
  $[\lambda]$.
\end{prop}

\begin{proof}
  The map in Proposition~\ref{prop:subsp-metrics} sends $\conf_k(\Xb)$
  to $\Xb_{[\lambda]}$.
\end{proof}

\begin{prop}[Stratum closures]\label{prop:strata-closure}
  The closure of $\conf_n(\Xb)$ is $\Xb^n$, and for each set partition
  $[\lambda]$, the closure of the stratum $\Xb_{[\lambda]}$ is the
  subspace $\Xb^{[\lambda]}$.
\end{prop}

\begin{proof}
  For the types of spaces considered here (Remark~\ref{rem:spaces}),
  it is easy to see that every point in $\Xb^n$ is a limit of points
  where the coordinates are distinct.  Thus the closure of
  $\conf_n(\Xb)$ is $\Xb^n$ and, by
  Proposition~\ref{prop:subsp-metrics}, the closure of
  $\Xb_{[\lambda]}$ contains all of $\Xb^{[\lambda]}$.  Conversely,
  the closure of $\Xb_{[\lambda]}$ is contained in $\Xb^{[\lambda]}$
  since any convergent sequence of points where specific coordinates
  are always equal has a limit where those coordinates remain equal.
\end{proof}

\subsection{Multiset strata}\label{subsec:mult-strata}
The horizontal map $\shape \colon \mult_n(\Xb) \to \ipart_{n}$ in
Figure~\ref{fig:part-mult} on page~\pageref{fig:part-mult} subdivides
$\mult_n(\Xb)$ into strata we call $\mult_\lambda(\Xb)$ and into
subspaces we call $\mult^\lambda(\Xb)$. The special case where 
$\Xb$ is the interval $\Ib = [x_\ell,x_r]$ is important in Part~\ref{part:geo-comb}.

\begin{figure}
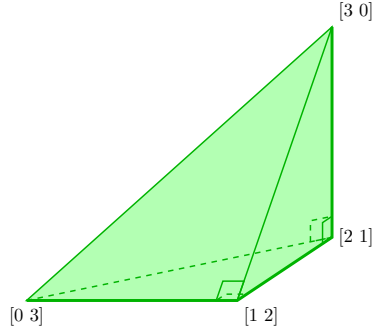

  \includestandalone[scale=.7]{fig-ortho3}
  \caption{A standard $3$-orthoscheme $\mult_3(\Ib)$. The vertices are
    multisets of the form $x_\ell^{a_\ell} x_r^{a_r}$ with $a_\ell + a_r = 3$, so we label
    them using the shorthand $[a_\ell \ a_r]$. The spine of the orthoscheme
    is indicated by the thick path from $[0\ 3]$ to
    $[1\ 2]$ to $[2\ 1]$ to $[3\ 0]$.\label{fig:ortho3}}
\end{figure}

\begin{defn}[Orthoschemes]\label{def:orthoschemes}    
  The cell structure on $\Ib^n$ of Example~\ref{ex:cubes} has $n!$
  top-dimensional simplices, it admits a cellular action by $\sym_n$
  via permuting coordinates, and the orbifold quotient $\mult_n(\Ib)=
  \Ib^n/\sym_n$ can be identified with a single closed $n$-dimensional
  simplex. Since this simplex is the convex hull of $n$
  pairwise-orthogonal line segments with equal length, it is called a
  \emph{standard $n$-orthoscheme} of side length $s$. This path of pairwise-orthogonal line segments 
  through the vertex set is its \emph{spine}.  The closed cells of a 
  standard $n$-orthoscheme are \emph{non-standard orthoschemes}
  with a \emph{spine} of pairwise-orthogonal line segments of length $(\sqrt{a_1})s, 
  \ldots, (\sqrt{a_k})s$ and $a_1 + \cdots + a_k \leq n$. 
  See Figure~\ref{fig:ortho3} for an example and see \cite{BrMc-10} and \cite{DoMcWi20} 
  for more on orthoschemes.
\end{defn}

\begin{defn}[Multiset strata]\label{def:mult-strata}
  For each $\lambda \vdash n$ in $\ipart_{n}$, let $\mult_\lambda(\Xb)
  = \{ M \in \mult_n(\Xb) \mid \shape(M)=\lambda\}$. These are the
  \emph{strata of $\mult_n(\Xb)$}, and $\mult_n(\Xb)$ is a disjoint
  union of these nonempty subspaces.  The indiscrete stratum
  $\mult_{[1^n]}$ is the isometric image of the long thin diagonal in 
  $\Xb^n$ sent to $\mult_n(\Xb)$.  The discrete stratum is
  $\mult_{1^n}(\Xb)$ is $\set_n(X)$, the collection of $n$ distinct
  unlabeled points in $\Xb$, also known as the \emph{unlabeled
  configuration space} $\uconf_n(\Xb)$).  The symmetric group acts
  freely on $\conf_n(\Xb)$, $\set_n(\Xb)$ is the quotient space
  $\conf_n(\Xb) / \sym_n$, and the map $\set\colon \conf_n(\Xb) \to
  \set_n(\Xb)$ which sends the $n$-tuple $\bx = (x_1,\ldots,x_n)$ to
  the $n$-element set $\set(\bx) = \{x_1,\ldots,x_n\}$ is an
  $n!$-sheeted covering map.  See Figure~\ref{fig:conf-spaces}.
\end{defn}

\begin{figure}
  \begin{tikzcd}[row sep=small, column sep=small]
    & \textrm{distinct} && \textrm{arbitrary} \\ \textrm{labeled} &
    \conf_n(\Xb) \arrow[rr,hookrightarrow]
    \arrow[dd,twoheadrightarrow,"\set"'] && \Xb^n
    \arrow[dd,twoheadrightarrow,"\mult"] \\ \\ \textrm{unlabeled} &
    \set_n(\Xb) \arrow[rr,hookrightarrow] && \mult_n(\Xb)
  \end{tikzcd}
  \caption{Configuration spaces with points that are distinct and/or
    labeled.  The horizontal maps are inclusions and the vertical maps
    are quotients. The one on the left is a covering
    map.\label{fig:conf-spaces}}
\end{figure}

\begin{defn}[Multiset subspaces]
  Next let $\mult^\lambda(\Xb) = \{ M \in \mult_n(\Xb) \mid \shape(M)
  \geq \lambda\}$.  We call these the \emph{subspaces of
  $\mult_n(\Xb)$} even though their structure is much more complicated
  than it is in the product case.  As before each subspace
  $\mult^\lambda(\Xb)$ is a disjoint union of strata: \[
  \mult^\lambda(\Xb) = \bigsqcup_{\mu \geq \lambda} \mult_\mu(\Xb).\]
\end{defn}

It is not too hard to see that for any set partition $[\lambda]$ with shape 
$\lambda = \shape([\lambda])$, the restricted map $\mult \colon 
\Xb_{[\lambda]} \to \mult_\lambda(\Xb)$ is a covering map.  This 
is because small changes to the coordinates of an $n$-tuple in the domain,
with equal coordinates staying equal, correspond to small changes to the elements 
of the underlying set $S = \set(M)$ without changing their multiplicities.
Although identifying the degree of the cover is, strictly speaking, unnecessary,
the answer is straightforward and we have not seen it in the literature, so 
we pause to review the relevant combinatorics.  We begin with a
concrete example.

\begin{example}[Preimages and symmetries]\label{ex:pre-sym}
  Let $a,b,c,d$ and $e$ be distinct points in $\Xb$ and let $\bx =
  (a,b,a,c,d,c,e)$ be a $7$-tuple in $\Xb^7$ with multiset $M =
  a^2b^1c^2d^1e^1$, set partition $[\lambda] = 13|2|46|5|7$ and shape
  $\lambda = 2^21^3$ as in Example~\ref{ex:part-mult}. Note that the
  set partition $[\lambda]$ and the multiset $M$ do not uniquely
  determine $\bx$. The set partition determines the blocks of equal
  coordinates, the multiset determines the set of coordinates to be
  assigned, and coordinate multiplicities must match the block sizes,
  but any assignment sending $a$ and $c$ to blocks $13$ and $46$, and
  $b$, $d$ and $e$ to blocks $2$, $5$ and $7$ will work.  The
  possible choices arise from the $\size{\sym_{\lambda}} = 2!\cdot 3!
  = 12$ symmetries of the shape $\lambda$.  Next, the set
  $\mult^{-1}(M) \subset \Xb^7$ contains $\binom{n}{\lambda} =
  \binom{7}{2,1,2,1,1} = 1260$ $7$-tuples, since this multinomial
  coefficient counts the number of rearrangements of two $a$'s, one
  $b$, two $c$'s, one $d$ and one $e$.  Finally, since the map from
  $\mult^{-1}(M)$ to set partitions of shape $\lambda$ is uniformly a
  $12$-to-$1$ map, there are $\frac{1}{\size{\sym_\lambda}}
  \binom{n}{\lambda} = \frac{1260}{12}=105$ distinct set partitions
  with shape $\lambda = 2^21^3$.
\end{example}

\begin{defn}[Symmetries of a shape]\label{def:sym-shape}
  Let $\lambda = \lambda_1^{a_1} \lambda_2^{a_2} \cdots
  \lambda_\ell^{a_\ell}$ be an integer partition of $n$ of length $k$
  and let $\ba = (a_1,\ldots,a_\ell)$ be the exponents of $\lambda$.
  The \emph{symmetry group of $\lambda$} is a subgroup of $\sym_k$.
  When the parts of $\lambda$ are viewed as distinguishable, the group
  $\sym_\lambda$ is the group of rearrangements of its $k$
  distinguished parts that maintain their weakly decreasing order.
  Concretely, the group is $\sym_\lambda = \sym_\ba = \prod_{i=1}^\ell
  \sym_{a_i}$ and it has size $\size{\sym_\lambda} = \ba! = \prod_{i=1}^\ell
  a_i!$. We also write $\binom{n}{\lambda}$ to denote the multinomial
  coefficient $\binom{\lambda}{a_1,\ldots,a_l}$. 
\end{defn}

Arguing as in Example~\ref{ex:pre-sym} establishes the following three
results.

\begin{prop}[Tuples with fixed multiset]\label{prop:fixed-mult}
  For any $M \in \mult_n(\Xb)$ with $\shape(M) =\lambda$, there are
  $\binom{n}{\lambda}$ $n$-tuples $\bx \in \Xb^n$ with $\mult(\bx)=M$.
\end{prop}

\begin{prop}[Tuples with fixed multiset and set partition]\label{prop:tuple-reconstruct}
  Given a multiset $M \in \mult_n(\Xb)$ and a set partition $[\lambda]
  \vdash [n]$ with a common shape $\lambda = \shape(M) =
  \shape([\lambda])$, there are $\size{\sym_{\lambda}}$ $n$-tuples
  $\bx \in \Xb^n$ with $\mult(\bx)=M$ and $\spart(\bx) = [\lambda]$.
\end{prop}

\begin{prop}[Set partitions with fixed shape]\label{prop:fixed-shape}
  For any integer partition $\lambda \vdash n$, there are
  $\frac{1}{\size{\sym_\lambda}} \binom{n}{\lambda}$ set partitions
  $[\lambda] \vdash [n]$ with $\shape([\lambda])=\lambda$.
\end{prop}

\begin{proof}
  Since the horizontal maps of Figure~\ref{fig:part-mult} are onto for any 
  $\Xb$ considered here,
  let $M \in \mult_n(\Xb)$ be a multiset of shape $\lambda$.  By
  Proposition~\ref{prop:fixed-mult} there are $\binom{n}{\lambda}$
  $n$-tuples in $\Xb^n$ with multiset $M$ and by
  Proposition~\ref{prop:tuple-reconstruct} each set partition of shape
  $\lambda$ accounts for $\size{\sym_{\lambda}}$ of these tuples.  The
  quotient of these two values counts the number of set partitions of
  shape $\lambda$.
\end{proof}

With these counts, we can clarify the structure of the stratum
$\mult_\lambda$.  We start with a concrete example stated in the
language of intermediate covers.

\begin{figure}
  \begin{tikzcd}
    \Xb_{13|2|46|5|7} \arrow[r,equal,"\sim"]
    \arrow[d,twoheadrightarrow,"2!3!"']
    & \conf_5(\Xb) \arrow[d,twoheadrightarrow,"2!3!"] &
    \Xb_{[\lambda]} \arrow[r,equal,"\sim"]
    \arrow[d,twoheadrightarrow,"\ba!"]
    & \conf_k(\Xb) \arrow[d,twoheadrightarrow,"\ba!"]\\
    \mult_{2^21^3}(\Xb) \arrow[r,equal,"\sim"]
    & \set_{2,3}(\Xb) \arrow[d,twoheadrightarrow,"\binom{5}{2,3}"] &
    \mult_{\lambda}(\Xb) \arrow[r,equal,"\sim"]
    & \set_{\ba}(\Xb) \arrow[d,twoheadrightarrow,"\binom{k}{\ba}"]\\
    & \set_5(\Xb) &
    & \set_k(\Xb)
  \end{tikzcd}
  \caption{The stratum $\mult_\lambda(\Xb)$ in $\mult_n(\Xb)$ is
    homeomorphic to an intermediate cover of $\set_k(\Xb)$, where $k =
    \len(\lambda)$.\label{fig:strata-cover}}
\end{figure}

\begin{example}[Stratum covers]\label{ex:strata-cover}
  The stratum $\Xb_{13|2|46|5|7}$ in $X^7$ is homeomorphic to
  $\conf_5(\Xb)$, but the corresponding stratum $\mult_{2^21^3}$ in
  $\mult_7(\Xb)$ is not homeomorphic to $\set_7(\Xb)$.  To see this,
  note that each element of $\set_7(\Xb)$ has $7!$ preimages in
  $\conf_7(\Xb)$, while each element of $\mult_{2^21^3}(\Xb)$ has only
  $\size{\sym_{2^21^3}} = 2!3! = 12$ preimages in $\Xb^7$
  (Proposition~\ref{prop:tuple-reconstruct}).  Instead,
  $\mult_{2^21^3}(\Xb)$ is homeomorphic to the intermediate cover
  $\set_{2,3}(\Xb)$, which is a $\binom{5}{2,3} = 10$ sheeted cover of
  $\set_5(\Xb)$. And $\conf_5(\Xb)$ is a $2!3!=12$ sheeted cover of
  $\set_{2,3}(\Xb)$.  See Figure~\ref{fig:strata-cover}.
\end{example}

As Example~\ref{ex:strata-cover} shows, the multiset strata
$\mult_\lambda(\Xb)$ are, in general, intermediate covers between
$\conf_k(\Xb)$ and $\set_k(\Xb)$.

\begin{defn}[Intermediate covers]\label{def:intermediate-covers}
  Let $\ba = (a_1,\cdots,a_\ell)$ be a tuple of positive integers with
  sum $k$.  The \emph{intermediate cover} $\set_{\ba}(\Xb)$ is the
  space of subsets of $\Xb$ of size $k$, with elements labeled so that
  there are exactly $a_i$ points with label $i$ for each $i\in[\ell]$.
  The map that forgets the label is a covering map $\set_\ba(\Xb) \to
  \set_k(\Xb)$ of degree $\binom{k}{\ba} =
  \binom{k}{a_1\ \cdots\ a_\ell}$.  If $f\colon [k] \to [\ell]$ is a
  function with the property that $|f^{-1}(i)| = a_i$ for all $i \in
  [\ell]$, then $f$ induces a map $\conf_k(\Xb) \to \set_\ba(\Xb)$
  which labels coordinates by their image under $f$, and this is an
  $\ba!$-sheeted covering since permutations within preimages do not
  change the result.
\end{defn}

This gives us a fairly precise description of the $\mult$ map as a
stratified covering.

\begin{thm}[Stratified covering map]\label{thm:mult-strata-cover}
  The $\mult$ map is a stratified covering map.  Concretely, if 
  $[\lambda] \vdash [n]$ is a set partition of shape $\lambda =
  \lambda_1^{a_1}\cdots \lambda_\ell^{a_\ell}$ with $\len(\lambda)=k$
  and $\ba = (a_1,\ldots,a_\ell)$ are the exponents of $\lambda$, then
  the stratum $\Xb_{[\lambda]}$ is homeomorphic to $\conf_k(\Xb)$, the
  stratum $\mult_\lambda(\Xb)$ is homeomorphic to the intermediate
  cover $\set_{\ba}(\Xb)$, and the map $\mult$ restricts to a covering
  map $\Xb_{[\lambda]} \onto \mult_\lambda(\Xb)$ of degree
  $\size{\sym_\lambda} = \ba!$.
\end{thm}

\begin{proof}
  There is a natural homeomorphism from $\mult_\lambda(\Xb)$ to
  $\set_\ba(\Xb)$ which sends a multiset of shape $\lambda$ to its
  underlying set of size $k$, where each element of the set is labeled
  by its multiplicity in the multiset. Meanwhile, we know by
  Proposition~\ref{prop:strata-metrics} that $\Xb_{[\lambda]}$ is
  homeomorphic to $\conf_k(\Xb)$ once we fix a bijection $[k] \to
  [\lambda]$. Finally, if we use the same bijection and the map $[k]
  \to [\lambda] \to [\ell]$ to construct a covering map $\conf_k(\Xb)
  \to \set_\ba(\Xb)$ as in Example~\ref{ex:strata-cover}, then the
  homeomorphism and the covers form the commuting square in
  Figure~\ref{fig:strata-cover}.
\end{proof}

\begin{cor}\label{cor:mult-loc-onto}
  The map $\mult \colon \Xb^n \to \mult_n(\Xb)$ is locally onto.
\end{cor}

\begin{proof}
  Fix $\bx \in \Xb^n$ with multiset $M = \mult(\bx)$, set partition
  $[\lambda]=\spart(\bx)$, and shape $\lambda=\shape(\bx)$. The fact
  that $\Xb_{[\lambda]} \to \mult_\lambda(\Xb)$ is a covering map (Theorem~\ref{thm:mult-strata-cover}) is
  sufficient to show that the image of a neighborhood of $\bx \in
  \Xb_{[\lambda]}$ contains a neighborhood of $M$ in
  $\mult_\lambda(\Xb)$.  For the other strata near $\bx$ and $M$ we
  note that a small neighborhood of $\bx \in \Xb_{[\lambda]}$ contains
  points in $\Xb_{[\mu]}$ if and only if $[\mu] \leq [\lambda]$ in
  $\spart_n$ and a small neighborhood of $M \in \mult_\lambda(\Xb)$
  contains points in $\mult_\mu(\Xb)$ if and only if $\mu \leq
  \lambda$ in $\ipart_n$.  Since the set partitions below $[\lambda]$
  map onto the integer partitions below $\lambda$, the image of this
  neighborhood contains multisets in $\mult_\mu(\Xb)$ near $M$ for
  every $\mu \leq \lambda$.  In particular, the portions of every
  stratum of $\mult_n(\Xb)$ with $M$ in the boundary are covered by a
  portion of a stratum in $\Xb^n$ with $\bx$ in its boundary.
\end{proof}

\section{Critical Values and Polynomials}\label{sec:poly-ll-maps}
This section describes a stratification of monic centered polynomials that 
turns the Lyashko--Looijenga map into a stratified covering map.
After introducing our version of the $\LL$ map (\ref{subsec:ll-map})
we discuss polynomial strata (\ref{subsec:poly-strata}).

\subsection{The $\LL$ map}\label{subsec:ll-map}
While the classical $\LL$ map is from $\C^n$ to $\C^n$, the version used here is 
a map from $\poly_d^{mt}$, the space of monic
degree-$d$ polynomials up to translation equivalence, to
$\mult_n$, the space of multisets in $\C$ of size $n = d-1$.  The
various spaces and maps discussed are shown in
Figure~\ref{fig:ll-maps}.

\begin{figure}
  \begin{tikzcd}[column sep=small]
    \C^n \arrow[d,equal] \arrow[rr,hook]
    \arrow[rrrr,bend left=20,"\mathsc{LL}-\textrm{classic}"]
    & & \C^d \arrow[rr,two heads] \arrow[d,equal] && \C^n \arrow[d,equal] \\
    \poly_d^{mc} \arrow[rr,hookrightarrow] \arrow[rd,equal] &&
    \poly_d^m \arrow[rr,twoheadrightarrow]
    \arrow[dl,twoheadrightarrow,"\aff"'] \arrow[dr,"\cvl"] &&
    \poly_n^m \arrow[dl,equal] \\
    & \poly_d^{mt} \arrow[rr,rightarrow,"\mathsc{LL}"'] &&
        \mult_n 
  \end{tikzcd}
  \caption{The classical Lyashko--Looijenga map is a map from 
  $\C^n$ to $\C^n$ with $n=d-1$.  It can also be formulated as a map $\LL$ from 
  $\poly_d^{mt}$ to $\mult_n$.\label{fig:ll-maps}}
\end{figure}
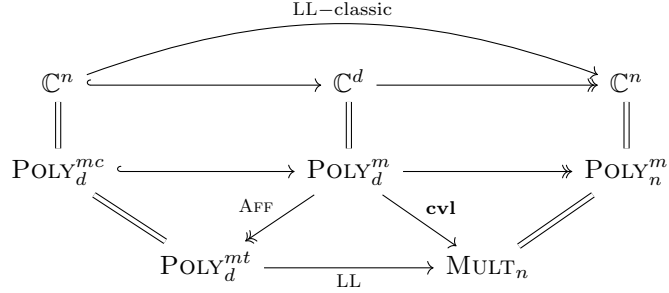

\begin{defn}[Critical value map]
  Let  $\mult_n = \mult_n(\C)$ be the space of
  $n$-element multisets in $\C$ (Definition~\ref{def:multisets}).  The 
  \emph{critical value map} $\cvl \colon \poly_d \to \mult_n$ 
  sends a degree-$d$ complex polynomial $p$ to its $n$-element multiset 
  $\cvl(p)$ of critical values in $\C$.
\end{defn}

In Figure~\ref{fig:ll-maps} the (diagonal) map $\cvl$ has been
restricted to polynomials that are monic.  Multisets in $\C$ are in
natural bijection with monic polynomials.

\begin{defn}[Multisets and monic polynomials]\label{def:mult-monic}
  There is a natural bijection between $\mult_n$ and $\poly_n^m$.
  The map $\poly \colon \mult_n \to \poly_n^m$ sends an
  $n$-element multiset $M \in \mult_n$ with $M = z_1^{m_1}
  z_2^{m_2}\cdots z_k^{m_k}$ and $\sum m_i = n$ to the polynomial
  $p=\poly(M)$ with $p(z) = (z-z_1)^{m_1} (z-z_2)^{m_2} \cdots
  (z-z_k)^{m_k}$.  The inverse map $\rts \colon \poly_n^m \to
  \mult_n$ sends a polynomial to its multiset of roots.
\end{defn}

Definition~\ref{def:mult-monic} is used in the lower righthand corner
of Figure~\ref{fig:ll-maps}.  Moreover, as discussed in
Definition~\ref{def:poly}, $\poly_d^m = \C^d$ and $\poly_d^{mc} =
\C^n$ using coefficients as coordinates, and these give the
identifications in Figure~\ref{fig:ll-maps} between the first and
second rows.  The classic version of the \emph{Lyashko--Looijenga map}
from $\C^n$ to $\C^n$ is shown in the top row.
Theorem~5.1.1 of \cite{LaZv04} summarizes its key properties.

\begin{thm}[$\LL$ map]\label{thm:ll-classic}
  The Lyashko--Looijenga map $\LL \colon \C^n \to \C^n$, from the space
  of monic centered degree-$d$ polynomials to the space of monic
  polynomials of degree $n=d-1$, is a polynomial finite map of degree
  $d^{d-2}$.
\end{thm}

These properties make it well-suited for investigations by
algebraic geometers.  Here is a brief definition of the map and its
key properties.

\begin{defn}[$\LL$ map, classic version]
  The classic version of the \emph{Lyashko--Looijenga map} sends a 
  monic centered polynomial $p$ of degree~$d$ to the monic polynomial of 
  degree $n=d-1$ whose roots are the critical values of $p$.  The
  coefficients of the polynomials in both the domain and range can be
  used as coordinates to give an induced map from $\C^n$ to
  $\C^n$. See Figure~\ref{fig:ll-maps}. It is a \emph{polynomial} map
  because the coordinates in the range are defined by multivariable
  polynomial functions of the coordinates in the domain, and it is a
  \emph{finite map of degree $d^{d-2}$} because a
  generic point in the range has exactly $d^{d-2}$ preimages.
\end{defn}

The bottom row of Figure~\ref{fig:ll-maps} uses polynomials with an
affine domain.

\begin{defn}[Affine map]
  We say that $p, q\colon \C \to \C$ are \emph{equivalent up to
  translation} if there is a constant $b \in \C$ such that $q(z) =
  p(z+b)$.  This is an equivalence relation; let $[p] = \{q \mid
  q(z) = p(z+b), b\in \C\}$ be the \emph{translation equivalence
  class} of $p$. If $\poly_d^{mt}$ denotes monic polynomials of
  degree $d$ up to translation, then there is a quotient map
  $\aff\colon \poly_d^m \onto \poly_d^{mt}$ that sends $p$ to $[p]$.
  We call this the \emph{affine map} since the domain of each
  polynomial $p$ becomes, in essence, an affine space with no fixed
  origin, instead of a $1$-dimensional complex vector space.
\end{defn}

Polynomials up to translation are equivalent to polynomials that are
centered.

\begin{rem}[Centering polynomials]\label{rem:center-trans}
  For monic polynomials, the coefficient of the term just below the
  leading term is the negative of the sum of its roots.  In
  particular, $c_1 = 0$ and $p$ is centered if and only if the average
  of the roots is at the origin.  Under precomposition with a
  translation, the average of the roots is translated, and every
  equivalence class $[p]$ contains a unique representative that is
  centered.  The map $\poly_d^{mt} \to \poly_d^{mc}$ sending $[p]$ to
  its centered representative is a section of the affine quotient map
  $\aff \colon \poly_d^m \onto \poly_d^{mt}$ and a homeomorphism.
\end{rem}

Translation equivalence is a special case of a more general situation.

\begin{rem}[Composing with linear functions]\label{rem:linear-comp}
  Let $f(z) = az+b$ be an invertible linear transformation, let $g(z)
  = \frac{1}{a}(z-b)$ be its inverse, and recall that these maps are
  Euclidean similarities of the plane: translating, dilating, and
  rotating.  Postcomposing $p\colon \C \to \C$ with $f$ applies the
  similarity $f$ to the coordinate system in the range.  Concretely,
  if $q(z) = f(p(z)) = a\cdot p(z) + b$, then $\cpt(q)=\cpt(p)$ and
  $\cvl(q) = f(\cvl(p)) = a\cdot \cvl(p) + b$.  Precomposing $p$ with
  $f$ applies the inverse similarity $g$ to the coordinate system in
  the domain.  If $q(z) = p(f(z)) = p(az+b)$, then $\cpt(q) =
  g(\cpt(p)) = \frac{1}{a}(\cpt(p)-b)$ and $\cvl(q) = \cvl(p)$. 
\end{rem}

This leads to the factorization of the critical value map shown in
Figure~\ref{fig:ll-maps}.

\begin{defn}[$\LL$ map]
  By Remark~\ref{rem:linear-comp}, all of the polynomials in $[p]$
  have the same multiset of critical values, so the critical value map
  $\cvl\colon \poly_d^m \to \mult_n$ factors through the affine
  map $\aff \colon \poly_d^m \onto \poly_d^{mt}$ sending $p \mapsto
  [p]$.  Our version of the $\LL$ map is the induced map $\LL \colon
  \poly_d^{mt} \to \mult_n$.  Because of the identifications on
  the left and right sides of Figure~\ref{fig:ll-maps}, this version
  of the $\LL$ map has many of the properties listed in
  Theorem~\ref{thm:ll-classic}, such as finitely many point preimages
  bounded above by $d^{d-2}$.
\end{defn}

Of particular interest here are restrictions the $\LL$ map to polynomials with
critical values is a specific portion of the range.

\begin{defn}[Restricted maps]\label{def:restricted}
  For any subspace $\Xb \subset \C$ (Remark~\ref{rem:spaces}), let 
  $\poly_d(\Xb)$ be the collection of polynomials whose critical values 
  lie in $\Xb$, and we use superscripts to restrict attention to those that 
  are monic ($m$), centered ($c$) or only considered up to translation ($t$).  
  For any such $\Xb$, the lower portion of Figure~\ref{fig:ll-maps} can be
  restricted in this subspace resulting in a \emph{restricted
  $\LL$ map} $\LL \colon \poly_d^{mt}(\Xb) \to \mult_n(\Xb)$. 
\end{defn}

\subsection{Polynomial strata}\label{subsec:poly-strata}
The stratification of multisets (Section~\ref{sec:prod-mult}) leads to a double
stratification of polynomials based on critical point shape,
critical value shape, and the arrow between them in the integer 
partition acyclic category.

\begin{defn}[Polynomial strata]\label{def:poly-strata}
  The space $\poly_d$ has a double stratification by critical point
  shape and critical value shape.  For any polynomial $p$ of degree
  $d$, its \emph{critical point shape} is $\lambda = \lambda(p) =
  \shape(\cpt(p))$, its \emph{critical value shape} is $\mu = \mu(p) =
  \shape(\cvl(p))$, and the map from $\cpt(p)$ to $\cvl(p)$ determines an
  arrow $\lambda(p) \stackrel{p}{\to} \mu(p)$ in the acyclic category 
  $\ipart_n$ (Definition~\ref{def:ordering-int-part}).  Let
  $\poly_{\lambda\to\mu} = \{ p \in \poly_d \mid \lambda(p) 
  \stackrel{p}{\to} \mu(p) \textrm{ is } \lambda\to \mu \}$ be the ``preimage'' 
  of an arrow in the acyclic category $\ipart_n$.
  The double stratification is 
  \[
  \poly_d = \bigsqcup_{\substack{\lambda,\mu \vdash n\\ \lambda \to
      \mu}} \poly_{\lambda\to\mu}.
  \]
\end{defn}

\begin{rem}[Extreme preimages]\label{rem:ll-extremes}
  If $p$ is a monic polynomial with an indiscrete critical value
  multiset, then it has an indiscrete critical point multiset
  (Lemma~\ref{lem:disk-preimage}).  In particular, the unique preimage of
  $\cvl(p) = \{c^n\}$ under the $\LL$ map is $[z^d+c] \in
  \poly_d^{mt}$.  At the other extreme, a discrete multiset with $n$
  distinct critical values is a generic point with $d^{d-2}$ preimages
  (Theorem~\ref{thm:ll-classic}).
\end{rem}

The double stratification in Definition~\ref{def:poly-strata} factors
through the affine map.

\begin{defn}[Strata up to translation]
  For a fixed polynomial $p$, its critical values are invariant under
  translation (Remark~\ref{rem:linear-comp}), so the critical value
  shape $\mu(p)$ is a function of the translation equivalence class $[p]$.  
  Its critical points are moved under translation, but this leaves their shape invariant,
  so the critical point shape $\lambda(p)$ is also a function of $[p]$.  In
  particular, the double stratification of $\poly_d^m$ factors through
  the affine map to give a double stratification of $\poly_d^{mt}$
  with $\poly_{\lambda\to\mu}^{mt} = \aff(\poly_{\lambda\to\mu}^m)$.
\end{defn}

The double stratification is needed because of accidental equalities.

\begin{defn}[Accidental equalities]\label{def:accidents}
  When $p$ is a polynomial with $\lambda(p) < \mu(p)$, there are distinct
  critical points $z_1 \neq z_2 \in \cpt(p)$ with equal critical values
  $p(z_1)=p(z_2) \in \cvl(p)$, and we say that $p$
  has \emph{accidental equalities}.
\end{defn}

The Chebyshev polynomial of the first kind, defined by the equation
$T_d(\cos(\theta)) = \cos(d\cdot \theta)$, is an extreme example of this 
phenomenon.  It has $d$ distinct roots, all real, and $n=d-1$ distinct real 
critical points.  Every critical value, on the other hand, is either $1$ or 
$-1$, so for $d$ at least $4$, $T_d(z)$ has accidental equalities. Such 
accidental equalities, however, have a limited impact on path lifting.

The $\LL$ map, like the $\mult$ map, is a stratified covering map. 

\begin{thm}[Stratified covering map]\label{thm:poly-strata-cover}
  For all integer partitions $\lambda, \mu \vdash n$ with $\lambda
  \to \mu$, the restricted map $\LL\colon \poly_{\lambda\to\mu}^{mt}
  \to \mult_\mu(\C)$ is a covering map.
\end{thm}

Theorem~\ref{thm:poly-strata-cover} is a consequence of
\cite[Theorem~5.2.11]{LaZv04} or \cite[Theorem~B]{DoMc-20}, but in
both cases a certain amount of interpretive work is necessary.  A proof based on \cite{DoMc-20} is
included in Appendix~\ref{app:poly-strata-cover}.

\begin{rem}[Constant finite point preimages]\label{rem:pt-preimages}
  Theorem~\ref{thm:poly-strata-cover} is the polynomial analogue of 
  Theorem~\ref{thm:mult-strata-cover}, but with less detail about the degrees of the covers. The degrees of the covers and a polynomial analogue of 
  Proposition~\ref{prop:fixed-shape} has been established by Zvonkine \cite{zvonkine97}.
  In particular, the restricted map $\LL \colon
  \poly^{mt}_{\lambda\to\mu}(\C) \to \mult_\mu(\C)$ has finite point
  preimages of constant size, and this constant can be derived from the 
  combinatorics of $\lambda \to \mu$.  See also \cite[Theorem~5.2.2]{LaZv04}.
\end{rem}

\begin{rem}[Empty strata]\label{rem:empty-strata}
  For some arrows $\lambda \to \mu$ the subspace
  $\poly_{\lambda\to\mu}^{mt}$ is empty.  If $p$ is a polynomial where
  $\mu(p)$ is the indiscrete partition, for example, then
  $\lambda(p)$ is also indiscrete (Remark~\ref{rem:ll-extremes}).
  Thus $\poly_{\lambda\to\mu}^{mt}$ is empty whenever $\mu$ is
  indiscrete and $\lambda \neq \mu$. Note that
  Theorem~\ref{thm:poly-strata-cover} remains vacuously true in this
  case since $\poly_{\lambda\to\mu}^{mt}$ is a $0$-sheeted cover of
  $\mult_\mu(\C)$.
\end{rem}

The critical value map and the $\LL$ map are locally onto.

\begin{lem}[Locally onto]\label{lem:local-onto}
  The map $\cvl\colon \poly_d^m \to \mult_n(\C)$ is locally onto.
  As a consequence, the map $\LL \colon \poly_d^{mt} \to
  \mult_n(\C)$ is also locally onto.
\end{lem}

\begin{proof}
  The main result of \cite{beardon02} is that the map from complex
  polynomials to critical values is onto.  In fact, their proof easily
  extends to show that it is locally onto in the sense that for any
  polynomial $p$ and for any neighborhood $N$ of $p$ in $\poly_d^m$
  the image $\cvl(N)$ contains a neighborhood of $\cvl(p)$.  As a
  factor of the $\cvl$-map, the $\LL$ map inherits this property.
\end{proof}

The stratified product, multiset and polynomial spaces $\C^n$, $\mult_n(\C)$, and $\poly_d^{mt}(\C)$, 
and the stratified maps $\mult$ and $\LL$ connecting them, give all three spaces a 
stratified Euclidean metric. 

\begin{defn}[Stratified Euclidean metrics]\label{def:strata-euclid-metric}
  The natural Euclidean product metric on $\C^n$ has a Euclidean metric 
  on each of its strata (Proposition~\ref{prop:strata-metrics}), and since the $\mult$ map 
  is a stratified cover defined as the quotient by the isometric symmetric group action, 
  there is a unique metric on $\mult_n(\C)$ that restricts to a local isometry 
  for each covering map in the stratification.  This is the \emph{stratified Euclidean metric on 
  $\mult_n(\C)$}. Simlilarly, the $\LL$ map is a stratified cover and there is a unique metric on $\poly_d^{mt}(\C)$ 
  that restricts to a local isometry for each covering map in its stratification.  This pulls the stratified Euclidean metric 
  on $\mult_n(\C)$ back to $\poly_d^{mt}(\C)$, and is the \emph{stratified Euclidean metric on $\poly_d^{mt}(\C)$}.  
  Since the $\LL$ map and the $\mult$ map use the same stratification in their common range, we have local 
  isometries between strata neighborhoods of polynomials, multisets and $n$-tuples. Concretely, if $p$ is a 
  monic centered polynomial and $\bx$ is an $n$-tuple with common multiset $M=\cvl(p) = \mult(\bx)$, then 
  there are small neighborhoods in the appropriate strata and local isometries so that $N_p \cong N_M \cong N_\bx$.
\end{defn}

Some cell structures on multiset spaces can also be lifted through the $\LL$ map.

\begin{defn}[Stratified cell structures]
    \label{def:stratified-cell-structures}
  A cell structure on $\mult_n(\Xb)$ is a \emph{stratified cell structure} 
  if all points in the same open cell have the same shape.  In other words, the 
  stratification of $\mult_n(\Xb)$ into open cells is a refinement of the 
  stratification $\mult_n(\Xb) = \bigsqcup_{\lambda\vdash n} \mult_\lambda(\Xb)$
  by the shape of these $n$-element multisets.  The open cells in a stratified 
  cell structure on $\mult_n(\Xb)$ lift through the individual covering maps of the 
  stratified $\LL$ map to provide the open cells of a cell structure on $\poly_d^{mt}(\Xb)$.
  And with this lifted cell structure on the domain, the $\LL$ map becomes a cellular map.
\end{defn}

\section{Unique Path Lifting}\label{sec:path-lift}
In this section we show that every path in $\mult_n(\C)$ lifts through the $\mult$ map and through 
the $\LL$ map, and this lift is unique when the shape of the path is weakly
increasing.  We focus first on the $\mult$ map (\ref{subsec:mult-path-lift}) and then
the $\LL$ map (\ref{subsec:poly-path-lift}).

\subsection{Path-lifting through the $\mult$ map}\label{subsec:mult-path-lift}
Lifting multiset paths to $\C^n$ is a simpler situation than lifting
them through the $\LL$ map, which is why they are being considered
first. The diagrams for the two situations are shown in
Figure~\ref{fig:mult-path-lift}. 

\begin{figure}
  \begin{tikzcd}
    \{0\} \arrow[d,hookrightarrow]
    \arrow[r,rightarrow,"\widetilde \alpha"] &
    \C^n \arrow[d, "\mult"] & 
    \{0\} \arrow[d,hookrightarrow]
    \arrow[r,rightarrow,"\widetilde \alpha"] &
    \poly_d^{mt}(\C) \arrow[d, "\LL"] \\
    {[0,1]} \arrow[r,"\alpha"]
    \arrow[ru,dashed,"\exists ! \widetilde \alpha"]
    \arrow[rd,"\geq"'] &
    \mult_n(\C) \arrow[d,"\shape"] & 
    {[0,1]} \arrow[r,"\alpha"]
    \arrow[ru,dashed,"\exists ! \widetilde \alpha"]
    \arrow[rd,"\geq"'] &
    \mult_n(\C) \arrow[d,"\shape"] \\
    & \ipart_n && \ipart_n
  \end{tikzcd}
  \caption{Paths can be uniquely lifted through the $\mult$ map (left)
    and the $\LL$ map (right) when the multiset path $\alpha$ has a
    weakly increasing shape. \label{fig:mult-path-lift}}
\end{figure}
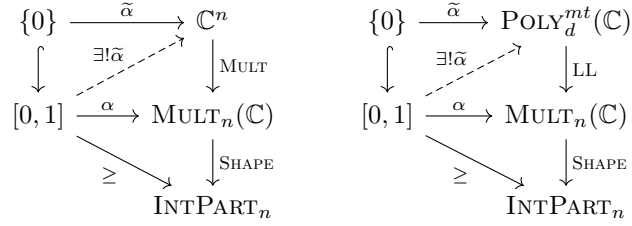

\begin{defn}[Shapes and Multiset Paths]\label{def:shape-mult}
  Let $\alpha \colon [0,1] \to \mult_n(\C)$ be a path and note that
  since $\alpha$ is continuous, the points in the set underlying the
  multiset move around continuously in $\C$. For each $t \in [0,1]$,
  the multiset $\alpha(t)$ has a shape $\lambda(t) =
  \shape(\alpha(t))$ in $\ipart_n$ that records its multiset of
  multiplicities.  We say that $\alpha$ has a \emph{weakly increasing
  shape} if for all $s \leq t$ in $[0,1]$, we have $\lambda(s) \leq \lambda(t)$ in
  $\ipart_n$.  Equivalently, the points (with multiplicity) moving around in 
  $\C$ are allowed to merge but not split,
  so that the size of the underlying set weakly decreases over time as
  the integer partition weakly increases. It has a \emph{weakly
  decreasing shape} if $s \leq t$ in $[0,1]$ implies $\lambda(s) \geq \lambda(t)$
  in $\ipart_n$, or equivalently, the
  underlying points are allowed to split but not merge.  For any path
  $\alpha \colon [0,1] \to Y$, the \emph{time reversal of $\alpha$} is
  $\alpha^\rev \colon [0,1] \to Y$ with $\alpha^\rev(1-t) =
  \alpha(t)$.  Note that $\alpha$ has a weakly increasing/decreasing
  shape if and only if $\alpha^\rev$ has a weakly
  decreasing/increasing shape.  Paths that remain in a single stratum
  $\mult_\lambda(\C)$ have a \emph{constant shape} $\lambda(t) =
  \lambda$ for all $t\in [0,1]$.
\end{defn}

We begin with an easy observation.

\begin{lem}[Paths and Strata]\label{lem:mult-paths-strata}
  Let $\alpha \colon [0,1] \to \mult_n(\C)$ be a multiset path and
  let $\beta \colon [0,1] \to \C^n$ be a path with $\alpha = \mult
  \circ \beta$.  If $\alpha(t)$ has constant shape $\mu$ and
  $\spart(\beta(0)) = [\mu]$, then $\alpha$ is a path in
  $\mult_\mu(\C)$ and $\beta$ is a path in $\C_{[\mu]}$.
\end{lem}

\begin{proof}
  Since $\alpha$ has constant shape, it is a path in $\mult_\mu(\C)$.
  View the path $\alpha$ as points with multiplicities moving around
  $\C$ without splitting or merging, and view $[\mu]=\spart(\beta(0))$
  as a way of replacing these multiplicities at the start point with
  initial labels.  The unique continuous way to lift $\alpha$ to a path
  $\beta$ in $\C^n$ is to drag these initial labels along as the
  points move, and this unique lift stays in $\C_{[\mu]}$.
\end{proof}

The main result in this section is that path liftings always exist and
the lifting is unique when the path to be lifted has a weakly
increasing shape.

\begin{thm}[Lifting paths to $\C^n$]\label{thm:mult-map-unique}
  For any path $\alpha \colon [0,1] \to \mult_n(\C)$ and for any lift
  of the start point $\alpha(0)$ to an $n$-tuple $\widetilde \alpha(0)
  \in \C^n$ with $\mult(\widetilde \alpha(0)) = \alpha(0)$, there
  always exists a lifted path $\widetilde \alpha \colon [0,1] \to
  \C^n$ with $\alpha = \widetilde \alpha \circ \mult$, and the lifted
  path $\widetilde \alpha$ is unique when $\alpha$ has weakly
  increasing shape.
\end{thm}

\begin{proof}
  Since the map $\mult$ is a locally onto
  (Corollary~\ref{cor:mult-loc-onto}), every path in the range has at
  least one lift to the domain.  Let $\beta_1$ and $\beta_2$ be two
  lifts of $\alpha$ and recall that the set $\{t \mid
  \beta_1(t)=\beta_2(t)\}$ where they agree is always a closed subset
  of $[0,1]$.  If $\alpha$ has constant shape $\mu$, then by
  Lemma~\ref{lem:mult-paths-strata}, $\alpha$ is a path in
  $\mult_\mu(\C)$ and both $\beta_1$ and $\beta_2$ are paths in
  $\C_{[\mu]}$.  Since the restricted map $\mult \colon \C_{[\mu]} \to
  \mult_\mu(\C)$ is a covering map (Theorem~\ref{thm:mult-strata-cover}),
  unique path lifting for covers shows that $\beta_1$ and $\beta_2$
  agree for all $t \in [0,1]$ and $\alpha$ has a unique lift in this
  case.  To prove this when $\alpha$ has a weakly increasing but
  non-constant shape, it is sufficient to consider $\alpha$ where
  $\mu(t) = \mu(0)$ for all $0\leq t<1$ and $\mu(0) < \mu(1)$.  More
  complicated weakly increasing paths are concatenations of finitely
  many paths of this restricted type.  Arguing as above for the
  subpaths restricted to $[0,t]$, we find that $\beta_1(t) =
  \beta_2(t)$ for all $t <1$.  But the portion on which they agree is
  closed so $\beta_1(1)$ and $\beta_2(1)$ are also equal.
\end{proof}

For the remainder of the section, we examine what can go wrong for
paths that do not have weakly increasing shape. To start, paths 
with weakly decreasing shape can have multiple lifts, but only finitely many.

\begin{cor}[Finitely many lifts I]\label{cor:mult-map-finite}
  Let $\alpha \colon [0,1] \to \mult_n(\C)$ be a path and let
  $\widetilde \alpha(0) \in \C^n$ be a lift of its start point
  $\alpha(0)$.  If $\alpha$ has a weakly descreasing shape, then there
  exist at most $n!$ lifted paths $\widetilde \alpha \colon [0,1] \to
  \C^n$ with $\alpha = \mult \circ \widetilde \alpha$.
\end{cor}

\begin{proof}
  Since $\alpha$ has a weakly decreasing shape, its time reversal
  $\alpha^\rev$ has a weakly increasing shape.  By
  Theorem~\ref{thm:mult-map-unique}, for every preimage of the
  endpoint $\alpha(1)$ in $\C^n$, there is a unique lifted path
  $\widetilde \alpha^\rev$ that starts at this preimage.  It must end
  at one of the preimages of $\alpha(0)$ and this defines a function
  from the preimages of $\alpha(1)$ to the preimages of $\alpha(0)$.
  The number of preimages of $\alpha(1)$ sent to a fixed preimage
  $\widetilde \alpha(0)$ of $\alpha(0)$ is equal to the number of
  lifts of $\alpha$ starting at $\widetilde \alpha(0)$.  Since points
  in $\mult_n(\C)$ have at most $n!$ preimages in $\C^n$, $\alpha$ has
  at most $n!$ lifts starting at~$\widetilde \alpha(0)$.
\end{proof}

The upper bound of $n!$ occurs when the path goes from one extreme to the
other.

\begin{exmp}[Maximal number of lifts I]\label{ex:mult-map-finite}
  Let $\alpha\colon [0,1] \to \mult_n(\C)$ be any path that has a
  weakly decreasing shape which starts at an indiscrete multiset
  (with $1$ point of multiplicity $n$) and ends at a discrete multiset
  (with $n$ points of multiplicity $1$).  The start point $\alpha(0)$
  has a unique preimage in $\C^n$ and the endpoint $\alpha(1)$ has
  $n!$ preimages in $\C^n$.  The time-reversed path $\alpha^\rev$ has a
  unique lift that starts at each of the preimages of $\alpha(1)$ and
  these lifts all end at the unique preimage of $\alpha(0)$.  In
  particular, $\alpha$ has $n!$ distinct lifts that start at the
  unique preimage of $\alpha(0)$.
\end{exmp}

Finally, paths that are not weakly monotonic can have uncountably many lifts.

\begin{example}[Uncountable lifts I]\label{ex:mult-map-many}
  For $n=2$ and $\R \subset \C$, we give an explicit example of a
  path $\alpha \colon [0,1] \to \mult_2(\R) \subset \mult_2(\C)$ that
  has uncountably many lifts to $\R^2 \subset \C^2$. The space
  $\mult_2(\R)$ is the closed half-plane obtained by folding the plane
  $\R^2$ across the line $y=x$, and we identify $\mult_2(\R)$ with the
  lower half-plane $\{ (x,y) \mid x \geq y\}$.  Every path in the
  plane projects to a path in the half-plane by composing with the map
  $\mult$, and every path in the half-plane has at least one lift to
  the plane with a given lift of its starting point, but there is a
  choice of lift whenever the path in the half-plane leaves the
  diagonal boundary line.  Let $s\colon [0,1] \to \R$ be the continuous
  function
  \[
  s(t) = \left\{ \begin{array}{cl} t \sin(\pi/t)
    & \textrm{ if $t\neq 0$ } \\ 0 & \textrm{ if $t=0$}
  \end{array}\right.
  \]
  and let $s_t$ be a shorthand for $s(t)$.  Note that $s^{-1}(0) = \{t
  \in I \mid s_t = 0\} $ is the set $\{0\} \cup \{\frac1k \mid k \in
  \N\}$, and that the complement $I - s^{-1}(0)$ is a countable union
  of open intervals.  Next, let $\beta \colon [0,1] \to \R^2$ be the
  path where $\beta(t) = (s_t,-s_t)$ and let $\alpha$ be its
  projection: $\alpha = \mult \circ \beta$.  The image of $\beta$ is
  contained in the line $x+y=0$ of slope $-1$, and it is on the
  diagonal $x=y$ exactly for $t \in s^{-1}(0)$.  In the half-plane
  $\mult_2(\R)$ we have $\alpha(t) = (\size{s_t},-\size{s_t})$.  The
  path $\beta$ is one lift of $\alpha$ that starts at the point
  $(0,0)$ but there many others.  For each open time interval between
  consecutive points where $\alpha(t)$ is in the boundary, there are
  two possible lifts, one above the diagonal line $y=x$ and one below.
  And since there are countably many such intervals, there are
  $2^{\aleph_0}$ possible lifts.
\end{example}

\subsection{Path-lifting through the $\LL$ map}\label{subsec:poly-path-lift}
We now prove analogous results for path-lifting through the $\LL$ map.  We start 
by establishing the analog of Lemma~\ref{lem:mult-paths-strata} using the metric complexes 
constructed in Part~\ref{part:single-poly}.  Let $\closedsquare$ be a large rectangle 
in the range of $p$ with $\cvl(p) \subset \inter(\closedsquare)$, and let $\Pb'_p$ be the 
corresponding metric rectangular critical point complex (Definition~\ref{def:br-coord-cplx}) 
with its $\cat(0)$ metric (Remark~\ref{rem:unique-geod}).

\begin{lem}[Distances]\label{lem:distances}
  If $z_1, z_2$ are points in $\Pb'_p$ and the interior of the unique
  geodesic between them is disjoint from $\cpt(p)$, then the distance
  $d_{\Pb'_p}(z_1,z_2)$ between them in $\Pb'_p$ is equal to the
  distance $d_{\Qb'_p}(p(z_1),p(z_2))$ between their images in $\Qb'_p$.
\end{lem}

\begin{proof}
  By Remark~\ref{rem:unique-geod}, $p$ is a local isometry from the
  $\cat(0)$ metric space $\Pb'_p$ to the $\cat(0)$ metric space $\Qb'_p$
  except at the points in $\cpt(p)$.  The unique geodesic from $z_1$
  to $z_2$ is a local geodesic in $\Pb'_p$ and, because it has no
  critical points in its interior, it is sent to a local geodesic in
  the rectangular complex $\Qb'_p$. A local geodesic in a Euclidean
  rectangle, however, is the unique geodesic between its endpoints.
\end{proof}

\begin{cor}[Points and values]\label{cor:pts-vls}
  If the minimum distance between distinct critical points in $\Pb'_p$
  is $\epsilon$, then the minimum distance between distinct critical
  values in $\Qb'_p$ is at most $\epsilon$.  In particular, if the
  distinct critical values in $\Qb'_p$ are at least
  $\epsilon$-separated, the distinct critical points in $\Pb'_p$ are at
  least $\epsilon$-separated.
\end{cor}

\begin{proof}
  If $z_1, z_2 \in \cpt(p)$ are distinct critical points that realize
  the minimal distance $\epsilon$, then there can be no critical point
  in the interior of the unique geodesic connecting them and, by
  Lemma~\ref{lem:distances}, the distance between $p(z_1), p(z_2) \in
  \cvl(p)$ is exactly~$\epsilon$.
\end{proof}

Note that the implications in Lemma~\ref{lem:distances} and
Corollary~\ref{cor:pts-vls} are not reversible. The Chebyshev polynomial 
$T_d$ with $d\geq 4$ has distinct critical points with equal critical values.  
A small perturbation of $T_d$, breaking an accidental equality, has 
critical points that remain $\epsilon$-separated, but critical values that 
are not.  Corollary~\ref{cor:pts-vls} can be reframed as an assertion about
paths and strata.

\begin{lem}[Paths and strata]\label{lem:poly-paths-strata}
  Let $\alpha \colon [0,1] \to \mult_n(\C)$ be a multiset path and
  let $\beta \colon [0,1] \to \poly_d^{mt}$ be a lift of $\alpha$
  through the $\LL$ map, so that $\alpha = \LL \circ \beta$.  If
  $\alpha(t)$ has a constant shape $\mu$ and $\beta(0)$ has shape
  $\lambda \to \mu$, then $\alpha$ is a path in the stratum
  $\mult_\mu(\C)$ and $\beta$ is a path in the stratum
  $\poly_{\lambda\to\mu}^{mt}$.
\end{lem}

\begin{proof}
  That $\alpha$ is a path in $\mult_\mu(\C)$ is clear from its
  constant shape (Lemma~\ref{lem:mult-paths-strata}).  
  Moreover, since the interval $[0,1]$ is compact and
  the points with multiplicity in $\alpha(t)$ neither merge nor split,
  there exists an $\epsilon>0$ such that the points with multiplicity
  in $\alpha(t)$ are $\epsilon$-separated for every $t$.  By
  Corollary~\ref{cor:pts-vls}, the distinct critical points of
  $\beta(t)$ are also $\epsilon$-separated for every $t$.  In
  particular, $\beta$ must have constant critical point shape, and
  this shape must be $\shape(\beta(0)) = \lambda$.
\end{proof}

We now proceed as before.  Any multiset path $\alpha$ with weakly 
increasing shape can be subdivided into finitely many subintervals 
where the shape is constant, and by Lemma~\ref{lem:poly-paths-strata}, 
the only possible lifts of these subintervals through the $\LL$ map occur 
in specific polynomial strata. The main result in this section is a polynomial version of
Theorem~\ref{thm:mult-map-unique}.

\begin{thm}[Lifting paths to $\poly_d^{mt}(\C)$]\label{thm:ll-map-unique}
  For any path $\alpha \colon [0,1] \to \mult_n(\C)$ and for any lift
  of $\alpha(0)$ to $\widetilde \alpha(0) =[p] \in \poly_d^{mt}(\C)$
  with $\LL(\widetilde \alpha(0)) = \cvl(p) = \alpha(0)$, there always
  exists a lifted path $\widetilde \alpha \colon [0,1] \to
  \poly_d^{mt}(\C)$ with $\alpha = \widetilde \alpha \circ \LL$, and
  the lifted path $\widetilde \alpha$ is unique when $\alpha$ has
  weakly increasing shape.
\end{thm}

\begin{proof}
  Since the $\LL$ map is locally onto (Lemma~\ref{lem:local-onto}),
  every path in the range has at least one lift to the domain.  Let
  $\beta_1$ and $\beta_2$ be two lifts of $\alpha$ and recall that the
  set $\{t \mid \beta_1(t)=\beta_2(t)\}$ where they agree is always a
  closed subset of $[0,1]$.  If $\alpha$ has constant shape $\mu$,
  then by Lemma~\ref{lem:poly-paths-strata}, $\alpha$ is a path in
  $\mult_\mu(\C)$ and both $\beta_1$ and $\beta_2$ are paths in
  $\poly_{\lambda\to\mu}^{mt}$.  Since the restricted map $\LL \colon
  \poly_{\lambda\to\mu}^{mt} \to \mult_\mu(\C)$ is a covering map
  (Theorem~\ref{thm:poly-strata-cover}), unique path lifting for
  covers shows that $\beta_1$ and $\beta_2$ agree for all $t \in
  [0,1]$ and $\alpha$ has a unique lift in this case. To prove this
  when $\alpha$ has a weakly increasing but non-constant shape, it is
  sufficient to consider $\alpha$ where $\mu(t) = \mu(0)$ for all
  $0\leq t<1$ and $\mu(0) < \mu(1)$. More complicated weakly
  increasing paths are concatenations of finitely many paths of this
  restricted type.  Arguing as above for the subpaths restricted to
  $[0,t]$, we find that $\beta_1(t) = \beta_2(t)$ for all $t <1$.  But
  the portion on which they agree is closed so $\beta_1(1)$ and
  $\beta_2(1)$ are also equal.
\end{proof}

The consequences of Theorem~\ref{thm:ll-map-unique} are similar to
those of Theorem~\ref{thm:mult-map-unique}.  Here are polynomial
versions of Corollary~\ref{cor:mult-map-finite},
Example~\ref{ex:mult-map-finite}, and Example~\ref{ex:mult-map-many}.

\begin{cor}[Finitely many lifts II]\label{cor:ll-map-finite}
  Let $\alpha \colon [0,1] \to \mult_n(\C)$ be a path and let
  $\widetilde \alpha(0)  \in \poly_d^{mt}(\C)$ be a lift of its
  start point.  If $\alpha$ has a weakly decreasing shape, then there
  are at most $d^{d-2}$ lifted paths $\widetilde \alpha \colon [0,1]
  \to \poly_d^{mt}(\C)$ with $\alpha = \LL \circ \widetilde \alpha$.
\end{cor}

The proof is similar to that of Corollary~\ref{cor:mult-map-finite}
and has been omitted.

\begin{exmp}[Maximal number of lifts II]\label{ex:ll-map-finite}
  Let $\alpha\colon [0,1] \to \mult_n(\C)$ be a path with a weakly
  decreasing shape that starts at the indiscrete multiset 
  and ends at a discrete multiset.  By Remark~\ref{rem:ll-extremes},
  the start point has a unique preimage and the endpoint has $d^{d-2}$
  preimages in $\poly_d^{mt}(\C)$.  Arguing as in
  Example~\ref{ex:mult-map-finite} shows that $\alpha$ has exactly
  $d^{d-2}$ lifts starting at the unique lift of the start point.
\end{exmp}

\begin{example}[Uncountable lifts II]\label{ex:ll-map-many}
  Let $\beta\colon [0,1] \to \poly_3(\C)$ be a continuous path in the
  space of monic cubic polynomials defined by setting $\beta(t) = p_t$
  where $p_t(z) = z^3 - \frac32 s_t z^2 + 1$ and $s_t$ is as defined
  in Example~\ref{ex:mult-map-many}.  Let $\alpha\colon [0,1] \to
  \mult_2(\C)$ be defined by the composition $\alpha = \cvl \circ
  \beta$.  Concretely, the derivative $p'_t(z)$ factors as $3z(z-
  s_t)$, $\cpt(p_t) = \{0,s_t\}$ and $\alpha(t) = \cvl(p_t) =
  \{1,1-\frac12 s_t^3\}$.  By Remark~\ref{rem:ll-extremes}, when
  $s_t=0$ and $\alpha(t) = \{1^2\}$ is indiscrete, $[z^3+1]$ is its
  unique preimage under the $\LL$ map, and when $s_t \neq 0$ and
  $\alpha(t)$ is generic, there are $3^1=3$ preimages.  As in
  Example~\ref{ex:mult-map-many}, the set $s^{-1}(0) = \{0\} \cup
  \{\frac1k \mid k \in \N\}$ and its complement $I - s^{-1}(0)$ is a
  countable union of open intervals.  For each open time interval
  between consecutive points where $\alpha(t)$ is indiscrete, there
  are three possible lifts.  And since there are countably many such
  intervals, there are $3^{\aleph_0}$ possible lifts of $\alpha$.
\end{example}

\section{Polynomial Homotopies}\label{sec:poly-homotopy}
Recall that $\poly_d^{mt}(\Ub)$ denotes the space of monic degree-$d$
polynomials up to translation with critical values in the subspace $\Ub
\subset \C$. For simplicity we call $\poly_d^{mt}(\Ub)$ the
\emph{polynomials over $\Ub$}.   The goal of this section is to clarify how 
the polynomial space $\poly_d^{mt}(\Ub)$ changes as $\Ub$ changes.  The 
first thing to point out is that there is something to prove.  For example, 
if $f\colon \C \to \C$ is a homeomorphism with $f(\Ub) = \Vb$, is it true that 
the space $\poly_d^{mt}(\Ub)$ of polynomials over $\Ub$ and the space $\poly_d^{mt}(\Vb)$ 
of polynomials over $\Vb$ are homeomorphic?  This is obvious for planar 
$d$-branched covers. 

\begin{example}
  Let $f\colon \C \to \C$ be a homeomorphism of $\C$ with $f(\closeddisk) = \closedsquare$, sending 
  the closed unit disk $\closeddisk$ homeomorphically to a closed rectangle $\closedsquare$.
  If $q$ is any planar $d$-branched cover with $\cvl(q) \subset \closeddisk$, then $f\circ q$ is a 
  planar $d$-branched cover with $\cvl(f \circ q) \subset \closedsquare$, and composing in the other 
  direction with $f^{-1}$ shows that this establishes a bijection between 
  the infinite-dimensional function space of planar $d$-branched 
  covers with critical values in $\closeddisk$ and the infinite-dimensional 
  function space of planar $d$-branched covers with critical values 
  in $\closedsquare$.  On the other hand, there is no homeomorphism $f\colon \C \to \C$ with 
  $f(\closeddisk) = \closedsquare$ where $f \circ p$ is a \emph{polynomial} for any polynomial $p$ with $\cvl(p) \subset \closeddisk$.
\end{example}

One way to show that this is true for polynomials is to characterize 
monic centered polynomials by their multiset of critical values and 
the monodromy, and we use this perspective in Part~\ref{part:geo-comb}.
Here we take a more direct approach. For the subspaces of $\C$ under 
consideration (Remark~\ref{rem:spaces}), homeomorphic subspaces of $\C$ 
can slowly transition from one to the other via a homotopy of $\C$, which 
ends up inducing a homeomorphism of the polynomials over those subspaces.  

\begin{defn}[Homotopies]\label{def:homotopies}
  A \emph{homotopy} is a continuous map $H \colon \Ub \times \Ib \to \Yb$ with $\Ib =
  [0,1]$. We write $H(u,t) = h_t(u) = h^u(t)$ for each $u \in \Ub$ and
  $t\in \Ib$.  These refer to the \emph{time map} $h_t \colon \Ub \to
  \Yb$, and the \emph{path map} $h^u \colon \Ib \to \Yb$.  We also
  write $u_t$ for $H(u,t) = h^u(t)$ and $\Ub_t$ for $H(\Ub,t) = h_t(\Ub)$,
  so that $h^u$ is a path from $u_0$ to $u_1$ and $H$ is a homotopy
  from $\Ub_0$ to $\Ub_1$.
\end{defn}

\begin{figure}
  \begin{tikzcd}[row sep=small, column sep=small]
    (z,t) \arrow[rr,mapsto] &&
    H(z,t) &
    \Ub \times \Ib \arrow[rr,"H"] &&
    \C  \\
    (\bz,t) \arrow[rr,mapsto] \arrow[dd,mapsto] &&
    H^n(\bz,t) \arrow[dd,mapsto] &
    \Ub^n \times \Ib \arrow[rr,"H^n"] \arrow[dd,"\mult \times \one"] &&
    \C^n  \arrow[dd,"\mult"] \\
    \\
    (M,t) \arrow[rr,mapsto] &&
    H_n(M,t) &
    \mult_n(\Ub) \times \Ib \arrow[rr,"H_n"] &&
    \mult_n(\C)
  \end{tikzcd}
  \caption{A point homotopy $H$, tuple homotopy $H^n$, and multiset
    homotopy $H_n$, with $z\in \Ub$, $\bz \in \Ub^n$ and $M \in
    \mult_n(\Ub)$.  \label{fig:induced-homotopies}}
\end{figure}

A homotopy of the points in $\C$ induces, for each $n$, a homotopy of
the tuples in $\C^n$ and a homotopy of multisets in $\mult_n(\C)$.

\begin{defn}[Induced homotopies]\label{def:induced-homotopies} 
    Given a subset $\Ub \subseteq \C$, a \emph{point homotopy} is 
    a homotopy $H \colon \Ub \times \Ib \to \C$ from
  $\Ub_0$ to $\Ub_1$ inside $\C$. This induces a \emph{tuple
  homotopy} $H^n \colon \Ub^n \times \Ib \to \C^n$ from $(\Ub_0)^n$ to
  $(\Ub_1)^n$ inside $\C^n$, obtained by applying the homotopy $H$ to 
  each factor of $\Ub^n$ simultaneously.
  Taking the quotient by the coordinate-permuting $\sym_n$ action on
  $\Ub^n$ yields the \emph{multiset homotopy} $H_n
  \colon \mult_n(\Ub) \times \Ib \to \mult_n(\C)$ from $\mult_n(\Ub_0)$ to
  $\mult_n(\Ub_1)$ inside $\mult_n(\C)$; see Figure~\ref{fig:induced-homotopies}.
  Concretely, if $M \in \mult_n(\Ub)$
  and $\bz \in \Ub^n$ is any $n$-tuple with $M = \mult(\bz)$, then
  $H_n(M,t) = \mult(H^n(\bz,t))$.
\end{defn}

\begin{defn}[Induced time maps and path maps]
  Let $H$ be a point homotopy with induced tuple homotopy $H^n$ and 
  multiset homotopy $H_n$.  The time maps are denoted $h_t \colon \Ub
  \to \C$ for $H$, $(h^n)_t \colon \Ub^n \to \C^n$ for $H^n$, and
  $(h_n)_t \colon \mult_n(\Ub) \to \mult_n(\C)$ for $H_n$.  
  Using the notation $\Ub_t = h_t(\Ub)$, we have the following path maps:
  a \emph{point path} $h^z \colon \Ib \to \Ub_t$ for $H$, a
  \emph{tuple path} $(h^n)^\bz \colon \Ib \to (\Ub_t)^n$ for $H^n$, and a
  \emph{multiset path} $(h_n)^M \colon \Ib \to \mult_n(\Ub_t)$ for $H_n$.
\end{defn}

The point paths in point homotopies may merge or split over time.

\begin{defn}[Splitting and merging]\label{def:split-merge}
  Let $H \colon \Ub \times \Ib \to \C$ be a point homotopy.  We say that
  $H$ \emph{splits points} if there are distinct points $u,v \in \Ub$
  and distinct times $s < t \in \Ib$ such that $u_s=v_s$ and $u_t \neq
  v_t$.  Similarly, we say that $H$ \emph{merges points} if there are
  distinct points $u,v \in \Ub$ and distinct times $s < t \in \Ib$ such
  that $u_s\neq v_s$ and $u_t=v_t$.  When $H$ does not split points it
  is \emph{nonsplitting} and when it does not merge points it is
  \emph{nonmerging}. A point homotopy \emph{preserves points} 
  if it is both nonsplitting and nonmerging.  In particular, for a
  point-preserving homotopy, $u_t = v_t$ at some time $s \in \Ib$ if and
  only if $u_t = v_t$ for all $t \in \Ib$.
\end{defn}

Nonsplitting homotopies can be initially simplified.

\begin{rem}[Initial inclusions]\label{rem:init-include}
  When $H$ is a nonsplitting point homotopy, points with $u_0 = v_0$
  stay together throughout, so the entire homotopy $H$ factors through
  the quotient map $h_0 \times \one \colon \Ub \times \Ib
  \twoheadrightarrow \Ub_0 \times \Ib$ to produce a simpler point homotopy
  $H' \colon \Ub_0 \times \Ib \to \C$.  This allows us to assume
  without loss of generality that all nonsplitting homotopies are
  injective at time $t=0$ with $\Ub = \Ub_0$.
\end{rem}

A nonsplitting point homotopy leads to multiset paths that are uniquely
liftable.

\begin{rem}[Nonsplitting]\label{rem:nonsplit}
  Let $H \colon \Ub \times \Ib \to \C$ be a point homotopy and let $H_n$
  be the corresponding multiset homotopy.  When $H$ is nonsplitting,
  the points in a multiset path can merge but not split, which means
  that they have a weakly increasing shape.  In particular, multiset
  paths in $H_n$ satisfy the necessary conditions to be uniquely
  liftable through the $\LL$ map (Theorem~\ref{thm:ll-map-unique}).
\end{rem}

\begin{figure}
    \begin{tikzcd}
        \poly_d^{mt}(\Ub) \times \Ib \arrow[r,"\widetilde{H}_n",dashed] \arrow[d,"\LL\times 1"] 
        & \poly_d^{mt}(\C) \arrow[d,"\LL"] \\
        \mult_n(\Ub) \times \Ib \arrow[r,"H_n"] & \mult_n(\C)
    \end{tikzcd}
    \caption{The polynomial homotopy $\widetilde{H}_n$ can be viewed as a lift
    of the multiset homotopy $H_n$ through the $\LL$ map, i.e. this
    diagram commutes.}
    \label{fig:LL-poly-homotopy}
\end{figure}

One consequence of Remark~\ref{rem:nonsplit} is homotopy at the level
of polynomial spaces.

\begin{defn}[Polynomial homotopies]\label{def:poly-homotopy}
  Let $H \colon \Ub \times \Ib \to \C$ be a nonsplitting point homotopy
  with $\Ub = \Ub_0 \subset \C$.  We can use Theorem~\ref{thm:ll-map-unique} 
  to define a function $\widetilde{H}_n \colon \poly_d^{mt}(\Ub) \times
  \Ib \to \poly_d^{mt}(\C)$ as follows.  Let $p$ be a monic degree-$d$
  polynomial with $[p] \in \poly_d^{mt}(\Ub)$ and let $M = \LL([p]) =
  \cvl(p) \in \mult_n(\Ub)$. By Remark~\ref{rem:nonsplit}, the multiset
  path $\alpha =(h_n)^M \colon \Ib \to \mult_n(\C)$ starting at $M$ has
  a weakly increasing shape and $[p]$ is a lift of its starting point
  $\alpha(0)$ through the $\LL$ map.  Next, by
  Theorem~\ref{thm:ll-map-unique} there is a unique lift of $\alpha$
  to a path $\beta \colon \Ib \to \poly_d^{mt}(\Ub)$ that starts at $[p]$.
  Finally, we define $\widetilde{H}_n$ at the point $([p],t) \in
  \poly_d^{mt}(\Ub) \times \Ib$ to be $\beta(t)$. By construction,
  $\LL \circ \widetilde{H}_n = H_n \circ (\LL\times 1)$ (see Figure~\ref{fig:LL-poly-homotopy}), 
  so $\widetilde{H}_n$ is a continuous map which we refer to as a \emph{polynomial homotopy}.
\end{defn}

\section{Polynomial Spaces}\label{sec:poly-spaces}

In this section we establish the homeomorphisms and compactifications
shown in Figure~\ref{fig:poly-homeo-compact} and the quotients and
deformation retractions shown in Figure~\ref{fig:poly-mult}.
The relationships between these polynomial spaces mirror the relationships
between subspaces of $\C$ used to define them, and this is what prompted 
our introduction of a visual shorthand (Definition~\ref{def:visual}).
We begin by establishing continuous maps between polynomial spaces.

\begin{rem}[Time maps]\label{rem:poly-time-maps}
    Let $H \colon \Ub \times \Ib \to \C$ be a nonsplitting point homotopy
    and let $\widetilde{H}_n$ be the corresponding polynomial homotopy.
    For any $s<t$ in $\Ib$, we can restrict $\Ib$ to the interval $[s,t]$ and apply
    Remark~\ref{rem:init-include} to obtain the point and multiset 
    time maps $h_t\colon \Ub_s \to \Ub_t$ and $(h_n)_t \colon \mult_n(\Ub_s) \to \mult_n(\Ub_t)$
    respectively. The polynomial homotopy $\widetilde{H}_n$ is then
    defined by lifting the multiset paths in $H_n$, which gives us the polynomial 
    time map $(\widetilde{h}_n)_t \colon \poly_d^{mt}(\Ub_s) \to \poly_d^{mt}(\Ub_t)$.
    In particular, this is a continuous map from one polynomial space to another.
\end{rem}

\begin{figure}
    \begin{tikzcd}
        \poly_{\lambda\to\mu}^{mt}(\Ub)  \arrow[r,"(\widetilde{h}_n)_t"] \arrow[d,"\LL"] 
        & \poly_{\lambda\to\mu}^{mt}(\Ub_t) \arrow[d,"\LL"] \\
        \mult_\mu(\Ub) \arrow[r,"(h_n)_t"] & \mult_\mu(\Ub_t)
    \end{tikzcd}
    \caption{The time map $(\widetilde{h}_n)_t$ is a lift of the time map
    $(h_n)_t$ through the $\LL$ map when restricted to a stratum in the double
    stratification of $\poly_d^{mt}(\C)$.}
    \label{fig:ll-time-map}
\end{figure}

\begin{prop}[Homeomorphisms]\label{prop:homeo}
    Let $\Ub$ be a compact subset of $\C$ and let 
    $H\colon \Ub\times \Ib \to \C$ be a point-preserving homotopy.
    Then for all $s < t$ in $\Ib$, the spaces
    $\poly_d^{mt}(\Ub_s)$ and $\poly_d^{mt}(\Ub_t)$ are homeomorphic.
\end{prop}

\begin{proof}
    First, consider the restricted map $H \colon \Ub \times [s,t] \to \C$
    and its time reversal $H^{\rev} \colon \Ub \times [s,t] \to \C$ defined by 
    $H^{\rev}(u,x) = H(u,s+t-x)$. Since $h_t \colon \Ub_s \to \Ub_t$ and the time reversal 
    $h_s^{\rev} \colon \Ub_t \to \Ub_s$ are inverses of one another, each time map is 
    a bijection, and since $\Ub$ is compact, it is also a homeomorphism.
    Since $H$ is nonsplitting, we can apply Remark~\ref{rem:poly-time-maps}
    to see that the polynomial time map is a continuous
    function $(\widetilde{h}_n)_t \colon \poly_d^{mt}(\Ub_s) \to \poly_d^{mt}(\Ub_t)$. 
    Since $H$ is nonmerging, the time reversal $H^{\rev}$ is a 
    nonsplitting point homotopy which yields the continuous time map
    $(\widetilde{h}_n)_s^{\rev} \colon \poly_d^{mt}(\Ub_t) \to \poly_d^{mt}(\Ub_s)$.
    Finally, the fact that $h_t$ and $h_s^{\rev}$ are inverses tells us that 
    $(\widetilde{h}_n)_t$ and $(\widetilde{h}_n)_s^{\rev}$ are
    inverses of one another, so the proof is complete.
\end{proof}

The following corollary is an immediate application of
Proposition~\ref{prop:homeo}.

\begin{cor}[Topological variations]\label{cor:homeo}
  If $\Ub$ and $\Vb$ are two closed intervals, two closed topological
  disks, two closed annuli, or two closed circles embedded in $\C$ (so
  that there is a point-preserving homotopy $\Ub$ to $\Vb$ inside $\C$),
  then $\poly_d^{mt}(\Ub)$ and $\poly_d^{mt}(\Vb)$ are homeomorphic.
\end{cor}

The homeomorphisms produced by Proposition~\ref{prop:homeo} and
Corollary~\ref{cor:homeo} are not canonically defined, but there are
sufficiently many to justify de-emphasizing the fine details of the
shapes listed in Definition~\ref{def:visual}.

\begin{rem}[Shapes]
  The parenthetical point-preserving homotopy condition in the
  statement of Corollary~\ref{cor:homeo} is, strictly speaking,
  unnecessary since all embeddings of these simple spaces into $\C$
  are connected by point-preserving homotopies, and more exotic
  possibilities would detract from our main point.  For any reasonably nice
  class of embeddings (such as all closed Euclidean rectangles in
  $\C$) there are more or less obvious point-preserving homotopies
  between them.  And by Proposition~\ref{prop:homeo} there are
  homeomorphisms between the corresponding polynomial spaces.  In
  particular, $\poly_d^{mt}(\closedsquare)$ can be discussed without
  specifying the precise closed rectangle, since there is only one
  such space up to a choice of homeomorphism.
\end{rem}

\begin{lem}[Closed and bounded]\label{lem:closed-bounded}
  If $\Ub \subset \C$ is closed/bounded/compact, then $\Ub^n$, $\mult_n(\Ub)$
  and $\poly_d^{mt}(\Ub)$ are closed/bounded/compact.
\end{lem}
\begin{proof}
  The properties of being closed, bounded and/or compact pass through
  finite direct products and these properties also project and lift
  through surjections where there is a finite upper bound on the size
  of a point preimage.
\end{proof}

\begin{lem}[Nested compact sets]\label{lem:nested-cpt}
    Let $U$ and $V$ be open subsets of $\C$, and suppose
    they can be expressed as the unions of nested compact sets 
    $\Ub_1 \subset \Ub_2 \subset \cdots$ and
    $\Vb_1 \subset \Vb_2 \subset \cdots$ respectively.
    If there is a homotopy from $U$ to $V$ such that the restriction
    to each $\Ub_i$ is point-preserving, 
    then $\poly_d^{mt}(U) \cong \poly_d^{mt}(V)$.
\end{lem}

\begin{proof}
    By Proposition~\ref{prop:homeo} and Lemma~\ref{lem:closed-bounded}, 
    $\poly_d^{mt}(\Ub_i)$ and $\poly_d^{mt}(\Vb_i)$
    are compact and homeomorphic. Since $\poly_d^{mt}(U)$ and 
    $\poly_d^{mt}(V)$ are each the nested union of these 
    homeomorphic compact sets, it follows that the two polynomial
    spaces are homeomorphic as well.
\end{proof}

We are now ready to prove Theorem~\ref{mainthm:homeomorphisms}.

\begin{thm}[Theorem~\ref{mainthm:homeomorphisms}]
    \label{thm:main-homeomorphisms}
    The complex plane $\C$, the punctured plane $\C_\zer$ and the real
    line $\R$ are homeomorphic to the open rectangle $\opensquare$, the
    open annulus $\openannulus$ and the open interval $\openint$
    respectively, and these induce homeomorphisms of polynomial spaces
    $\poly_d^{mt}(\C) \cong \poly_d^{mt}(\opensquare)$,
    $\poly_d^{mt}(\C_\zer) \cong \poly_d^{mt}(\openannulus)$ and
    $\poly_d^{mt}(\R) \cong \poly_d^{mt}(\openint)$.
\end{thm}

\begin{proof}
    Let $\opensquare \subset \C$ be an open rectangle
    and let $H \colon \C \times \Ib \to \C$ be the standard
    point-preserving homotopy from $\C$ (viewed simply as 
    the plane) to $\opensquare$ obtained by rescaling the real 
    and imaginary coordinates separately. For each positive integer $k$,
    let $\Ub_k \subset \C$ be the closed square of side length $k$ centered
    at the origin. Then $\C$ is the union of the compact sets 
    $\Ub_1 \subset \Ub_2 \subset \cdots$, and $H$ 
    transforms these to a sequence of homeomorphic nested compact sets 
    $\Vb_1 \subset \Vb_2 \subset \cdots$ with $\opensquare$ as their union.
    By Lemma~\ref{lem:nested-cpt}, 
    $\poly_d^{mt}(\C) \cong \poly_d^{mt}(\opensquare)$.
    The other two homeomorphisms follow from similar arguments.
\end{proof}

Next, we consider the compatibility of the homeomorphisms above
with natural compactifications.

\begin{prop}[Special compactifications]\label{prop:compact}
  Let $\Ub \subset \C$ be a closed interval $\closedint$ / closed square
  $\closedsquare$ / closed disk $\closeddisk$ / closed annulus
  $\closedannulus$ and let $U$ be the corresponding open interval
  $\openint$ / open square $\opensquare$ / open disk $\opendisk$ /
  open annulus $\openannulus$, so that $\Ub$ is the compact closure of
  $U$.  Then the compact closure of $U^n$ in $\C^n$ is $\Ub^n$, the
  compact closure of $\mult_n(U)$ in $\mult_n(\C)$ is $\mult_n(\Ub)$,
  and the compact closure of $\poly_d^{mt}(U)$ in $\poly_d^{mt}(\C)$
  is $\poly_d^{mt}(\Ub)$.
\end{prop}

\begin{proof}
  By Lemma~\ref{lem:closed-bounded}, $\Ub^n$, $\mult_n(\Ub)$ and
  $\poly_d^{mt}(\Ub)$ are compact and closed, so it is clear that the
  closure of the $U$-version is contained in the $\Ub$-version in each
  case.  In the other direction, there exist point-preserving
  homotopies $H$ that shrink $\Ub$ into the interior of $U$ in each
  case, and the time-reversed version $H'$ expands a closed subspace
  of $U$, homeomorphic to $\Ub$, to all of $\Ub$.  Once we lift $H'$ to a
  tuple, multiset and polynomial homotopy, it becomes clear that every
  tuple / multiset / polynomial in $\Ub^n$ / $\mult_n(\Ub)$ /
  $\poly_d^{mt}(\Ub)$ is the endpoint of a path that otherwise remains
  in $U^n$ / $\mult_n(U)$ / $\poly_d^{mt}(U)$.  In particular, the
  version with $\Ub$ is contained in the closure of the version with
  $U$.
\end{proof}

As an immediate consequence, we obtain a proof of Theorem~\ref{mainthm:compactifications}.

\begin{thm}[Theorem~\ref{mainthm:compactifications}]
    \label{thm:main-compactifications}
    The polynomial spaces $\poly_d^{mt}(\closedsquare)$, $\poly_d^{mt}(\closedannulus)$ and 
    $\poly_d^{mt}(\closedint)$ are compactifications of $\poly_d^{mt}(\opensquare)$, 
    $\poly_d^{mt}(\openannulus)$ and $\poly_d^{mt}(\openint)$.
\end{thm}

\begin{proof}
    This follows from Proposition~\ref{prop:compact} and Theorem~\ref{mainthm:homeomorphisms}.
\end{proof}

A simple argument shows that some polynomial spaces are dense in others.

\begin{rem}[Dense]\label{rem:dense}
  Let $H \colon \T \times \Ib \to \C$ be a point-preserving homotopy that
  rigidly rotates the unit circle and let $p$ be any polynomial where
  $-1$ is a critical value.  By choosing an appropriate portion of
  circle rotation over time and lifting this to a path of polynomials
  in $\poly_d^{mt}(\circleint)$, we can see $[p] \in
  \poly_d^{mt}(\circleint)$ is a limit of polynomials in
  $\poly_d^{mt}(\opencircleint)$.  In particular,
  $\poly_d^{mt}(\opencircleint)$ is dense in
  $\poly_d^{mt}(\circleint)$.
\end{rem}

\begin{figure}
  \begin{tikzcd}[column sep=1em]
    \poly_d^{mt}(\R) \arrow[d,"\cong",<->] \arrow[r,hookrightarrow] &
    \poly_d^{mt}(\C) \arrow[d,"\cong",<->] \arrow[r,equal] &
    \poly_d^{mt}(\C) \arrow[d,"\cong",<->] \arrow[r,hookleftarrow] &
    \poly_d^{mt}(\C_\zer) \arrow[d,"\cong",<->] \arrow[r,hookleftarrow] &
    \poly_d^{mt}(\T)  \arrow[d,equal] 
    \\
    \poly_d^{mt}(\openint) \arrow[r,hookrightarrow] \arrow[d,"\cmpt",hookrightarrow] &
    \poly_d^{mt}(\opensquare) \arrow[d,"\cmpt",hookrightarrow] \arrow[r,<->,"\cong"] &
    \poly_d^{mt}(\opendisk) \arrow[d,"\cmpt",hookrightarrow] \arrow[r,hookleftarrow] &
    \poly_d^{mt}(\openannulus) \arrow[d,"\cmpt",hookrightarrow]
    \arrow[r,hookleftarrow] &
    \poly_d^{mt}(\circleint) \arrow[d,equal] &
    \\
    \poly_d^{mt}(\closedint) \arrow[r,hookrightarrow] &
    \poly_d^{mt}(\closedsquare) \arrow[r,"\cong",<->] &
    \poly_d^{mt}(\closeddisk) \arrow[r,hookleftarrow] &
    \poly_d^{mt}(\closedannulus) \arrow[r,hookleftarrow] &
    \poly_d^{mt}(\circleint) 
  \end{tikzcd}
  \caption{Homeomorphisms and compactifications of spaces of 
    polynomial over $\R$, $\C$, $\C_\zer$, $\T$, an interval, a
    rectangle, a disk, an annulus, a circle and their
    interiors.\label{fig:poly-homeo-compact}}
\end{figure}

\begin{exmp}[Homeomorphisms and compactifications]\label{ex:homeo-compact}
  Figure~\ref{fig:poly-homeo-compact} shows a variety of inclusions,
  homeomorphisms and compactifications of polynomial spaces over $\R$,
  $\C$, $\C_\zer$, $\T$, an open interval $\openint$, an open
  rectangle $\opensquare$, an open disk $\opendisk$, an open annulus
  $\openannulus$, a closed interval $\closedint$, a closed rectangle
  $\closedsquare$, a closed disk $\closeddisk$, a closed annulus
  $\closedannulus$ and a circle $\circleint$.  The inclusions are
  immediate: if $V \subset U$, then $\poly_d^{mt}(V) \subset
  \poly_d^{mt}(U)$ by definition.  The homeomorphisms follow from
  Proposition~\ref{prop:homeo} and the obvious point-preserving
  homotopies.  And the compactifications follow from
  Proposition~\ref{prop:compact}.
\end{exmp}

\begin{prop}[Special quotients]\label{prop:quotients}
  There is a surjective quotient map from $\poly_d^{mt}(\closedint)$
  to $\poly_d^{mt}(\circleint)$ where the only identifications are
  among the polynomials with a critical value at an endpoint of the
  interval.  Similarly, there is a surjective quotient map from
  $\poly_d^{mt}(\closedsquare)$ to $\poly_d^{mt}(\closedannulus)$ where
  the only identifications are among the polynomials with a critical
  value in the left or right side of the rectangle.
\end{prop}

\begin{proof}
  A closed interval in $\C$ can be stretched around and moved in a
  point-preserving way until identifying its endpoints at time $t=1$
  to form an embedded circle.  For example, if $U = \closedint$ is
  embedded in $\C$ as the left half of the unit circle $\T$ and
  parameterized as $U = \{e(s) \mid s \in [-\frac14,\frac14] \}$ using
  Definition~\ref{def:pts-subsets}, then the map $H \colon U \times
  \Ib \to \T$ that sends $(e(s),t)$ to $e((1+t)\cdot s)$ is a
  nonsplitting point homotopy with these properties.  This map is
  point-preserving except that at time $t=1$, its length has doubled,
  covering all of $\T$ with the endpoints overlapping at $e(\frac12) =
  e(-\frac12) = -1$.  In particular, if $H'$ is $H$ restricted to $V
  \times \Ib$ where $V$ is $U$ with both endpoints removed, then $H'$ is
  point-preserving.  The time $t=1$ map of the polynomial homotopy
  $\widetilde{H}_n$ is a map $q$ that sends $\poly_d^{mt}(\closedint)$ to
  $\poly_d^{mt}(\circleint)$ (Lemma~\ref{rem:poly-time-maps}), and the
  restricted version $\widetilde{H_n'}$ is a homeomorphism between
  $\poly_d^{mt}(\openint)$ and $\poly_d^{mt}(\opencircleint)$
  (Lemma~\ref{lem:nested-cpt}). 
  Since $\poly_d^{mt}(\closedint)$ is
  the closure of $\poly_d^{mt}(\openint)$
  (Lemma~\ref{lem:closed-bounded}) and $\poly_d^{mt}(\opencircleint)$
  is dense in $\poly_d^{mt}(\circleint)$ (Remark~\ref{rem:dense}), the
  map $q$ is surjective.  And the homeomorphic embedding of the
  polynomials with all critical values in the interior of the
  interval means that the only identifications are among those
  polynomials with critical values at the endpoints.  The argument in
  the second case is nearly identical except that the unit circle is expanded
  to include a closed interval of possible positive magnitudes.
\end{proof}

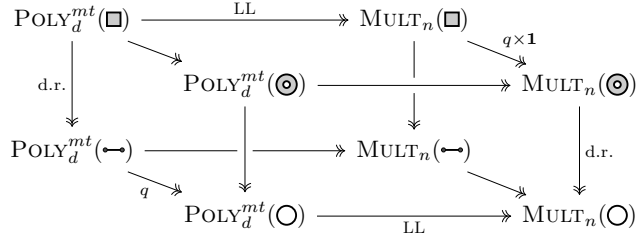
\begin{figure}
  \adjustbox{scale=.9,center}{
    \begin{tikzcd}[column sep=small, row sep=small]
      \poly_d^{mt}(\closedsquare)
      \arrow[rr,twoheadrightarrow,"\LL"]
      \arrow[dr,twoheadrightarrow]
      \arrow[dd,twoheadrightarrow,"\textrm{d.r.}"']
      &&
      \mult_n(\closedsquare) 
      \arrow[dr,twoheadrightarrow,"q \times \one"]
      \arrow[dd,twoheadrightarrow]
      \\
      &
      \poly_d^{mt}(\closedannulus)
      \arrow[rr,twoheadrightarrow,crossing over]
      &&
      \mult_n(\closedannulus) 
      \arrow[dd,twoheadrightarrow,"\textrm{d.r.}"]
      \\
      \poly_d^{mt}(\closedint)
      \arrow[rr,twoheadrightarrow]
      \arrow[dr,twoheadrightarrow,"q"']
      &&
      \mult_n(\closedint) 
      \arrow[dr,twoheadrightarrow]
      \\
      &
      \poly_d^{mt}(\circleint)
      \arrow[rr, twoheadrightarrow,"\LL"']
      \arrow[from=uu,twoheadrightarrow,crossing over]
      &&
      \mult_n(\circleint)
    \end{tikzcd}
  }
  \caption{The left-to-right $\LL$ maps, back-to-front quotient maps, and the top-to-bottom 
  deformation retractions for the $2n$-dimensional cell complexes associated with the four
    key shapes.\label{fig:poly-mult}}
\end{figure}

The special quotient from $\poly_d^{mt}(\closedsquare)$ to
$\poly_d^{mt}(\closedannulus)$ is best understood through an example.  
See Example~\ref{ex:rectangle-to-annulus}. 

\begin{defn}[Deformation retracting homotopies]
  Recall that a \emph{deformation retraction} from a topological space
  $\Ub$ to a subspace $\Vb$ is a map $H \colon \Ub \times \Ib \to \Ub$ such that
  $h_0(u)=u$, $h_t(v) = v$ and $h_1(u) \in \Vb$ for all $u\in \Ub$, $v\in
  \Vb$ and $t\in \Ib$.  Let $\Vb \subset \Ub$ be subspaces of $\C$ and let $H
  \colon \Ub \times \Ib \to \C$ be a nonsplitting point homotopy.  We say
  that $H$ is a \emph{nonsplitting deformation retracting point
  homotopy from $\Ub$ to $\Vb$} if the image of $H$ lies in $\Ub$ and $H$,
  with the range restricted to $\Ub$, is a deformation retraction from
  $\Ub$ to $\Vb$.
\end{defn}

\begin{prop}[Deformation retractions]\label{prop:def-retracts}
  If $H$ is a nonsplitting deformation retracting point homotopy from
  $\Ub$ to $\Vb$ inside $\C$, then there is a $\widetilde{H}_n$-induced
  deformation retraction from $\poly_d^{mt}(\Ub)$ to $\poly_d^{mt}(\Vb)$
  inside $\poly_d^{mt}(\C)$.
\end{prop}

\begin{proof}
  At time $t=0$, $H$ is the identity map on $\Ub$ and $\widetilde{H}_n$ is the identity
  map on $\poly_d^{mt}(\Ub)$, and since the range of $H$ remains in
  $\Ub$, the range of $\widetilde{H}_n$ remains in $\poly_d^{mt}(\Ub)$.  Next, a
  multiset in $\Vb$ remains fixed under $H_n$ and the unique lift of this constant 
  multiset path is a constant polynomial path.  Thus, $\widetilde{H}_n$ restricts the identity
  map on $\poly_d^{mt}(V)$ at each time $t$.  Finally, at time $t=1$
  the range of $H$ is in $\Vb$, so at time $t=1$ the range of $\widetilde{H}_n$
  is in $\poly_d^{mt}(\Vb)$.
\end{proof}

We now establish the last of our main tools.

\begin{thm}[Theorem~\ref{mainthm:deformations}]
    \label{thm:main-deformations}
    A deformation retraction of $\closedsquare$ onto any embedded
    arc $\closedint$ induces a deformation
    retraction from $\poly_d^{mt}(\closedsquare)$ to 
    $\poly_d^{mt}(\closedint)$. Similarly, the deformation retraction 
    of $\closedannulus$ onto any core curve $\circleint$ induces a 
    deformation retraction from $\poly_d^{mt}(\closedannulus)$ to
  $\poly_d^{mt}(\circleint)$. 
\end{thm}

\begin{proof}
    This follows immediately from Proposition~\ref{prop:def-retracts}.
\end{proof}

The results from this section establish the maps between the polynomial space
on the lefthand side of Figure~\ref{fig:poly-mult}.  The two back-to-front
  quotient maps are the special quotients of
  Proposition~\ref{prop:quotients}.  The two top-to-bottom deformation
  retractions follow from the obvious nonsplitting deformation
  retracting point homotopies from a closed rectangle to a horizontal
  line segment and from a closed annulus to a circle.  Finally, if we
  choose quotients and deformation retractions that
  commute at the level of subspaces of $\C$, the corresponding
  polynomial versions also commute.  The maps between the multiset
  spaces on the right are induced by the same set of point homotopies,
  and the $\LL$ map sends the four polynomial spaces on left to the
  four multiset spaces on the right.

\newpage
\part{Geometric Combinatorics}\label{part:geo-comb}

The goal of Part~\ref{part:geo-comb} is to prove Theorems~\ref{mainthm:intervals}, 
\ref{mainthm:circles}, \ref{mainthm:rectangles}, and~\ref{mainthm:annuli}, which are assertions 
about metric cell structures for the polynomial spaces $\poly_d^{mt}(\Xb)$ when $\Xb$ is $\closedint$, 
$\circleint$, $\closedsquare$, or $\closedannulus$, respectively.  
Part~\ref{part:geo-comb} is structured as follows.  In Section~\ref{sec:monodromy}, we recall 
the monodromy action of a planar branched cover and discuss its relationship with the side 
permutations defined in Section~\ref{sec:geo-comb}.
In Section~\ref{sec:intervals-thmA} we prove Theorem~\ref{mainthm:intervals},
in Section~\ref{sec:circle-thmB} we proof Theorem~\ref{mainthm:circles},
in Section~\ref{sec:rect-thmC} we prove Theorem~\ref{mainthm:rectangles},
and in Section~\ref{sec:annulus-thmD} we prove Theorem~\ref{mainthm:annuli}.

\section{Monodromy and Side Permutations}\label{sec:monodromy}

There is a direct connection between the monodromy action of a branched cover
and the side permutations defined in Section~\ref{sec:geo-comb}.  We begin with the general 
results (\ref{subsec:monodromy}) and then connect them to the 
combinatorics of our complexes (\ref{subsec:side-permutations}).

\subsection{Monodromy}\label{subsec:monodromy}
The monodromy action of a branched cover is a classical notion.  Here we recall the 
basic definitions and record a useful characterization of polynomials (Proposition~\ref{prop:monodromy-cvl}).
For a different perspective on some of the monodromy, see \cite[\S9]{DoMc22}. 

\begin{defn}[Monodromy]\label{def:monodromy}
  Let $p \colon \C \to \C$ be a planar $d$-branched cover 
  and let $z$ be a regular point in the range.  Regular paths based at $z$ in the range lift to $d$ paths 
  that permute the $d$ preimages of $z$ in the domain.  This \emph{monodromy action} can be encoded in a 
  group homomorphism $\pi_1(\C \setminus \cvl(p), z) \to \sym_d$.  Both the topology of the connected 
  $d$-sheeted cover and the monodromy action can be reconstructed from this group homomorphism.  See \cite{hatcher02-algtop} 
  for details.  
\end{defn}

\begin{defn}[Constellations]\label{def:constellation}
  When $p$ has $\ell$ distinct critical values, the fundamental group $\pi_1(\C \setminus \cvl(p), z)$ 
  is a free group $\F_\ell$ of rank $\ell$, and the map $\F_\ell \to \sym_d$ can be described 
  by the image of an ordered basis, i.e. by an ordered $\ell$-tuple of permutations 
  $\mathbf{g} = [ g_1\ g_2\ \cdots\ g_\ell] \in (\sym_d)^\ell$. One way of producing 
  $\ell$ loops that represent an ordered basis of $\F_\ell$ comes from drawing a star
  graph with $\ell$ arcs from $z$ to the $\ell$ critical values with disjoint interiors.  
  The $\ell$ loops are those that travel along one of these arcs, stopping just short of
  a critical value, going clockwise around the critical value, and returning along the
  same arc to $z$.  These loops define a basis for the free group $\F_\ell$. When the basis elements
  and the corresponding permutations are linearly ordered according to the clockwise order 
  that the arcs leave $z$, the product of these $\ell$ permutations in this order is a $d$-cycle. 
  In \cite{LaZv04} these are called \emph{constellations}. In the generic case where there are 
  $n$ critical values of multiplicity $1$, the $g_i$ are transpositions and the $n$-tuple $\mathbf{g}$ 
  is a minimum length factorization of a fixed $d$-cycle into $n = d-1$ transpositions.
\end{defn}

A continuous motion of the critical values that keeps them distinct but returns them setwise 
to their original positions alters the monodromy action in predictable ways.  In the generic
polynomial case, there is a $\braid_n$ action on $n$-tuples of transpositions.  

\begin{defn}[Hurwitz action]\label{def:hurwitz-action}
  Let $G$ be a group and denote each $n$-tuple $\mathbf{g}$ in $G^n$ by a row vector of the form
  $\mathbf{g} = [g_1\ g_2\ \cdots\ g_n]$, and let  $\{\beta_1,\ldots,\beta_{n-1}\}$ be the standard
  generating set for the $n$-strand braid group $\braid_n$. The \emph{Hurwitz action} of $\braid_n$
  on $G^n$ is defined by setting $\beta_i \cdot \mathbf{g}$ equal to
  \[
  \beta_i \cdot [g_1\ \cdots\ g_{i-1}\ \underline{g_i\ g_{i+1}}\ g_{i+2}\ \cdots\ g_n]
  = [g_1\ \cdots\ g_{i-1}\ \underline{g_{i+1}\ g_i^{g_{i+1}}}\ g_{i+2}\ \cdots\ g_n]
  \]
  where the altered portion has been underlined and $x^y$ denotes the conjugate $y^{-1}xy$. 
  This is called an \emph{elementary Hurwitz move}. One easily verifies that this action satisfies the 
  relations in the standard presentation of $\braid_n$.  Under the Hurwitz action,
  the product of the entries in this order remains constant.
\end{defn}    

Hurwitz defined this action in his 1891 paper \cite{hurwitz91} well before 
Artin's formal introduction of braid groups as an object of study \cite{artin25,artin47}.
In the generic case, the action is transitive.

\begin{rem}[Hurwitz transitivity]\label{rem:hurwitz-transitive}
  It is well-known that there are exactly $d^{d-2}$ ways to factor the $d$-cycle 
  $\delta = (1\ 2\ \cdots d)$ into $n=d-1$ transpositions, which correspond to the 
  $d^{d-2}$ maximal chains in $\ncperm_d = \ncpart_d$, and to the $d^{d-2}$ $n$-dimensional 
  simplices in $\size{\ncperm_d}_\Delta = \size{\ncpart_d}_\Delta$. And the Hurwitz action 
  is transitive on these sets \cite{BDSW14}.  In fact, for any permutation $\pi \in \sym_d$
  of absolute length $k$ (Definition~\ref{def:abs-order}), the $\braid_k$ action on factorizations 
  of $\pi$ into a product of $k$ transpositions is also transitive.  This is clear since the factorizations 
  and the transitivity take place cycle by disjoint cycle.  This more general property is known as 
  \emph{local Hurwitz transitivity}.
\end{rem}    

This action can be used to distinguish the $d^{d-2}$ generic monic centered polynomials with the same set of 
$n$ distinct critical values.\footnote{Recall that monic centered polynomials are the distinguished 
representatives of equivalence classes of monic polynomials up to precomposition with a translation 
(Remark~\ref{rem:center-trans}).} More generally, monic centered polynomials are distinguished by their 
critical values and their monodromy actions (Proposition~\ref{prop:monodromy-cvl}).  This elementary result
is not usually stated in these terms, so we include a proof.  We begin with a special case.

\begin{lem}[Monodromy and critical values]\label{lem:monodromy-cvl}
  Monic centered polynomials with the same generic set of critical values and the 
  same monodromy action are equal.  
\end{lem}

\begin{proof}
  Let $p$ and $q$ be generic monic centered polynomials of degree $d$ with 
  $\cvl(p) = \cvl(q) = M$ with $\shape(M)=1^n$ as their common (multi)set of critical values.
  We know that there are exactly $d^{d-2}$ polynomials in $\poly_d^{mc}(\C)$ 
  with $\cvl(p) = M$ (Theorem~\ref{thm:ll-classic}).  Next, any braided motion of the $n$ critical 
  points extends to a point-preserving homotopy of $\C$, which induces an automorphism 
  of $\poly_d^{mc}(\C)$ (Lemma~\ref{lem:nested-cpt}).  As a consequence there exists a polynomial in $\poly_d^{mc}(\C)$ 
  with $\cvl = M$ for each of the $d^{d-2}$ possible monodromy actions in the orbit of the Hurwitz action
  (Remark~\ref{rem:hurwitz-transitive}).  In other words, the map from the $d^{d-2}$ polynomials in $\poly_d^{mc}(\C)$
  with $\cvl = M$ to the $d^{d-2}$ possible monodromies is onto, and therefore also 
  injective. 
\end{proof}

\begin{prop}[Monodromy and critical values]\label{prop:monodromy-cvl}
  Monic centered polynomials with the same multiset of critical values and the 
  same monodromy action are equal.  
\end{prop}

\begin{proof}
  Let $p$ and $q$ be monic centered polynomials of degree $d$ with $\cvl(p)=\cvl(q) = M$, 
  $S = \set(M)$ and $\ell = \size{S} = \size{M}_\set$.  Let $\Ub \subset \C$ be a small closed neighborhood of $S$ 
  and note that $\mult_n(U)$ 
  is a small open neighborhood of $M$ in $\mult_n(\C)$.  Let $p_1$ and $q_1$ be generic polynomials in $\mult_n(U)$ 
  near $p$ and $q$ (Lemma~\ref{lem:local-onto}).   By dragging their critical values around inside $\Ub$ we may assume 
  that $\cvl(p_1) = \cvl(q_1)=M_1$ of shape $1^n$.  By Remark~\ref{rem:hurwitz-transitive} and 
  Lemma~\ref{lem:monodromy-cvl} we know that there is a braided motion of $M_1$ which transforms 
  $p_1$ into $q_1$, but more is true.  Let $\pi^i$ be the monodromy permutation
  corresponding to the unique critical value of $M$ in the component $\Ub_i$, (coming from an arc 
  from $z$ to $\partial \Ub_i$, clockwise around the boundary of $\Ub_i$ and then back to $z$) and 
  note that $\pi^i$ is the same permutation for both $p$ and $q$ since their monodromy actions agree, and 
  for their perturbations $p_1$ and $q_1$ since the loop remains regular and the permutation unchanged as 
  $p$ and $q$ are perturbed.
  For an appropriate choice of little loops, $\pi^i$ can be factored into transpositions coming 
  from the elements of $M_1$ in $\Ub_i$.  In particular, both $p_1$ and $q_1$ produce local transposition 
  factorizations of $\pi^i$.  Local Hurwitz transitivity means that the transposition factorization 
  of $\pi^i$ coming from $p_1$ can be transformed into the transposition factorization 
  of $\pi^i$ coming from $q_1$ by only moving the critical values of $M_1$ inside $\Ub_i$.
  Once all of these local modifications are made in each component of $\Ub_i$, the transformed $p_1$ has the same monodromy as 
  $q_1$ and they are equal by Lemma~\ref{lem:monodromy-cvl}.  By picking $\Ub$ to be an arbitrarily small neighborhood 
  of $S$, one can show that the distance from $p$ to $p_1$ to $q_1$ to $q$ in $\poly_d^{mt}$ is also arbitrarily small.
  Thus $p=q$.
\end{proof}

\subsection{Side permutations}\label{subsec:side-permutations}
We now connect the monodromy action to the side permutations defined in Section~\ref{sec:geo-comb}.
Let $p \in \poly_d^{mt}(\closedsquare)$ be a polynomial with $\cvl(p) \subset \closedsquare$ and
let $\Qb_p$ be its regular value complex. The basepoint $z = (x_\ell,y_b) \in \Qb_p$ is regular 
by construction and its $d$ indexed preimages $\{z_1,\ldots,z_d\}$ are arranged 
in counterclockwise order in the boundary of $\Pb_p$ (Definition~\ref{def:monic-poly-labels}). 
The connection between the \emph{monodromy action} of $\pi_1(\Qb_p - \cvl(p),z)$
on the preimages of $z$ and $\ncperm_d$ is straightforward up to a choice of conventions.  

\begin{defn}[Monodromy and $\Qb_p$]\label{def:monodromy-Qp}
  Permutations typically compose as functions from right-to-left, while paths are often
  concatenated from left-to-right.  To reconcile this difference, the monodromy 
  permutation is defined with a change of direction.  If $\gamma$ is an oriented path based at $z$
  that represents an element of  $\pi_1(\Qb_p - \cvl(p),z)$, the \emph{monodromy permutation}
  of $\gamma$ is the permutation whose disjoint cycles list the indices of the preimages of $z$
  as they occur when the lifted paths are traced in the \emph{opposite} direction.  This converts the
  right monodromy action defined by oriented paths to a left action by monodromy permutations.  The change of direction 
  ensures that left-to-right path concatenation matches right-to-left permutation composition.
  For example, if $\Vb$ is a closed disk in $\Qb_p$ with $z \in \partial \Vb$ and $\Ub = p^{-1}(\Vb)$ 
  as its preimage in $\Pb_p$, the \emph{clockwise} loop around $\partial \Vb$ lifts to clockwise arcs 
  in $\partial \Ub$, but the monodromy permutation is the \emph{counterclockwise} order in which the 
  indexed preimages of $z$ occur in boundary cycles of the components of $\Ub$.  Compare this with 
  Example~\ref{ex:factors}.  
\end{defn}    

As an illustration, consider side surfaces and the associated side permutations.

\begin{lem}\label{lem:factoring-delta}
  For each polynomial $p \in \poly_d^{mt}(\closedsquare)$ the $d$-cycle $\delta = (1\ 2\ \cdots\ d)$ can
  be factored as $\delta = \pi_i^L \cdot \pi_i^R$ for each $i \in [k]$, and as $\delta = \pi_j^T 
  \cdot \pi_j^B$ for each $j\in [l]$.
\end{lem}

\begin{proof}
  Let $\Vb_i^L$ and $\Vb_i^R$ the be left and right side surfaces of $\Qb$ determined by the arc $\alpha_i$.
  Let $\gamma$ be the clockwise path around $\Qb_p$ based at $z$, let $\gamma^L$ be the clockwise path around 
  $\partial \Vb_i$ based at $z$, and let $\gamma^R$ be the path from $z$ to $b_i$ in $\Bb$, clockwise 
  around $\partial \Vb_i^R$, and from $b_i$ back to $z$ in $\Bb$.  The concatenation $\gamma^L.\gamma^R$ is
  homotopic to $\gamma$ inside $\C \setminus \cvl(p)$.  The monodromy permutations of $\gamma^L$ and $\gamma^R$
  are the side permutations $\pi_i^L$ and $\pi_i^R$, and the monodromy permutation of $\gamma$ is
  $\delta = (1 2 \cdots d)$.  Finally the composition $(\pi_i^L) (\pi_i^R)$ is $\delta$, by the argument
  in Example~\ref{ex:factors}.  The top-bottom case is nearly identical.
\end{proof}  

The fact that permutation composition comes from concatenating clockwise boundary cycles based at $z$ 
explains the need to multiply top then bottom and left then right.  Lemma~\ref{lem:factoring-delta} 
extends to more general subsurfaces.

\begin{defn}[Interval permutations]\label{def:interval-perms}
  For each subinterval $[x_{i_1},x_{i_2}] \subset \Ib_p$, the \emph{interval 
  subsurface} $\Vb_{i_1,i_2}^I  = \Vb_{i_1}^R \cap \Vb_{i_2}^L$ is the subrectangle
  $[x_{i_1},x_{i_2}] \times \Jb_p \subset \Qb_p$.  For each subinterval
  $[y_{j_1},y_{j_2}] \subset \Jb_p$, the \emph{interval subsurface} 
  $\Vb_{j_2,j_1}^J = \Vb_j^T \cap \Vb_{j'}^B$ is the rectangle $\Ib_p \times 
  [y_{j_1},y_{j_2}] \subset \Qb_p$. Note that the order of the subscripts $j_1 < j_2$ 
  have been switched in the notation $\Vb_{j_2,j_1}^J$ to reflect the counterclockwise order 
  that $\ell_{j_2}$ and $\ell_{j_1}$ occur in the boundary cycle $\partial \Qb_p$.  This swap 
  simplifies the statement of Lemma~\ref{lem:factor-side-perm} below. 
  The preimage subsurfaces are denoted $\Ub^I_{i_1,i_2} 
  = p^{-1}(\Vb^I_{i_1,i_2})$ and $\Ub^J_{j_2,j_1} = p^{-1}(\Vb^J_{j_2,j_1})$.  Since
  many of these subsurfaces do not contain the basepoint $z$, we connect $\Vb_{i_1,i_2}^I$ to $z$
  with the portion of $\Bb$ from $z$ to $b_{i_1}$, and we connect $\Vb_{j_2,j_1}^J$ to $z$
  with the portion of $\Lb$ from $z$ to $\ell_{j_1}$.  In particular, the clockwise loop around 
  $\partial \Vb_{i_1,i_2}^I$ \emph{based at $z$ via $\Bb$} is the concatenation of the path from $z$ to 
  $b_{i_1}$ in $\Bb$, the clockwise loop around $\partial \Vb_{i_1,i_2}^I$, and the return path 
  from $b_{i_1}$ to $z$ in $\Bb$.  The monodromy permutation for this path is the \emph{interval permutation} 
  $\pi_{i_1,i_2}^I = \perm([\lambda_{i_1,i_2}^I])$ of the \emph{interval partition} $[\lambda_{i_1,i_2}]^I = \ncpart^B(\Ub^I_{i_1,i_2})$.
  Similarly, the clockwise loop around $\Vb_{j_2,j_1}^J$ \emph{based at $z$ via $\Lb$} is the 
  concatenation of the path from $z$ to $\ell_{j_1}$ in $\Lb$, the clockwise loop around 
  $\partial \Vb_{j_2,j_1}^J$, and the return path from $\ell_{j_1}$ to $z$ in $\Lb$.  
  And the monodromy permutation for this path is the \emph{interval permutation} $\pi_{j_2,j_1}^J = 
  \perm([\lambda_{j_2,j_1}^J])$ of the \emph{interval partition} $[\lambda_{j_2,j_1}]^J = \ncpart^L(\Ub^J_{j_2,j_1})$.
\end{defn}

\begin{rem}[Sides and intervals]\label{rem:side-interval}
  Side surfaces are examples of subinterval surfaces, meaning 
  side partitions are examples of subinterval partitions and side permutations 
  are examples of subinterval permutations.  For example, $\Qb_p = \Vb^{I}_{\ell,r} = 
  \Vb_{t,b}^J$, so $\delta = \pi^I_{\ell,r} = \pi^J_{b,t}$.  Similarly, 
  $\Vb_i^L =\Vb_{\ell,i}^I$, $\Vb_i^R = \Vb_{i,r}^I$, $\Vb_j^T = \Vb_{t,j}^J$ and 
  $\Vb_j^B = \Vb_{j,b}^J$, so $\pi_i^L = \pi_{\ell,i}^I$, $\pi_i^R = \pi_{i,r}^I$, 
  $\pi_j^T = \pi_{t,j}^T$ and $\pi_j^B = \pi_{j,b}^J$. 
\end{rem}

We have the following generalization of Lemma~\ref{lem:factoring-delta}.

\begin{lem}[Factoring side permutations]\label{lem:factor-side-perm}
  For each triple $i_1 < i_2 < i_3$ in $[k]$, there is a horizontal factorization $\pi_{i_1,i_3}^I = 
  (\pi_{i_1,i_2}^I) (\pi_{i_2,i_3}^I)$, 
  and for each triple $j_1 < j_2 < j_3$ in $[l]$, there is a vertical factorization $\pi_{j_3,j_1}^J = 
  (\pi_{j_3,j_2}^J) (\pi_{j_2,j_1}^J)$.
\end{lem}

\begin{proof}
  The rectangles $\Vb_{i_1,i_2}^I$ and $\Vb_{i_2,i_3}^I$ overlap on the 
  vertical arc $\alpha_{i_2}$, so a clockwise loop around $\Vb_{i_1,i_2}^I$ based at $z$ via $\Bb$,
  followed by a clockwise loop around $\Vb_{i_2,i_3}^I$ based at $z$ via $\Bb$, is homotopy equivalent 
  in $\C \setminus \cvl(p)$ to a clockwise loop around $\Vb_{i_1,i_3}$ based at $z$ via $\Bb$. The 
  corresponding permutation factorization follows by Definition~\ref{def:monodromy}.  Similarly the rectangles 
  $\Vb_{j_3,j_2}^J$ and $\Vb_{j_2,j_1}^J$ overlap on the horizontal arc $\beta_{j_2}$, and a clockwise 
  loop around $\Vb_{j_3,j_2}^J$ based at $z$ via $\Lb$, followed by a clockwise loop around $\Vb_{j_2,j_1}^J$ 
  based at $z$ via $\Lb$, is homotopy equivalent in $\C \setminus \cvl(p)$ to a clockwise loop around 
  $\Vb_{j_3,j_1}^J$ based at $z$ via $\Lb$.
\end{proof}

Let $\sigma_i = \pi_{i,i+1}^I$ be the permutation from a single edge subinterval of $\Ib$, 
and let $\tau_j = \pi_{j+1,j}^J$ be the permutation from a single edge subinterval of $\Jb$.  We call these 
\emph{basic horizontal side permutations} and \emph{basic vertical side permutations} respectively.  In this 
language we have the following corollary where we have factored horizontally and vertically as much as possible.

\begin{cor}[Side constellations]\label{cor:side-constellations}
  For each polynomial $p \in \poly_d^{mt}(\closedsquare)$ the $d$-cycle $\delta$ can
  be factored into basic horizontal side permutations 
  $\delta = \pi_1^L \cdot \sigma_1 \cdot \sigma_2 \cdots \sigma_{k-1} \cdot \pi_k^R$ 
  with corresponding horizontal side constellation 
  $[\pi_1^L\ \sigma_1\ \sigma_2\ \cdots\ \sigma_{k-1}\ \pi_k^R],$
  and into basic vertical side permutations 
  $\delta = \pi_l^T \cdot \tau_{l-1} \cdots \tau_2 \cdot \tau_1 \cdot \pi_1^B$
  with corresponding vertical side constellation 
  $[\pi_l^T\ \tau_{l-1}\ \cdots\ \tau_2\ \tau_1\ \pi_1^B].$
\end{cor}

\begin{proof}
  By Lemma~\ref{lem:factor-side-perm} we have $\delta = \pi_{\ell,r}^I =(\pi_{\ell,1}^I) (\pi_{1,2}^I) 
  (\pi_{2,3}^I) \cdots (\pi_{k-1,k}^I) (\pi_{k,r}^I)$ from the edges of $\Ib_p$, and $\delta = 
  \pi_{t,b}^J = (\pi_{l,1}^J) (\pi_{l,l-1}^J) \cdots (\pi_{3,2}^J) (\pi_{2,1}^J) (\pi_{1,b}^I)$ 
  from the edges of $\Jb_p$.  The statement simply uses simpler names for the factors 
  (Remark~\ref{rem:side-interval}).
\end{proof}

When the critical values of a permutation have distinct real or distinct imaginary parts, one
side chain determines the monodromy.

\begin{rem}[Side constellations and monodromy]\label{rem:side-monodromy}
  Let $p \in \poly_d^{mt}(\closedsquare)$ be a polynomial.  If the critical values in $\cvl(p)$ have
  distinct real parts, then the left side chain determines the monodromy of $p$. In fact, the loops based at $z$ 
  around the single edge subsurfaces $\Vb_{i,i+1}^I$ which contain a (necessarily unique) critical value are a basis for 
  the free group $\pi_1(\Qb_p \setminus \cvl(p),z)$.  In particular, the $(k+2)$-tuple 
  $[\pi_1^L\ \sigma_1\ \sigma_2\ \cdots\ \sigma_k\ \pi_k^R]$ is almost a constellation of permutations encoding 
  the monodromy map (Definition~\ref{def:constellation}).
  When the left side of $\closedsquare$ is regular, the first column of $\Qb_p$ is regular, the first entry 
  $\pi_1^L$ is the identity permutation and it is not needed.  When the right side of $\closedsquare$ is regular, 
  the last column of $\Qb_p$ is regular, the last entry $\pi_1^R$ is the identity permutation and it is not needed.  
  Once any initial and/or final identity permutations are removed, what remains is the tuple of permutations that 
  encode the monodromy action.  Similarly, if the critical values in $\cvl(p)$ have distinct imaginary parts, 
  then the bottom side chain determines the monodromy of $p$.  The argument is analogous.
\end{rem}

\section{Intervals and Theorem~\ref{mainthm:intervals}}\label{sec:intervals-thmA}

This section introduces and analyzes a metric cell structure on $\poly_d^{mt}(\closedint)$, 
thereby proving Theorem~\ref{mainthm:intervals}.  Let $\closedint = \Ib = [x_\ell, x_r]$ be a 
closed interval in $\R$, so that $(\closedint)^n = \Ib^n$ is an 
$n$-cube.  Recall that the space $\mult_n(\closedint) = \mult_n(\Ib)$ is a standard $n$-orthoscheme 
(Definition~\ref{def:orthoschemes}).  We begin with an example illustrating how to go from a multiset 
in $\closedint$ to a point in a face of this simplex.

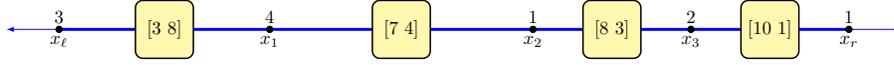
\begin{figure}
    \centering
    \begin{tikzpicture}

\tikzstyle{BlueLine}=[line width=0.3mm,color=blue,text=black]
\tikzstyle{BluePoly}=[BlueLine,fill=blue!20]
\tikzstyle{RedLine}=[line width=0.3mm,color=red,text=black]
\tikzstyle{RedPoly}=[RedLine,fill=red!20]
\tikzstyle{GreenLine}=[thick,color=black!30!green,text=black]
\tikzstyle{GreenPoly}=[thick,color=green!50!black,fill=green!30,join=bevel]
\tikzstyle{PurpleLine}=[line width=0.3mm,color=red!60!blue!80,text=black]
\tikzstyle{PurplePoly}=[PurpleLine,fill=red!40!blue!50]
\tikzstyle{OrangeLine}=[thick,color=orange]
\tikzstyle{GrayLine}=[thick,color=black!50!gray]
\tikzstyle{GrayPoly}=[GrayLine,fill=gray!20]
\tikzstyle{dot}=[shape=circle,draw,color=black,fill=black,inner sep=1.5pt]
\tikzstyle{opendot}=[shape=circle,draw,color=black,fill=white,inner sep=1.5pt]
\tikzstyle{bigdot}=[dot,inner sep=2pt]
\tikzstyle{littledot}=[dot,inner sep=0.75pt]
\tikzstyle{littleopendot}=[opendot,inner sep=0.75pt]
\tikzstyle{smalldot}=[dot,inner sep=1pt]
\tikzstyle{disk}=[thick,shape=circle,draw,color=black,fill=yellow!10]
\tikzstyle{lightdisk}=[thick,shape=circle,draw,color=black!40,fill=yellow!20]
\tikzstyle{plate}=[thick,shape=rectangle,draw,color=black,fill=yellow!40,rounded corners,minimum size=1.1cm]

\coordinate (xl) at (-10,0);
\coordinate (x1) at (-6,0);
\coordinate (x2) at (-1,0);
\coordinate (x3) at (2,0);
\coordinate (xr) at (5,0);

\draw[BlueLine,ultra thick] (xl) -- (xr);
\draw[BlueLine,thin,latex-latex] (-11,0) -- (6,0);

\foreach \v in {xl, x1, x2, x3, xr} {
    \filldraw[dot,color=black] (\v) circle (.6mm);
}

\node[anchor=north] () at (xl) {$x_\ell$};
\node[anchor=north] () at (x1) {$x_1$};
\node[anchor=north] () at (x2) {$x_2$};
\node[anchor=north] () at (x3) {$x_3$};
\node[anchor=north] () at (xr) {$x_r$};

\node[plate] () at ($(xl)!0.5!(x1)$) {$[3\ 8]$};
\node[plate] () at ($(x1)!0.5!(x2)$) {$[7\ 4]$};
\node[plate] () at ($(x2)!0.5!(x3)$) {$[8\ 3]$};
\node[plate] () at ($(x3)!0.5!(xr)$) {$[10\ 1]$};

\node[anchor=south] () at (xl) {$3$};
\node[anchor=south] () at (x1) {$4$};
\node[anchor=south] () at (x2) {$1$};
\node[anchor=south] () at (x3) {$2$};
\node[anchor=south] () at (xr) {$1$};

\end{tikzpicture}
    \caption{The multiset $M = x_\ell^3 x_1^4 x_2^1 x_3^2 x_r^1 \in 
    \mult_{11}(\Ib)$ labels a point in a $3$-dimensional face of the 
    $11$-dimensional simplex $\mult_{11}(\Ib)$.  The $3$-dimensional face
    is determined by the linear composition $\bm = \comp(M) = [3\ 4\ 1\ 2\ 1]$ and the 
    four vertex labels $[3\ 8]$,  $[7\ 4]$, $[8\ 3]$ and $[10\ 1]$ below $\bm$ in the linear composition 
    order.  The exact point in this open $3$-simplex is determined by location
    of the $3$-element set $C = \{x_1,x_2,x_3\}$ in $I$.  The vertices below the linear composition are shown
    superimposed on the $4$ edges of the $3$-subdivided interval $\Ib_C$.}
    \label{fig:spine-orthoscheme}
\end{figure}

\begin{example}\label{rem:spine-orthoscheme}
  Let $\Ib = [x_\ell,x_r]$ be an interval, and let $M = x_\ell^3 x_1^4 x_2^1 x_3^2 x_r^1$ be 
  an $11$-element multiset in $\Ib$. The multiset $M$ labels a point in the
  $11$-dimensional orthoscheme $\mult_{11}(\Ib)$, and we can separate out the metric and combinatorial 
  information that $M$ contains.  There are $3$ points $C =\{x_1,x_2,x_3\}$ in the interior 
  of $\Ib$, and a $5$-tuple $\bm = [3\ 4\ 1\ 2\ 1]$ that records the multiplicities of the 
  $5$ vertices of the $3$-subdivided interval $\Ib_C$.  The open $3$-simplex 
  containing the point labeled by $M$ is determined by the $5$-tuple $\bm$ and the exact point in this 
  open $3$-simplex is determined by location of the $3$-element set $C$ in the interior of $\Ib$.  See 
  Figure~\ref{fig:spine-orthoscheme}.
\end{example}

\begin{defn}[Multisets in an interval]\label{def:multiset-int}
  An arbitrary element $M \in \mult_n(\closedint)$ with $\closedint = \Ib = [x_\ell,x_r]$
  can be written  in the form $M = x_\ell^{m_\ell}x_1^{m_1} \cdots x_k^{m_k} x_r^{m_r}$.  It 
  can also be split into two pieces of information.  Adding the elements of $M$ as vertices of $\Ib$
  creates a subdivided interval $\Ib_C$, where $C = \{x_1,\ldots,x_k\} = \set(M) \cap I$ 
  records the elements of $M$ in the interior of $\Ib$ indexed in the left-to-right order they occur 
  (Definition~\ref{def:int-subdivide}). There is also a $(k+2)$-tuple $\bm = [m_\ell\ m_1\ \cdots\ m_k\ m_r]$ 
  that records the multiplicities of the vertices of $\Ib_C$ in the multiset $M$, listed in the same 
  left-to-right order.  We say $\bm = \comp(M)$ is the \emph{linear composition} of $M$.  Since there 
  need not be elements of $M$ at either end of $\Ib = \closedint$, we have $m_\ell, m_r \geq 0$, but
  $m_i>0$ for $i \in [k]$.  The linear composition $\bm$ of length $k+2$ determines the open $k$-simplex 
  of the $n$-orthoscheme $\mult_n(\closedint)$ and the choice of a $k$-element subset $C \subset I$ 
  specifies a point in that open $k$-simplex.  Note that $M$ can be reconstructed from $C$ and $\bm$.
\end{defn}

Figure~\ref{fig:comp-3-orthoscheme} shows a standard $3$-dimensional orthoscheme $\mult_3(\Ib)$ with 
the linear compositions that label its faces.  The face poset of simplex $\mult_n(\closedint)$ is the 
finite poset of linear compositions of $n$.

\begin{defn}[Linear compositions]\label{def:linear-comp}
  A \emph{linear composition of $n$ with length $k+2$} is a row vector 
  $\bm = [m_\ell\ m_1\ \cdots\ m_k\ m_r]$ of sum $n$ with integers $m_\ell,m_r \geq 0$ and 
  $m_i > 0$ for $i \in [k]$.  For $k>0$, an \emph{elementary merge} of $\bm$ replaces 
  two adjacent entries with their sum.  This produces a new linear composition of $n$ of length 
  $(k-1)+2 = k+1$.
  Let $\comp_n(\closedint)$ denote the set of all linear compositions of $n$ together with the 
  partial order $\bm \geq \bm'$ if there is a sequence of elementary merges that starts at $\bm$ and ends at $\bm'$.
  The graded poset $\comp_n(\closedint)$ has a unique maximal element $[0\ 1\ 1\ \cdots\ 1\ 0]$ and $n+1$ 
  minimal elements $[m_\ell\ m_r]$ with $m_\ell,m_r \geq 0$ and $m_\ell+ m_r =n$.
\end{defn}

Multisets in an interval and the corresponding linear compositions
are explored with more detail in \cite{dm-cnc}.

\begin{figure}
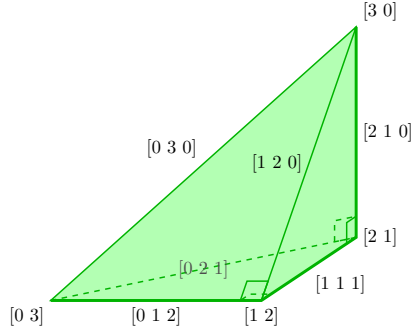

    \centering
    \includestandalone[scale=.7]{fig-comp-3-orthoscheme}
    \caption{The faces of the $3$-orthoscheme $\mult_3(\protect\closedint)$ have linear composition labels.
    The linear composition labels of its $0$-cells and $1$-cells are shown. The open $3$-cell has 
    label $[0\ 1\ 1\ 1\ 0]$ and the four $2$-cells have labels $[0\ 1\ 1\ 1]$ (bottom), $[0\ 1\ 2\ 0]$ (front), 
    $[0\ 2\ 1\ 0]$ (back), and $[1\ 1\ 1\ 0]$ (right).  Its spine is the thick path from $[0\ 3]$ to $[1\ 2]$ to
    $[2\ 1]$ to $[3\ 0]$.}
    \label{fig:comp-3-orthoscheme}
\end{figure}

\begin{rem}[Stratified cell structure]\label{rem:poly-strata-cells}
  Let $\comp \colon \mult_n(\closedint) \to \comp_n(\closedint)$ be the map that sends $M \mapsto \comp(M)$
  (Definition~\ref{def:multiset-int}).  The $\comp$ map determines the simplicial cell structure 
  on $\mult_n(\closedint)$ in the sense that $M_1$ and $M_2$ belong to the same open simplicial 
  face if and only if $\comp(M_1) = \comp(M_2)$.  Moreover, since $\shape(M)$ is the multiset of positive 
  entries in $\comp(M)$, the simplicial cell structure is a stratified cell structure in the sense of 
  Definition~\ref{def:stratified-cell-structures}. As a consequence there is an induced cell structure 
  on $\poly_d^{mt}(\closedint)$ and the $\LL$ map becomes a cellular map $\LL \colon \poly_d^{mt}(\closedint) 
  \to \mult_n(\closedint)$.  Since the generic degree is $d^{d-2}$ (Theorem~\ref{thm:ll-classic}), the 
  cell structure on $\poly_d^{mt}(\closedint)$ is built out of $d^{d-2}$ standard $n$-orthoschemes.
  There is also an induced simplicial cell structure on the $n$-cube $(\closedint)^n$ that turns the $\mult$ map 
  into a cellular map.  This is the typical subdivision of the $n$-cube into $n!$ standard $n$-orthoschemes that are permuted
  by the symmetric group action.  See Figure~\ref{fig:interval-cell-cplxes}.
\end{rem}

\begin{figure}
  \begin{tikzcd}
    \poly_d^{mt}(\closedint)
    \arrow[r,twoheadrightarrow,"d^{d-2}"] &
    \mult_n(\closedint)
    \arrow[r,twoheadleftarrow,"n!"] &
    (\closedint)^n
  \end{tikzcd}
  \caption{Three cell complexes related to a closed interval $\protect\closedint$, all with 
  simplicial orthoscheme metrics. The $\LL$ map on the left and the $\mult$ map on the right
  are cellular maps with respect to these cell structures.\label{fig:interval-cell-cplxes}}
\end{figure}

We are now ready to prove Theorem~\ref{mainthm:intervals}.

\begin{thm}[Theorem~\ref{mainthm:intervals}]\label{thm:main-intervals}
  The space $\poly_d^{mt}(\closedint)$ of polynomials with critical values in a 
  closed interval (with the stratified Euclidean metric) is isometric to the 
  order complex $|\ncpart_d|_\Delta$ (with the orthoscheme metric).  
\end{thm}

\begin{proof}
  Let $\closedsquare$ be a closed rectangle $\Qb' = \Ib' \times \Jb' = \closedsquare$,  let $\closedint$ be $\Ib' \times \{y_1\}$ 
  for some $y_1 \in \Jb'$, and let $\geocomb \colon \poly_d^{mt}(\closedint) \to \size{\ncperm_d^L}_\Delta$
  be the $\geocomb$ map of Definition~\ref{def:geom-comb} with the domain restricted to the subspace
  $\poly_d^{mt}(\closedint)$ and the range projected onto the first factor.  We show that this version 
  of the $\geocomb$ map is a bijective cellular homeomorphism.
  The cellular nature of the map is immediate from the way the cell structures were defined on the domain and 
  range.  In particular, different polynomials in the same cell differ only in their
  barycentric coordinates and these coordinates are used in all three spaces $\mult_n(\closedint)$, $\poly_d^{mt}(\closedint)$
  and $\size{\ncperm_d^L}_\Delta$. This also shows that the local metrics are preserved. 
  
  To see that $\geocomb$ is surjective, note that the image contains at least one point in 
  each of the $d^{d-2}$ open top-dimensional cells of $\size{\ncperm_d}_\Delta$ by 
  Remark~\ref{rem:hurwitz-transitive}, it contains all of these points because the map is 
  cellular, and it contains the union of the closed top-dimensional cells because the image 
  is compact and thus closed (Proposition~\ref{prop:compact}). Finally, every point 
  in $\size{\ncperm_d}_\Delta$ is in the image, because every point is in the boundary of a 
  top-dimensional cell.  This is the order complex version of the fact that every chain of 
  noncrossing permutations extends to a maximal chain.
  
  To see that $\geocomb$ is also injective, suppose $p, q \in \poly_d^{mt}(\closedint)$ have
  the same image $\geocomb(p)=\geocomb(q)$. This means that $p$ and $q$ lie in the same
  simplex of $\poly_d^{mt}(\closedint)$ and have the same left side chains,
  which means that $p$ and $q$ have the same monodromy (Remark~\ref{rem:side-monodromy}).
  Moreover, the multiplicity of the unique critical value in an edge of $\Ib'_p$ is 
  determined by the absolute length of the basic horizontal side permutations $\sigma_i$ 
  (Corollary~\ref{cor:side-constellations}), so $\cvl(p)$ and $\cvl(q)$ have the same 
  linear composition $\bm$.  Next, the fact that $p$ and $q$ are sent to the same point 
  in this simplex means that they have the same barycentric coordinates. The barycentric 
  coordinates encode relative widths (Definition~\ref{def:int-subdivide}) and we can use 
  these to reconstruct the locations of the critical values in $\Ib'$.  Since $p$ and $q$ 
  have the same set of critical values with the same multiplicities, we have $\cvl(p)=\cvl(q)$ 
  as multisets.  By Proposition~\ref{prop:monodromy-cvl}, $p=q$ and $\geocomb$ is injective.
  Finally, as a bijective map from a compact space to a Hausdorff space, it is a homeomorphism.  
\end{proof}

The metric simplicial complex $\poly_d^{mt}(\closedint) = \size{\ncpart_d}_\Delta$ is also 
the branched line complex $\br_d^m(\closedint)$ whose points are labeled by marked $d$-branched lines.  
 
\begin{defn}[Branched lines]\label{def:br-line-cplx}
  Let $\Ib' = [x'_\ell, x'_r]$ be an interval in $\R$ and let $\Ib = [x_\ell,x_r]$ be an interval that contains
  $\Ib'$ in its interior.  For each polynomial $p \in \poly_d^{mt}(\Ib')$, there is a preimage branched 
  line $p^{-1}(\Ib)$ which is a metric banyan (Example~\ref{ex:banyans}).  It is a single banyan rather than a banyan 
  grove because the interval $\Ib$ contains all of the critical values of $p$.  The endpoints of $\Ib$, being
  regular, have $d$ preimages each that can be marked / labeled as through they were representative points on the left and right
  sides of a $d$-branched rectangle $\Pb_p$.  Let $\br_d^m(\Ib')$ be the space of all such marked metric 
  banyans / marked branched lines.  The procedure just described gives a map $\poly_d^{mt}(\Ib')$ to $\br_d^m(\Ib')$.
  Next, from a marked metric banyan, it is possible to read off the left side partitions
  and the relative widths of the intervals between the critical values.  In other words, the marked metric banyan 
  contains the same information as the corresponding point in $\size{\ncpart_d}_\Delta$, and we have a map $\br_d^m(\Ib') \to 
  \size{\ncpart_d}_\Delta$.  Finally, Theorem~\ref{mainthm:intervals} provides a map back from $\size{\ncpart_d}_\Delta \to \poly_d^{mt}(\Ib')$.  One can trace through the definitions of these maps to see that they are consistent and thus all three are bijections.
  It is in this sense that the space of monic centered degree-$d$ polynomials with critical values in a fixed interval is the same as the 
  space of marked $d$-branched lines.
 \end{defn}

\section{Circles and Theorem~\ref{mainthm:circles}}\label{sec:circle-thmB}
We now derive Theorem~\ref{mainthm:circles} from Theorem~\ref{mainthm:intervals}.  First, 
recall the definition of the dual braid complex.

\begin{defn}[Dual braid complex]\label{def:dual-braid-cplx}
  The \emph{dual braid complex} $K_d$ is a quotient space of
  the order complex $\size{\ncperm_d}_\Delta$.  Recall that there is an ordered $k$-simplex in 
  $\size{\ncperm_d}_\Delta$ for each chain $\pi_0 < \pi_1 < \cdots < \pi_k$ of noncrossing permutations.  
  The \emph{edge labels} of this $k$-simplex are $\sigma_i = \pi_{i-1}^{-1}\pi_i$ and the 
  corresponding \emph{factorization} of  $\delta = (1\ 2\ \cdots\ d)$ has constellation 
  $[\pi^L\ \sigma_1\ \sigma_2\ \cdots\ \sigma_k\ \pi^R]$ where $\pi^L = \pi_0$ and $\pi^R = (\pi_k)^{-1} \delta$.  
  The possibly trivial permutation $\pi^L$ determines the vertex where the ordered $k$-simplex 
  starts in $\size{\ncperm_d}_\Delta$, and the possibly trivial permutation $\pi^R$ determines the vertex where it ends. 
  The edge labels are nontrivial. In the dual braid complex $K_d$, two simplices are identified if and only if they have the same sequence of edge labels.  
  The result is a $\Delta$-complex with one vertex and $d^{d-2}$ 
  top-dimensional cells that are standard $n$-orthoschemes.
\end{defn}

\begin{defn}[Standard representatives]\label{def:std-circle-reps}
  For every equivalence class of simplices in $\size{\ncperm_d}_\Delta$ that are identified to form one simplex 
  in the dual braid complex $K_d$, there is a standard representative where $\pi^R$ is the identity.  Concretely, 
  if  $[\pi^L\ \sigma_1\ \sigma_2\ \cdots\ \sigma_k\ \pi^R]$ is a constellation where 
  $\pi^R$ is nontrivial, then its standard representative is $[\pi^L_\textrm{new}\ \sigma_1\ \sigma_2\ \cdots\ \sigma_k\ \pi^R_\textrm{new}]$ where $\pi^R_\textrm{new}$ is the identity and $\pi^L_\textrm{new}$ is 
  $\delta\cdot \pi^R \cdot \delta^{-1} \cdot \pi^L$.
\end{defn}

\begin{thm}[Theorem~\ref{mainthm:circles}]\label{thm:main-circles}
  The space $\poly_d^{mt}(\circleint)$ of polynomials with critical
  values in a circle is homeomorphic to a quotient of the complex $\poly_d^{mt}(\closedint)$ 
  by face identifications. As a metric $\Delta$-complex, $\poly_d^{mt}(\circleint)$ 
  is the dual braid complex $K_d$ with the orthoscheme metric.
\end{thm}

\begin{proof}
  Let $\closedint$ be a closed horizontal interval in $\R$.  As $\closedint$ is dragged 
  around to form a circle $\circleint$ at time $t=1$, the polynomials in $\poly_d^{mt}(\closedint)$, 
  with their monodromy and critical value characterization, are dragged around to the 
  polynomials in $\poly_d^{mt}(\circleint)$.  This process is completely reversible except 
  for polynomials with critical values at both endpoints of $\closedint =[x_\ell,x_r]$.  When $p$ is 
  a polynomial of this type, its endpoint critical values are merged at time $t=1$ and the corresponding 
  monodromy permutations around these critical values, defined by appropriate paths based at $z$, are multiplied.  
  In particular, if $p$ has horizontal side constellation 
  $[\pi_1^L\ \sigma_1\ \sigma_2\ \cdots\ \sigma_k\ \pi^R_1]$ (Corollary~\ref{cor:side-constellations})
  then the final polynomial with one fewer critical value has a monodromy constellation of the 
  form $[\pi^L_\textrm{new} \ \sigma_1\ \sigma_2\ \cdots\ \sigma_k\ \pi^R_\textrm{new}]$
  where $\pi^R_\textrm{new}$ is the identity and 
 $\pi^L_\textrm{new}$ is $\delta\cdot \pi^R_1 \cdot \delta^{-1} \cdot \pi^L_1$.  The conjugation of $\pi_1^R$ 
 by $\delta$ represents a change of path surrounding the critical value at $x_r$ so that the clockwise loop around $x_r$ concatenated
 with the standard clockwise loop around $x_\ell$ is homotopic to a loop surrounding both which is dragged 
 to a clockwise loop around their merged critical value. 
 Finally, the constellations label the faces of $\poly_d^{mt}(\closedint)$ and each particular point is encoded in 
 the metric information of the relative widths of the edges in $\Ib'_p$.  Since the endpoint identification only 
 changes the monodromy and leaves the metric information unchanged, the induced identification
 on $\poly_d^{mt}(\closedint)$ is an isometric face identification.  It is also clear from the labels that this is the same identification as that used to create the dual braid complex 
 (Definition~\ref{def:std-circle-reps}).
\end{proof}

\begin{defn}[Branched circles]\label{def:br-circle-cplx}
  Let $\Ub$ be a closed disk that contains $\circleint$ in its interior.  We give $\Ub$ a minimal cell 
  structure so that $\circleint$ is a subcomplex.  It has two vertices, three edges and two $2$-cells.
  There is one vertex $v$ in $\circleint$ and another vertex $u$ in $\partial \Ub$.  The rest of $\circleint$
  is an open edge as is the rest of $\partial \Ub$.  The third edge $e$ connects $u$ and $v$.
  For any particular polynomial $p \in \poly_d^{mt}(\circleint)$, the preimage of $\Ub$ is a nonsingular disk 
  diagram and the preimages of $u$ can be marked and cyclically labeled by the set $[d]$.  The preimage of $\circleint$ is a metric cactus (Example~\ref{ex:cacti}).  We use the regular point $u \in \partial \Ub$
  as the basepoint for the monodromy action on the preimages of $u$ and we call cyclically marked 
  preimages of $u$ a \emph{marking}.  Let $\br_d^m(\circleint)$ be the space of all such marked metric 
  cacti / marked branched circles.  The procedure just described gives a map $\poly_d^{mt}(\circleint)$ to 
  $\br_d^m(\circleint)$,  and from each marked metric cactus it is possible to read off the monodromy and to recover the map from the metric branched circle to $\circleint$, including the location and multiplicity of the 
  critical values.  In particular, it is possible to recover the multiset $\cvl(p)$.  With the monodromy and the 
  critical value multiset we recover the polynomial $p$. Thus the map from polynomials to marked branched circles is 
  injective and it is not too hard to show that it is also onto. It is in this sense that the space of monic centered degree-$d$ polynomials with critical values in a fixed circle is the same as the space of marked $d$-branched circles.
\end{defn}

\section{Rectangles and Theorem~\ref{mainthm:rectangles}}\label{sec:rect-thmC}

This section introduces and analyzes a metric cell structure on $\poly_d^{mt}(\closedsquare)$, 
thereby proving Theorem~\ref{mainthm:rectangles}.  The cell structure and the method of 
proof are patterned after the linear case.

\begin{figure}
  \begin{tikzcd}
    \poly_d^{mt}(\Qb) \arrow[d,hookrightarrow,"1"] \arrow[r,twoheadrightarrow,"d^{d-2}"] &
    \mult_n(\Qb) \arrow[d,twoheadrightarrow,"n!"] \arrow[r,twoheadleftarrow,"n!"] &
    \Qb^n \arrow[d,equal,"1"] 
    \\
    \poly_d^{mt}(\Ib) \times \poly_d^{mt}(\Jb) 
    \arrow[r,twoheadrightarrow,"(d^{d-2})^2"] &
    \mult_n(\Ib) \times \mult_n(\Jb) 
    \arrow[r,twoheadleftarrow,"(n!)^2"] &
    \Ib^n\times \Jb^n  
  \end{tikzcd}
  \caption{Six complexes related to a rectangle $\protect\closedsquare = \Qb =\Ib \times \Jb$, all with 
  bisimplicial orthoscheme metrics. The horizontal maps are $\LL$ maps and $\mult$ maps and 
  the vertical maps come from deformation retracting onto the sides of $\protect\closedsquare$.
  The numbers on the arrows indicate the generic degree of these cellular maps.\label{fig:six-rectangle-spaces}}
\end{figure}

\begin{defn}[Bisimplicial cells]\label{def:cellular}
  Let $\closedsquare$ be a closed rectangle $\Qb = \Ib \times \Jb$ as in 
  Definition~\ref{def:coord-rectangle} where $\Ib$ and $\Jb$ are intervals 
  of length $s$ and $t$ respectively. By focusing on the real and imaginary 
  coordinates separately, the product space $(\closedsquare)^n = \Qb^n$ can 
  be viewed as $\Ib^n \times \Jb^n$, an $n$-cube of side length $s$ times an 
  $n$-cube of side length $t$.  If we quotient by the full action of $\sym_n 
  \times \sym_n$ with the first symmetric group acting on the real coordinates 
  and second acting on the imaginary coordinates, the quotient space is 
  $\mult_n(\Ib) \times \mult_n(\Jb)$, a direct product of a standard $n$-orthoscheme 
  of side length $s$ with a standard $n$-orthoscheme of side length $t$.  We call 
  this an oriented \emph{bisimplex} or an \emph{biorthoscheme} when we are viewing it as 
  a metric object. The space $\mult_n(\Qb)$ is an intermediate space where 
  we only quotient by the diagonal action of $\sym_n$ on the two factor $n$-cubes. The cells
  of $\mult_n(\Qb)$ are orbits of cells in $\Qb^n$.  In particular, in these $2n$-dimensional 
  spaces, $\Qb^n$ with its $(n!)^2$ top-dimensional cells which are the product of two $n$-orthoschemes.
  and $\mult_n(\Qb)$ has $n!$ top-dimensional biorthoschemes.  This complete our description of 
  the righthand square of Figure~\ref{fig:six-rectangle-spaces}.
\end{defn}

To clarify the cell structure we focus on the projections onto $\Ib$ and $\Jb$.

\begin{defn}[Multisets in a rectangle]\label{def:mult-rectangle}
  For each $n$-element multiset $M \in \mult_n(\Qb)$, let $C = \Re(M) \cap I$ 
  with $k = \size{C}$, let $D = \Im(M) \cap J$ with $l = \size{D}$, and let 
  $\Ib_C$, $\Jb_D$ and $\Qb_{C,D}$ be the corresponding subdivisions of $\Ib$, 
  $\Jb$ and $\Qb$.  All of the elements of $M$ are at the vertices of $\Qb_{C,D}$
  and as in the linear case, we can split $M$ into a combinatorial part and a metric 
  part.  The combinatorial aspect is a $(k+2) \times (l+2)$ grid is nonnegative integers 
  that record the multiplicities of the vertices of $\Qb_{C,D}$.  The metric part 
  are the numbers $\bary(\Ib_C)$ and $\bary(\Jb_D)$ which record the relative widths of the
  columns and rows of cells in $\Qb_{C,D}$ respectively.  The combinatorics determine the open 
  bisimplicial cell and the barycentric coordinates determine the point in the cell. The numbers 
  $\bary(\Ib_C)$ determine the point in the first simplex and the numbers $\bary(\Jb_D)$ 
  determine the point in the second simplex.
\end{defn}

\begin{rem}[Cellular maps]\label{rem:cellular-maps}
  Since varying the widths and heights of the rectangles does not change the shape of the multiset, 
  the cell structure described in Definition~\ref{def:mult-rectangle} is a stratified cell structure 
  on $\mult_n(\Qb)$, and so by Definition~\ref{def:stratified-cell-structures} it induces a
  stratified cell structure on $\poly_d^{mt}(\Qb)$ built out of biorthoschemes, turning the 
  horizontal $\LL$ map in the upper left part of Figure~\ref{fig:six-rectangle-spaces} into a cellular map.
  The induced map from $\poly_d^{mt}(\Qb)$ in the upper left to biorthoscheme $\mult_n(\Ib)\times \mult_n(\Jb)$
  is completely determined by the horizontal and vertical barycentric coordinates.  Next consider the 
  $\LL \times \LL$ map from $\poly_d^{mt}(\Ib) \times \poly_d^{mt}(\Jb)$ to $\mult_n(\Ib) \times \mult_n(\Jb)$.
  This is already known to be cellular by the results in Section~\ref{sec:intervals-thmA}.  Finally, the vertical
  map on the left is the $\geocomb$ map of Section~\ref{sec:geo-comb} with the range relabeled using 
  Theorem~\ref{mainthm:intervals}.  The definitions are consistent, this map is also cellular and the lefthand square
  commutes.
\end{rem}

To complete the proof of Theorem~\ref{mainthm:rectangles} we only need to show that the vertical 
$\geocomb$ map on the left-hand side is injective.

\begin{thm}[Theorem~\ref{mainthm:rectangles}]
    \label{thm:main-rectangles}
    The space $\poly_d^{mt}(\closedsquare)$ of polynomials with critical
    values in a closed rectangle (with the stratified Euclidean metric) 
    is isometric to a subcomplex of 
    $|\ncperm_d|_\Delta \times |\ncperm_d|_\Delta$
    (with the orthoscheme metric). 
\end{thm}

\begin{proof}
  Let $\closedsquare$ be a closed rectangle $\Qb' = \Ib' \times \Jb' = \closedsquare$ and note 
  that the $\geocomb$ map is cellular and metric-preserving on each open cell by Remark~\ref{rem:cellular-maps}.  
  Once we show it is injective, its image is a subcomplex and the proof is complete.
  To see injectivity, suppose we are given a point $\geocomb(p)$ in the interior of a product of two 
  orthoschemes in $\size{\ncperm_d^L}_\Delta \times \size{\ncperm_d^B}_\Delta$ that came from a polynomial 
  $p \in \poly_d^{mt}(\closedsquare)$.  This point can be expressed as a pair of side chains $\pi_1^L < \pi_2^L 
  < \cdots < \pi_k^L$ and $\pi_1^B < \pi_2^B < \cdots < \pi_l^B$ and the (ordered) barycentric 
  coordinates in each (ordered) factor simplex.  The left noncrossing permutations 
  $\pi_i^L$ can be converted first to left noncrossing partitions $[\lambda_i]^L$ 
  (Definition~\ref{def:ncperm}) and then to top-bottom matchings $[\mu_i]^{TB}$ 
  (Proposition~\ref{prop:match-part}), which we can draw as noncrossing multiarcs in a standard $d$-branched 
  rectangle $\Pb$, with the arcs of $[\mu_i]^{TB}$ connecting the $d$ points $t_{m,i}$ in the top 
  sides $\{T_m\}$ to the $d$ points $b_{m,i}$ in the bottom sides $\{B_m\}$.  This is the multiarc 
  $\widetilde \alpha_i$.  Similarly, the chain of noncrossing bottom permutations 
  $\pi_j^B$ allow us to draw the multiarc $\widetilde \beta_j$.  Together this reconstructs the full 
  $1$-skeleton of the regular point complex $\Pb_p$. By adding in the bounded regions as $2$-cells we 
  recover all of the regular point complex $\Pb_p$.  Next,  we reconstruct the critical point complex 
  $\Pb'_p$ as the cellular dual of $\Pb_p$, then the critical value complex $\Qb'_p$ and the critical 
  complex map $\Pb'_p \to \Qb'_p$ from $\Pb'_p$ (Definition~\ref{def:br-coord-cplx}).  From the critical complex map we 
  also know the multiplicity of each vertex in $\Qb'_p$. This is the combinatorial information contained in 
  the multiset $\cvl(p)$.  Together with the barycentric coordinates, we can reconstruct the critical 
  value multiset $M = \cvl(p)$.  Finally, the cellular critical complex map $\Pb'_p \to \Qb'_p$ 
  determines the branched cellular regular complex map $\Pb_p \to \Qb_p$. From the $d$-sheeted cover between 
  their $1$-skeletons one can read off the monodromy action. By Proposition~\ref{prop:monodromy-cvl} this 
  means that we can reconstruct $p$ from its image under the $\geocomb$ map, and the $\geocomb$ map is injective.
\end{proof}

From the generic degrees of the maps, it is clear that the subcomplex in the image of the $\geocomb$ map only 
contains $d^{d-2} \cdot n!$ of the $(d^{d-2})^2$ top-dimensional bisimplices.  Here is one natural way to characterize which 
vertices, and more generally which simplices are in this subcomplex.

\begin{figure}
  \begin{tabular}{ccc}
    \bt{c} \includegraphics[scale=.48]{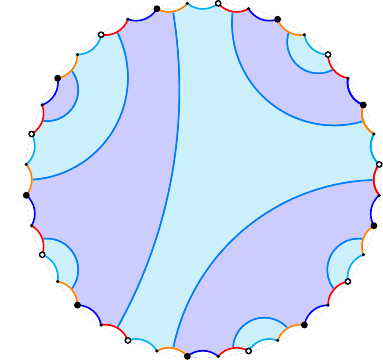} \et &
    \bt{c} \includegraphics[scale=.48]{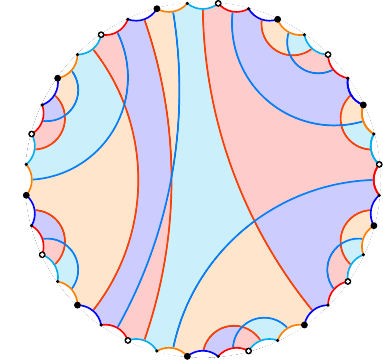} \et &
    \bt{c} \includegraphics[scale=.48]{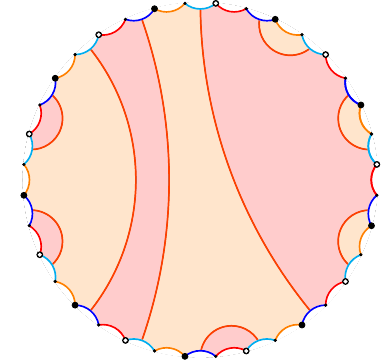} \et \\
    \bt{c} \includegraphics[scale=.5]{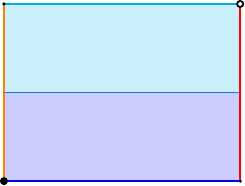} \et & 
    \bt{c} \includegraphics[scale=.5]{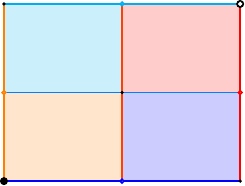} \et & 
    \bt{c} \includegraphics[scale=.5]{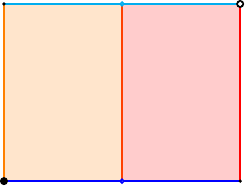} \et \\
  \end{tabular}
  \caption{The left-right matching shown in the upper left and the bottom-top matching 
  shown on the upper right combine to form the basketball in the middle. \label{fig:lr-tb-rect}}
\end{figure}

\begin{defn}[Basketballs]\label{def:basketballs}
  The vertices of $\poly_d^{mt}(\closedsquare)$ correspond to polynomials where the multiset $\cvl(p)$ is contained 
  in the four corners of $\closedsquare$.  The left and bottom chains become single noncrossing permutations
  $\pi_1^L$ and $\pi_1^B$, which encode a top-bottom matching $[\mu_1]^{TB}$ and a left-right matching $[\mu_1]^{LR}$, 
  or more topologically the multiarc $\widetilde \alpha_1$ and the multiarc $\widetilde \beta_1$, which are the 
  preimage under $p$ of a regular vertical arc $\alpha_1$ and a regular horizontal arc $\beta_1$.  The lift of this regular 
  ``plus sign'' in $\closedsquare$ must be $d$ ``plus signs'' in the $d$-branched $4d$-gon.  This is illustrated in 
  Figure~\ref{fig:lr-tb-rect}. The resulting combinatorial structure is what Martin, Savitt and Singer 
  call a \emph{basketball} \cite{martin-savitt-singer-07}.  See \cite{gersten-stallings-88,sjogren-15} for 
  more background and \cite{savitt-09} for more on basketballs. As described in \cite[Theorem~2.8]{martin-savitt-singer-07},
  basketballs are in one-to-one correspondence with the noncrossing partitions of $[4d]$ in which each block has size $4$.
  They also show that the number of basketballs is the Fuss--Catalan number $C_d^{(4)} = \frac{1}{3d+1} \binom{4d}{d}$, and each
  is obtained from a polynomial in the manner described here. Thus, the
  image of $\geocomb$ is a subcomplex of $|\ncperm_d^L|_\Delta \times |\ncperm_d^B|_\Delta$
  with $C_d^{(4)}$ vertices.    More generally, a chain of left side permutations and a chain of bottom side permutations
  describe a bisimplex in the image of the $\geocomb$ map if and only if for every choice of left permutation $\pi_i^L$ and 
  choice of bottom permutation $\pi^R_j$, the combination encodes a pair of multiarcs that form a basketball.
\end{defn}

\begin{defn}[Branched rectangles]\label{def:br-rect-cplx}
  Let $\closedsquare = \Qb' = \Ib' \times \Jb'$ be a closed rectangle contained in the 
  interior of a larger closed rectangle $\Qb = \Ib \times \Jb$ based at its bottom left corner $z$ and add an 
  edge from $z$ to the bottom left corner of $\Qb'$.  This gives $\Qb$ a cell structure.
  For each  polynomial $p\in \poly_d^{mt}$, the preimage disk $\Pb = p^{-1}(\Qb)$ receives a cell structure
  and it contains the metric critical value complex $\Pb'_p$ as a subcomplex built out of Euclidean rectangles. There
  is also the natural labeling of the preimages of $z$ (Definition~\ref{def:monic-poly-labels}). 
  We call this a \emph{marked $d$-branched rectangle}. Let $\br_d^m(\closedsquare)$ be the collection of all 
  marked $d$-branched rectangles.  As described in the proof of Theorem~\ref{mainthm:rectangles}, the 
  marked $d$-branched rectangle is sufficient information to reconstruct $p$, and it is also not too hard to 
  prove that every possible such marked $d$-branched rectangle arises from some polynomial.  It is in this sense 
  that the space of monic centered degree-$d$ polynomials with critical values in a fixed rectangle is the same 
  as the space of marked $d$-branched rectangles.
\end{defn}

\section{Annuli and Theorem~\ref{mainthm:annuli}}\label{sec:annulus-thmD}

In this section we prove Theorem~\ref{mainthm:annuli} by performing a gluing
of $\closedsquare$ into $\closedannulus$ similar to gluing of $\closedint$ into $\circleint$ 
in the proof of Theorem~\ref{mainthm:circles} in Section~\ref{sec:circle-thmB}.

\begin{figure}
  \begin{tabular}{ccc}
    \bt{c} \includegraphics[scale=.48]{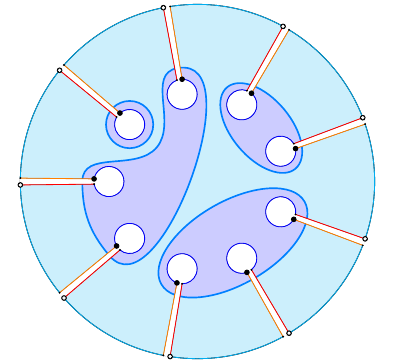} \et &
    \bt{c} \includegraphics[scale=.48]{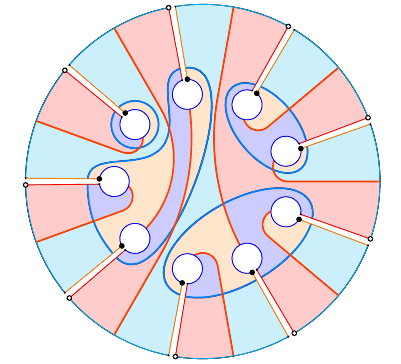} \et & 
    \bt{c} \includegraphics[scale=.48]{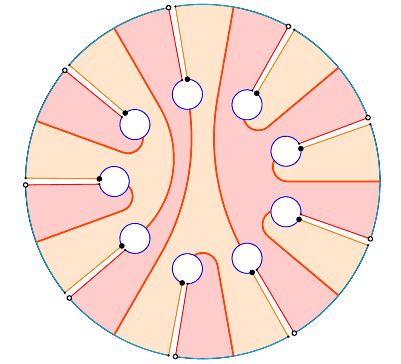} \et \\
    \bt{c} \includegraphics[scale=.5]{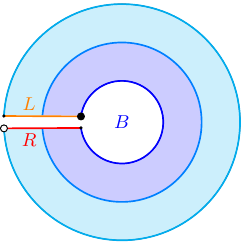} \et &
    \bt{c} \includegraphics[scale=.5]{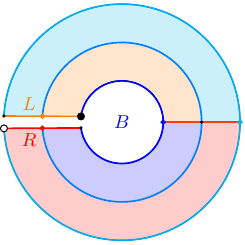} \et &
    \bt{c} \includegraphics[scale=.5]{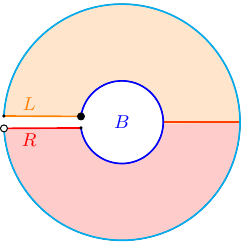} \et \\
  \end{tabular}
  \caption{A standard version of a left-right matching and a
    compatible bottom-top matching, when the rectangle in $\C$ has
    been moved by a point-preserving homotopy so that the left and
    right sides nearly touch and the bottom side nearly encloses a
    disk in the interior of the closed annulus about to be
    formed.
    \label{fig:lr-tb-ann}}
\end{figure}

\begin{thm}[Theorem~\ref{mainthm:annuli}]
    \label{thm:main-annuli}
    The space $\poly_d^{mt}(\closedannulus)$ of polynomials with
    critical values in a closed annulus is homeomorphic to 
    a quotient of $\poly_d^{mt}(\closedsquare)$ by face identifications.
\end{thm}

\begin{proof}
  Let $H \colon \closedsquare \times [0,1] \to \C$ be a nonsplitting homotopy that is point preserving except at 
  time $t=1$ when it identifies the left and right sides to create a closed annulus $\closedannulus$ with the 
  top side becoming the outer circle and the bottom side becoming the inner circle.  As described in 
  Definition~\ref{def:poly-homotopy} this induces to a polynomial homotopy $\widetilde H_n \colon 
  \poly_d^{mt}(\closedsquare) \times [0,1] \to \poly_d^{mt}(\C)$. For each polynomial $p \in 
  \poly_d^{mt}(\closedsquare)$, we have continuously deforming polynomials $p_t = (\widetilde h_n)_t(p) 
  \in \poly_d^{mt}(h_t(\closedsquare))$. By Proposition~\ref{prop:homeo}, this in turn 
  induces a quotient of the space $\poly_d^{mt}(\closedsquare)$ that becomes $\poly_d^{mt}(\closedannulus)$.  
  We can say something a bit more precise.  The shape of the multiset $\cvl(p_t)$ changes at time $t=1$ if and 
  only if there are critical values at a point in the left and right sides of $\closedsquare$ at the same height---so they become one point at time $t=1$.  When this happens, the corresponding monodromy permutations, from 
  appropriate loops based at a regular basepoint, are multiplied. Distinct points in the same open bisimplex of 
  $\poly_d^{mt}(\closedsquare)$ only differ in the relative metrics assigned to the horizontal and vertical 
  subdivisions.  The identification process is independent of these metrics which means that the quotient is by 
  face identifications.
\end{proof}    

\begin{example}\label{ex:rectangle-to-annulus}
  As described in the proof of Theorem~\ref{mainthm:annuli}, there is a continuous deformation 
  from $\closedsquare$ at time $t=0$ to
  $\closedannulus$ at time $t=1$.  This determines a continuous deformation from 
  a branched rectangle to a branched annulus. We illustrate the process for a vertex of $\poly_d^{mt}(\closedsquare)$.
  The left-right and bottom-top matchings obtained by pulling back a ``plus sign'' in 
  $\closedsquare$ to a $d$-branched $4d$-gon are shown in Figure~\ref{fig:lr-tb-rect},
  along with the basketball formed by the union of both matchings. The effect
  of the continuous deformation of $\closedsquare$ into $\closedannulus$
  on this basketball is illustrated in Figure~\ref{fig:lr-tb-ann}, and this represents a vertex of 
  $\poly_d(\closedannulus)$.
  This is a ``standardized'' version of what is happening topologically. For the 
  actual continuously deforming polynomial $p_t$ under the corresponding polynomial homotopy (Definition~\ref{def:poly-homotopy}), 
  there is a more concrete version which has this cell structure. It starts with a standard form like Figure~\ref{fig:lr-tb-rect},
  and ends near time $t=1$ as an embedded complex in $\C$ in which the left and right sides with matching
  subscripts are about to meet, as in Figure~\ref{fig:lr-tb-ann}.  
\end{example}

\begin{rem}[From rectangular to polar]\label{rem:rectangular-to-polar}
  The transition from focusing on the rectangular coordinates of the critical values to focusing on 
  their polar coordinates involves continuously varying the collection of polynomials under consideration
  using a polynomial homotopy.  In particular, maintaining the same geometric and combinatorial data
  means changing the polynomial under consideration.  Conversely, keeping the polynomial constant means 
  drastically changing the geometric combinatorics.  For example, polynomials might be in a lower 
  dimensional cell if there are coincidences in the coordinates of its critical values, but one can have 
  coincidences in one coordinate system without having them in the other.
\end{rem}

Just as we can view $\poly_d^{mt}(\closedsquare)$ as a subcomplex
of $|\ncpart_d^L| \times |\ncpart_d^B|$, the proof above shows that $\poly_d^{mt}(\closedannulus)$ is
isomorphic to a subcomplex of the direct product 
$K_d \times |\ncpart_d^B|$, where $K_d$ is the dual braid complex (Definition~\ref{def:dual-braid-cplx}).
We call this subcomplex the (marked) branched annulus complex $\br_d^m(\closedannulus)$.  Rather than
define it explicitly, we comment on standard names for its bisimplices and refer the reader to the first article
in this series \cite{DoMc22}.

\begin{defn}[Standard representatives]\label{def:std-annulus-reps}
  For every equivalence class of bisimplices in the branched rectangle complex $\br_d^m(\closedsquare)$ 
  that are identified to form one bisimplex in the branched annulus complex $\br_d^m(\closedannulus)$, 
  there is a standard representative where the right side of $\closedsquare$ is regular.  In particular, 
  the vertices of the branched annulus complex are labeled by polynomials $p$ where the multiset $\cvl(p)$ 
  lives in the endpoints of the left side $\Lb$ of $\closedsquare$.  By Theorem~\ref{mainthm:intervals} 
  applied to $\br_d^m(\Lb)$, these are indexed by noncrossing partitions and $\br_d^m(\closedannulus)$
  has exactly $C_d$ vertices, where $C_d$ is a Catalan number.  The process of standardizing the pair of 
  side chain labels of a bisimplex is straightforward on the left side permutation chain.  It is exactly as 
  described in Definition~\ref{def:std-circle-reps}. The impact on the bottom side permutation chain, however, 
  is more difficult to characterize cleanly.  See the discussion in \cite{DoMc22}.
\end{defn}

\appendix
\section{A Proof of Theorem~\ref{thm:poly-strata-cover}}\label{app:poly-strata-cover}

This appendix derives Theorem~\ref{thm:poly-strata-cover} from Theorem~B in \cite{DoMc-20}.
In that article all of the polynomials satisfy $p(0)=0$, but here we
use a slight generalization.

\begin{defn}[Base Pair]\label{def:base-pair}
  Let $\poly_d^{m,b\to c} =\{ p \in \poly_d^m \mid p(b)=c\}$ be the
  subset of $\poly_d^m$ sending $b$ to $c$.  We call $b \to c$ the
  point/value \emph{base pair} for this subspace.  The polynomials in
  $\poly_d^{m,b\to c}$ are indexed by their critical points.  For
  example, if $\cpt(p) = z_1^{m_1} \cdots z_k^{m_k}$, then $p'(z) =
  d\cdot (z-z_1)^{m_1} \cdots (z-z_k)^{m_k}$ and
  \[p(z) = d\cdot \left( \int_b^z (w-z_1)^{m_1} \cdots
  (w-z_k)^{m_k}\ dw \right) + c.\] 
\end{defn}

The space of monic centered polynomials is homemorphic to the space of
monic polynomials up to translation (Remark~\ref{rem:center-trans}),
and the space of monic polynomials with a fixed base pair $b \mapsto
c$ is a $d$-sheeted cover of these spaces, so long as we restrict our
attention to polynomials with critical values that avoid $c$.

\begin{rem}[Base Pairs and Covers]\label{rem:base-pair-cover}
  If $\cb = \{c\}$, then the restricted affine map $\aff\colon
  \poly_d^{m,b\to c}(\C_\cb) \to \poly_d^{mt}(\C_\cb)$ is a
  $d$-sheeted covering map.  The restriction to polynomials with
  critical values in $\C_\cb$ means that for these polynomials, $c$
  has $d$ distinct preimages.  For each preimage there is a unique
  representative of the translation equivalence class where this
  preimage is labeled $b$.  Also note that since the cover map
  preserves the shapes of the critical points and critical values, it
  restricts to $d$-sheeted covering maps $\aff\colon
  \poly_{\lambda\to\mu}^{m,b\to c}(\C_\cb) \to
  \poly_{\lambda\to\mu}^{mt}(\C_\cb)$.
\end{rem}

The following definitions are slight variations of those in
\cite{DoMc-20} and \cite{beardon02}.

\begin{defn}[$\cpt$ to $\cvl$]\label{def:cpt-to-cvl}
  Let $b\to c$ be a base pair, let $\bm = (m_1,\ldots,m_k)$ be a
  $k$-tuple of positive integers with $n=m_1+\cdots+m_k$, and let
  $\bz_\bm = (z_1,\ldots,z_k)$ be a point in $\C^k$ where the
  subscript $\bm$ reminds us that coordinates in $\bz_\bm$ come with
  assigned multiplicities.  Let \[p_{\bm}^{b\to c}(z) =
  p_{\bm,\bz_\bm}^{b\to c}(z) = d\cdot \left( \int_b^z
  (w-z_1)^{m_1} \cdots (w-z_k)^{m_k}\ dw \right) + c\] be the unique
  monic polynomial with $\cpt(p) = z_1^{m_1} \cdots z_k^{m_k}$ that
  sends $b$ to $c$.  The map $\theta_\bm^{b\to c} \colon \C^k \to
  \C^k$, defined by sending each $\bz_\bm$ to \[\theta_\bm^{b\mapsto
    c}(\bz_\bm) = (\theta_{\bm,1}^{b\mapsto
    c}(\bz_\bm),\ldots,\theta_{\bm,k}^{b\to c}(\bz_\bm)) =
  (p_\bm^{b\to c}(z_1),\ldots,p_\bm^{b\to c}(z_k))\] takes the
  critical points of $p_\bm^{b\to c}$ to the critical values of
  $p_\bm^{b\to c}$.
\end{defn}

Theorem~B of \cite{DoMc-20} describes a factorization of the
determinant of the $k \times k$ Jacobian matrix $\bj_\bm^{b\mapsto
  c}$, defined by $(\bj_\bm^{b\to c})_{ij} =
\frac{\partial}{\partial z_i} \theta_{\bm,j}^{b\to c}(\bz_\bm) =
\frac{\partial}{\partial z_i} p_\bm^{b\to c}(z_j)$.

\begin{lem}[Invertibility]\label{lem:invertible}
    Let $b \to c$ be a base pair, let $\bm = (m_1,\ldots,m_k)$ be
    a $k$-tuple of positive integers with $m_1+\cdots + m_k = n$, and
    let $\bz_\bm = (z_1,\ldots,z_k) \in \mathbb{C}^k$. The determinant
    of the Jacobian $\bj_\bm^{b\to c}$ of the map
    $\theta_\bm^{b\to c}\colon \C^k \to \C^k$ factors as follows:
    \[
    \det \bj_\bm^{b\to c} = \frac{d^k}{\binom{n}{m_1,\ldots,m_k}}
    \left(\prod_{j\in [k]} (b-z_j)^{m_j}\right)
    \left(\prod_{\substack{i,j\in [k]\\i\neq j}}
    (z_i-z_j)^{m_j}\right).
    \]
    Thus $\bj_\bm^{b\to c}$ is invertible if and only if
    $z_1,\ldots,z_k$ are distinct and not equal to $b$.  In
    particular, $\theta_{\bm}^{b \to c}\colon \conf_k(\C_\bb) \to
    \C^k$ is a local homeomorphism.
\end{lem}

\begin{proof}
  There have been two modifications compared to Theorem B in
  \cite{DoMc-20}.  First, the polynomial $p_\bm^{b\to c}$ has been
  multiplied by $d$ so that $p_\bm^{b\to c}$ is monic.  This adds
  a factor of $d$ to every entry of the $k$-by-$k$ matrix
  $\bj_\bm^{b\to c}$ and a factor of $d^k$ to the factored
  determinant.  Next, the shift from base pair $p(0)=0$ to base pair
  $p(b)=c$ introduces the constant $b$ into the factorization.  The
  constant $c$ plays no role.
\end{proof}

\begin{figure}
  \begin{tikzcd}
    \bz_\bm \arrow[d,"\cong",equal]
    \arrow[r,mapsto,"\theta_\bm^{b\to c}"] &
    \bw_\bm \arrow[d,"\cong",equal] &
    \conf_k(\C_\bb) \arrow[r,"\theta_\bm^{b\to c}"]
    \arrow[d,"\cong",equal] &
    \C^k \arrow[d,"\cong",equal]\\
    \bz \arrow[d,equal]
    \arrow[r,mapsto,"\theta_{[\lambda]}^{b\to c}"] &
    \bw \arrow[d,equal] &
    (\C_\bb)_{[\lambda]} \arrow[r,"\theta_{[\lambda]}^{b\to c}"]&
    \C^{[\lambda]}\\
    \bz \arrow[d,mapsto]
    \arrow[r,mapsto,"\theta_{[\lambda]}^{b\to c}"] &
    \bw \arrow[dd,mapsto] &
    Z \arrow[d,twoheadrightarrow] \arrow[u,hookrightarrow]
    \arrow[r,rightarrow,"\theta_{[\lambda]}^{b\to c}"] &
    W \arrow[dd,twoheadrightarrow,"\mult"]
    \arrow[u,hookrightarrow]\\
    p \arrow[d,mapsto] & & 
    \poly_{\lambda\to\mu}^{m,b\to c}(\C_\cb)
    \arrow[d,twoheadrightarrow] \\
    {[p]} \arrow[r,mapsto,"\LL"] \arrow[d,equal] &
    \cvl(p) \arrow[d,equal]&
    \poly_{\lambda\to\mu}^{mt}(\C_\cb) \arrow[r,"\LL"] \arrow[d,hookrightarrow]&
    \mult_\mu(\C_\cb) \arrow[d,hookrightarrow]\\
    {[p]} \arrow[r,mapsto,"\LL"] &
    \cvl(p) &
    \poly_{\lambda\to\mu}^{mt} \arrow[r,"\LL"] & \mult_\mu(\C)
  \end{tikzcd}
  \caption{The local behavior of the stratified $\LL$ map is related
    to the local behavior of the stratified $\cpt$-to-$\cvl$-map
    $\theta_\bm^{b\to c}$.  In the square between the second and 
    third rows $Z$ is the restriction of $(\C_\bb)_{[\lambda]}$ 
    to the preimage of $W =(\C_\cb)_{[\mu]}$ under 
    $\theta_{[\lambda]}^{b\to c}$.  The spaces $\conf_k(\C_\bb) 
    \cong (\C_\bb)_{[\lambda]}$ and $\C^k
    \cong \C^{[\lambda]}$ in the top two rows are $2k$-dimensional
    manifolds and the other $7$ spaces are $2\ell$-dimensional
    manifolds, where $k$ and $\ell$ are the number of distinct
    critical points and critical values, respectively.  Every arrow is
    an inclusion or a local homeomorphism.\label{fig:covering-maps}}
\end{figure}

The map $\theta_\bm^{b\to c}$ can also be reformulated as a map
between subspaces of $\C^n$.

\begin{rem}[Subspaces of $\C^n$]\label{rem:subspaces-of-c^n}
  Let $[\lambda] \vdash [n]$ be a set partition with $k$ blocks
  indexed by the order they occur in the standard shorthand
  (Definition~\ref{def:set-part}).  For $\bz \in \C^{[\lambda]}
  \subset \C^n$, let $z_i$ be the common value of the coordinates
  indexed by the $i^{th}$ block, let $\bz_\bm = (z_1,\ldots,z_k)$, and
  let $\bm = (m_1,\ldots,m_k)$ where $m_i$ is the size of the $i^{th}$
  block.  For $\bz \in \C^{[\lambda]}$ with $[\lambda]= 124|3|57|6$,
  we have $\bz = (z_1,z_1,z_2,z_1,z_3,z_4,z_3)$, $\bz_\bm =
  (z_1,z_2,z_3,z_4)$ and $\bm = (m_1,m_2,m_3,m_4) = (3,1,2,1)$.  The
  map from $\bz \mapsto \bz_\bm$ is an isomorphism $\C^{[\lambda]}
  \cong \C^k$ (Proposition~\ref{prop:subsp-metrics}).  To indicate
  that the domain and range are replaced with the subspace
  $\C^{[\lambda]} \subset \C^n$ where $[\lambda]$ is any set partition
  whose $i^{th}$ block has size $m_i$, we replace the subscript $\bm$ on
  $\theta_\bm^{b\to c}$ with $[\lambda]$.  Under this
  homeomoprhism, the subspace $\conf_k(\C_\bb) \subset \C^k$, where
  $\theta_\bm^{b\to c}$ is a local homeomorphism, becomes
  $(\C_\bb)_{[\lambda]} \subset \C^{[\lambda]}$, where
  $\theta_{[\lambda]}^{b\to c}$ is a local homeomorphism.  See the
  top two rows of Figure~\ref{fig:covering-maps}.
\end{rem}

Every polynomial $p$ is part of a commuting diagram as in
Figure~\ref{fig:covering-maps}.

\begin{rem}[Commuting Maps]\label{rem:commuting-maps}
  Let $p$ be a monic degree-$d$ polynomial with critrical point shape
  $\lambda = \lambda(p) \vdash n$ and critical value shape $\mu
  =\mu(p) \geq \lambda$.  The equivalence class $[p]$ lies in
  $\poly_{\lambda\to\mu}^{mt}$ and the $\LL$ map send $[p]$ to the
  mulitset $M =\cvl(p) \in \mult_\mu(\C)$.  For any $c \not \in
  \cvl(p)$ and any $b \in p^{-1}(c)$, $[p]$ is in
  $\poly_{\lambda\to\mu}^{mt}(\C_\cb)$ and $p$ is in
  $\poly_{\lambda\to\mu}^{m,b\to c}(\C_\cb)$.  Next, let $\bz$ be an
  $n$-tuple listing the $n$ critical points of $p$ with the
  appropriate multiplicities and let $[\lambda] = \spart(\bz) \vdash
  [n]$ be its set partition.  Note that $\bz \in (\C_\bb)_{[\lambda]}$
  by construction and its image, $\bw = \theta_{[\lambda]}^{b\mapsto
    c}(\bz)$ is in $W=(\C_\cb)_{[\mu]}$ for some set partition $[\mu]
  \geq [\lambda]$ whose shape is $\mu = \mu(p)$ and $\mult(\bw) =
  \cvl(p)$. We define $Z$ as the restriction of $(\C_\bb)_{[\lambda]}$
  to the preimage of $W =(\C_\cb)_{[\mu]}$ under
  $\theta_{[\lambda]}^{b\to c}$, and note that $p$ can be
  reconstruction as $p = p_\bm^{b\to c}$.  Writing $\bz_\bm$ and
  $\bw_\bm$ for the images of $\bz$ and $\bw$ under the isomorphism
  from $\C^{[\lambda]}$ to $\C^k$ completes the construction of the
  commuting diagram containing the polynomial $p$ as shown in
  Figure~\ref{fig:covering-maps}.
\end{rem}

\begin{thm}[Stratified covering map]
  For all integer partitions $\lambda, \mu \vdash n$ with $\lambda
  \to \mu$, the restricted map $\LL\colon \poly_{\lambda\to\mu}^{mt}
  \to \mult_\mu(\C)$ is a covering map.
\end{thm}

\begin{proof}
  We first show that the restricted map $\LL\colon
  \poly_{\lambda\to\mu}^{mt} \to \mult_\mu(\C)$ is a local
  homeomorphism.  Let $p$ be a representative of $[p] \in
  \poly_{\lambda\to\mu}^{mt}$ and pick a base pair $b\to c$ and an
  ordering of the $n$ critical points of $p$ to construct the maps in
  Figure~\ref{fig:covering-maps} as described in
  Remark~\ref{rem:commuting-maps}.  Since $b \not \in \cpt(p)$ and $c
  \not \in \cvl(p)$, the restriction to $\LL\colon
  \poly_{\lambda\to\mu}^{mt}(\C_\cb) \to \mult_\mu(\C_\cb)$ in the
  penultimate row does not change the local neighborhood of $[p]$ in
  the domain or $\cvl(p)$ in the range.  Next note that the map
  $\mult\colon W = (\C_\cb)_{[\mu]} \to \mult_\mu(\C_\cb)$ is a
  covering map (Theorem~\ref{thm:mult-strata-cover}), and so is the
  map $\poly_{\lambda\to\mu}^{b \to c}(\C_\cb) \to
  \poly_{\lambda\to\mu}^{mt}$ (Remark~\ref{rem:base-pair-cover}).  Thus,
  if the two maps emerging from $Z$ are local homeomorphisms, so is
  the $\LL$ map in a neighborhood of $[p]$.  The space
  $W=(\C_\cb)_{[\mu]}$ is a $2\ell$-dimensional manifold inside the
  $2k$-dimensional manifold $\C^{[\lambda]}$, where $k$ and $\ell$ are
  the number of blocks in $[\lambda]$ and $[\mu]$, or the number of
  distinct critical points and distinct critical values, respectively.
  The space $Z$ is the preimage of the $2\ell$-dimensional manifold
  $W$ under the local homeomorphism $\theta_{[\lambda]}^{b\to c}$
  (Remark~\ref{rem:subspaces-of-c^n}), restricted to the open
  $2k$-dimensional submanifold $(\C_\cb)_{[\lambda]}$, which makes it
  a $2\ell$-dimensional manifold mapped to $W$ by a local
  homeomorphism.  Finally, the map from $Z$ to
  $\poly_{\lambda\to\mu}^{b\to c}(\C_\cb)$ merely erases the
  indexing from the critical points of $p$, so it is also a covering
  map.  This shows that the $\LL$ map in the bottom row is locally a
  homeomorphism.  To show that it is, in fact, a covering map, it
  suffices to note that it is a surjective local homeomorphism between
  Hausdorff spaces with constant finite size point preimages
  (Remark~\ref{rem:pt-preimages}).
\end{proof}

\bibliographystyle{amsalpha}
\bibliography{gcp}

\end{document}